\documentclass{article}
\usepackage[a4paper,top=2cm,bottom=2cm,left=2cm,right=2cm]{geometry}

\usepackage{amsmath}
\usepackage{amssymb}
\usepackage{amsthm}
\usepackage{mathrsfs}
\usepackage{tikz-cd}
\usepackage{stmaryrd}
\usepackage{authblk}

\usepackage{mathtools}

\usepackage[toc,page]{appendix}

\def\std{{\raisebox{\depth}{\scalebox{1.45}[-1]{\mbox{$\mathbb \Delta$}}}}_{\!}}
\def\dst{{}_{\,}\reflectbox{\mbox{$\std$}}_{\!}}
\usepackage[bbgreekl]{mathbbol}
\usepackage{bbm}

\numberwithin{equation}{section}

\newtheorem{proposition}{Proposition}[section]
\newtheorem{definition}[proposition]{Definition}
\newtheorem{lemma}[proposition]{Lemma}
\newtheorem{theorem}[proposition]{Theorem}
\newtheorem{remark}[proposition]{Remark}
\newtheorem{example}[proposition]{Example}
\newtheorem{corollary}[proposition]{Corollary}

\title{\bf Metric compatibility and Levi-Civita Connections on Quantum Groups}
\author{{\Large Paolo Aschieri}\footnote{paolo.aschieri@uniupo.it}~~${}^{1,2}$ }
\author{{\Large Thomas Weber}\footnote{thomas.weber@unito.it}~~${}^{2,3}$}
\affil{\large{
\centerline{$^1${\sl Dipartimento di Scienze e Innovazione 
Tecnologica, Universit\`a del Piemonte Orientale}}

\centerline{{ \sl
Viale T. Michel - 15121, Alessandria, Italy}}

\vskip 0.1 cm

\centerline{$^2${\sl Istituto Nazionale di Fisica Nucleare, Sezione di Torino,
  via P. Giuria 1, 10125 Torino}}

\vskip 0.5 cm

\centerline{\sl
$^3${ Dipartimento di Matematica ``Giuseppe Peano", Universit\`a 
degli Studi di Torino}}

\centerline{\sl  Via Carlo Alberto 10, 10123 Torino, Italy}
}
}
\date{\today}

\begin{document}

\maketitle

\begin{abstract}

  Arbitrary connections on a generic Hopf algebra $H$ are studied and
shown to extend to connections on tensor fields. On this ground a general
definition of metric compatible connection is proposed. This leads to a sufficient
criterion for the existence and uniqueness of the Levi-Civita
connection, that of invertibility of an
$H$-valued matrix.
Provided invertibility for one metric, existence and uniqueness of the
Levi-Civita connection for all metrics conformal to the initial one is proven.
This class consists of metrics which are neither central (bimodule maps) nor
equivariant, in general.
For central and bicoinvariant metrics the invertibility condition is further
simplified to a metric independent one.
 Examples include metrics on $SL_q(2)$.


\end{abstract}

\tableofcontents

\section{Introduction}
Quantum groups are leading examples for the study of noncommutative
geometry. Their differential geometry has been developed in a seminal
paper in \cite{Woronowicz1989}, followed by explicit constructions for
unitary, orthogonal and symplectic quantum groups in \cite{Jurco},
\cite{Watamura}, and their classification in \cite{Sch}.  
Similarly, in the study of noncommutative Riemannian geometry quantum groups are
a testing ground for different approaches 
to the fundamental theorem of Riemannian geometry, stating existence and uniqueness
of the Levi-Civita connection for any metric tensor. This is presently under active
investigation in noncommutative geometry.  Indeed, while noncommutative differential
geometry on quantum groups and homogeneous spaces is relatively
well-understood, there is no obvious noncommutative Riemannian
theory. This is already seen when introducing metrics on quantum algebras.
In the literature there are indeed two approaches, the first one is to
require metrics
to be compatible with the noncommutative structure,
e.g. to be central  ($a\mathbf{g} = \mathbf{g}a$ for all elements $a$ in
the noncommutative algebra $A$), otherwise stated, to be a bimodule map.
This is quite natural, as an analogy, if we have a compact manifold that
has also the structure of a simple Lie group
the first metric we would consider is the Killing one. 
The other approach is to consider arbitrary metrics, not necessarily
compatible with the noncommutative structure. This is useful when
$\mathbf{g}$ is a dynamical field, like in
gravity. Indeed, one motivation for noncommutative Riemannian geometry
is the study of gravity theories on quantum spacetimes.
We face a similar situation for connections. The metric compatibility condition  
$\nabla\mathbf{g}=0$ for a metric $\mathbf{g}\in \mathfrak{X}\otimes_A\mathfrak{X}$ requires
extending a connection on the $A$-bimodule of vector fields $\mathfrak{X}$ to the tensor product $\mathfrak{X}\otimes_A\mathfrak{X}$. This is typically done by
considering bimodule connections \cite{D-VMichor, Mourad}; again, a connection,
treated as a dynamical field, is not  a bimodule
connection in general.

Noncommutative Riemannian geometries with metrics compatible with the
quantum algebra structure have been considered in \cite{Madore}, 
\cite{MajidBeggs2011}, \cite{BhowmickLandi},
\cite{BhowmickLandi2}, see also the recent monograph
\cite{MajidBeggsBook}.
Levi-Civita connections for $H$-equivariant metrics on algebras $A$
with a triangular Hopf algebra symmetry have been studied in
\cite{Weber}, while in \cite{AIL} a compatibility
condition in the form of a differential equation for metrics on the
quantum group $S^3_q$ has been considered.

On the other hand, a selected class of noncommutative algebras allows
for Levi-Civita connections of arbitrary metrics.
See \cite{NCG1} for the noncommutative hyperplane  $\mathbb{R}^n_\theta$,  \cite{Rosenberg13}
for the associated noncommutative torus, \cite{Arnlind} for the
noncommutative $\theta$-sphere in three dimensions,
\cite{AschieriCartanStructure} for the general case of a
noncommutative algebra with triangular Hopf algebra symmetry, which encompasses the
case of  triangular quantum groups.
\\

In this paper we study the Riemannian geometry of quantum groups 
(Hopf algebras): no (quasi)triangularity  is assumed. 
We simply rely on the canonical braiding $\sigma^\mathcal{W}$ of bicovariant $H$-bimodules
with  $H$ the Hopf algebra.

A first result, relaxing the triangularity condition of
\cite{AschieriCartanStructure}, is a general definition of metric
compatible connection. The condition $\nabla\mathbf{g}=0$ for a
metric $\mathbf{g}\in \mathfrak{X}\otimes_H\mathfrak{X}$ requires
extending a connection on the bicovariant bimodule of vector fields
$\mathfrak{X}$ to the tensor product
$\mathfrak{X}\otimes_H\mathfrak{X}$. We present this sum of
connections construction for arbitrary  connections. This is done by extending the tensor product of $H$-linear maps to $\Bbbk$-linear
maps (more precisely internal homomorphisms of the monoidal category of
bicovariant bimodules with tensor product over the ground field $\Bbbk$,
not $H$). It also requires lifting the canonical braiding $\sigma^\mathcal{W}$ on the tensor product over $H$ of bicovariant
bimodules to the tensor product over $\Bbbk$. 
Furthermore, while in the commutative case (as well as in the triangular case) the
metric compatibility condition reads $\nabla \mathbf{g}=0$, where
$\nabla$ is the connection on $\mathfrak{X}\otimes_H\mathfrak{X}$ or equivalently
its restriction to the subbimodule of braided symmetric tensor fields
$\mathfrak{X}\vee\mathfrak{X}$, in the present more general setting
the metric compatibility condition requires the construction of the connection on
$\mathfrak{X}\vee\mathfrak{X}$. This latter is obtained from that on
$\mathfrak{X}\otimes_H\mathfrak{X}$ by restricting the domain and then
projecting to the direct summand
$\mathfrak{X}\vee\mathfrak{X}\subseteq
\mathfrak{X}\otimes_H\mathfrak{X}$.

Based on this metric compatibility condition we give a sufficient
criterion for the existence and uniqueness of the Levi-Civita
connection via the invertibility of a right $H$-linear map 
$\Phi_\mathbf{g}:(\Omega^1\vee\Omega^1)\otimes_H\mathfrak{X}
    \rightarrow\Omega^1\otimes_H(\mathfrak{X}\vee\mathfrak{X})$ of free
$H$-modules (an $H$-valued matrix).
This map depends on $\mathbf{g}$ and was inspired by the one
introduced in \cite{BhowmickCov}; here  however $\mathbf{g}$ is an arbitrary metric, neither central nor  bicoinvariant.
The differential equations defining Levi-Civita connections are then
solved by the purely algebraic condition of invertibility of $\Phi_\mathbf{g}$.

Furthermore, we show that if the map $\Phi_\mathbf{g}$ is invertible for a metric $\mathbf{g}$, it is also invertible for all metrics conformally equivalent to $\mathbf{g}$ (the scale factor being an arbitrary invertible element in the algebra). 
For $\sigma^\mathcal{W}$-central metrics (like e.g. bicovariant central metrics, see Definition~\ref{DefSigmaCentral}), the invertibility of $\Phi_\mathbf{g}$ is reduced to the study of the invertibility of a map $(\Omega^1\vee\Omega^1)\otimes_H\Omega^1\to
\Omega^1\otimes_H(\Omega^1\vee\Omega^1)$ depending only on the
braiding $\sigma^\mathcal{W}$. This allows to prove existence and
uniqueness of the Levi-Civita connection for any metric conformally
equivalent to a $\sigma^\mathcal{W}$-central one, provided the above
isomorphism holds. We thus obtain existence and uniqueness of the
Levi-Civita connection on a quantum group for this class of metrics
that are neither central nor equivariant. This includes the example of
metrics on $SL_q(2)$. 
As a byproduct, these results {provide a way out to}  a gap in \cite[Thm. 7.9]{BhowmickCov}
 where existence and uniqueness of Levi-Civita connections is claimed
 for bicoinvariant metrics on quantum groups, but implicitly assuming centrality of the metric (cf. Remark \ref{remBM}).
 \\

The paper is organized as follows.
In Section~\ref{Sec2} we give the categorical background necessary to
construct the sum of connections. For this we recall the known theory of
comodules (respectively covariant modules) and their closed monoidal categories
in Section~\ref{Sec2.1} (respectively Section~\ref{Sec2.2}).
The first original contribution of the paper is in Section~\ref{Sec2.3}, where
we introduce a tensor product of internal homomorphisms of the monoidal category 
$({}_H^H\mathcal{M}_H^H,\otimes)$ of bicovariant bimodules with respect to the 
tensor product $\otimes$ over the ground field $\Bbbk$. The latter is obtained from a
lifting of the braiding $\sigma^\mathcal{W}$.
Section~\ref{Sec3}, on the differential geometry on quantum groups, recalls  the
construction of the braided exterior algebra (Section~\ref{Sec3.1})
and of the
bicovariant differential calculus (Section~\ref{Sec3.2}).
In Section~\ref{Sec4} we study connections, with basic definitions
in Section~\ref{Sec4.1}, curvature and torsion in Section~\ref{Sec4.2} and duality of
connections in Section~\ref{Sec4.3}. One of the main contributions of this paper,
the sum of connections on bicovariant
bimodules (or braided derivation formula for connections), is discussed in Section~\ref{Sec4.4}, with an explicit expression of its
curvature in Section~\ref{Sec4.5}. Finally, in Section~\ref{Sec5} we introduce metrics and
metric-compatibility via the sum of connections and their
projections to symmetric tensors. Adding the torsion condition we
obtain the notion of Levi-Civita
connection for arbitrary metrics on quantum groups. We then prove the
mentioned sufficient conditions for existence and uniqueness of Levi-Civita connections.
In particular on $SL_q(2)$ there is a unique Levi-Civita connection
for each metric which is conformally equivalent to  central and
bicoinvariant ones. We further discuss the class of examples arising
from braided derivations on cotriangular Hopf algebras.

\section{Covariant Modules and Rational Morphisms}\label{Sec2}

We work over a base field $\Bbbk$. Linear (vector) spaces,
algebras, Hopf algebras etc. are understood over $\Bbbk$, similarly linear maps.
The tensor product of vector spaces is denoted $\otimes$.
Let $H$ be a Hopf algebra with coalgebra structure $(\Delta,\epsilon)$
and antipode $S$. We shall always assume $S$ to be 
invertible, with inverse denoted $\overline{S}$.
We use Sweedler's notation $\Delta(h)=h_1\otimes h_2$ for the coproduct of an
element $h\in H$ and similarly write
$h_1\otimes h_2\otimes h_3:=
(\Delta\otimes\mathrm{id}_H)(\Delta(h))=(\mathrm{id}_H\otimes\Delta)(\Delta(h))$ etc.
for higher coproducts.

Linear maps between $H$-modules form an
$H$-module under the $H$-adjoint action (if $\phi: M\to N$, $h\in H$, the
adjoint action is $h\triangleright
\phi$, with  $(h\triangleright
\phi)(m)=h_1\triangleright(\phi (S(h_2)\triangleright m))$ for all
$m\in M$). On the contrary, linear maps between $H$-comodules do not
in general form an $H$-comodule and one has to restrict to the subspace
of rational morphisms. These are linear maps which admit adjoint
coactions. We recall this notion in Section~\ref{Sec2.1},
following \cite{Ulbrich, SVO99, CaenGu07}.
In Section~\ref{Sec2.2} we study compatible $H$-modules and comodules,
called covariant modules or Hopf modules.
 We recall that they are free $H$-modules.
As proven in \cite{Schauenburg94},
bicovariant bimodules form a braided monoidal category (with tensor
product $\otimes_H$) and we observe that
this category is closed monoidal when considering the internal Hom-functor of right $H$-linear
maps between bicovariant bimodules.
The tensor product of  right $H$-linear
maps is extended to linear maps
in Section~\ref{Sec2.3} via lifting the braiding  $M\otimes_HN\to
N\otimes_H M$ to a map  $M\otimes N\to N\otimes M$, for $M,N$
bicovariant bimodules. 
Closed braided monoidal categories are treated in several textbooks,
see for example \cite{Ka95,Ma95}.

\subsection{Hopf Algebra Comodules and Rational Morphisms}\label{Sec2.1}
Recall that a right $H$-comodule is a linear space $M$
endowed with a linear map $\delta_M\colon M\rightarrow M\otimes H$ satisfying
\begin{equation*}
    (\delta_M\otimes\mathrm{id}_M)\circ\delta_M
    =(\mathrm{id}_M\otimes\Delta)\circ\delta_M
    \text{ and }
    (\mathrm{id}_M\otimes\epsilon)\circ\delta_M
    =\mathrm{id}_M~.
\end{equation*}
Similarly we define a left $H$-comodule structure on $M$ via the
linear map $\lambda_M\colon M\rightarrow H\otimes M$ satisfying
\begin{equation*}
    (\mathrm{id}_H\otimes\lambda_M)\circ\lambda_M
    =(\Delta\otimes\mathrm{id}_M)\circ\lambda_M
    \text{ and }
    (\epsilon\otimes\mathrm{id}_M)\circ\lambda_M
    =\mathrm{id}_M~.
\end{equation*}
The maps $\delta_M$ and $\lambda_M$ are called right and left
$H$-coaction, respectively.
In analogy with Sweedler's notation we write $\delta_M(m)=m_0\otimes m_1$ and
$\lambda_M(m)=m_{-1}\otimes m_0$ for $m\in M$.
We also define
\begin{equation*}
    m_0\otimes m_1\otimes m_2
    :=(\delta_M\otimes\mathrm{id}_H)(\delta_M(m))
    =(\mathrm{id}_M\otimes\Delta)(\delta_M(m))
\end{equation*}
and similarly for $\lambda_M$.
For a right comodule $M$ the vector space of right
coinvariant (coaction-invariant) elements is denoted by
$M^{\mathrm{co}H}:=\{m\in M~|~\delta_M(m)=m\otimes 1\}$
and the vector space of left coinvariant elements of a left $H$-comodule $M$
by ${}^{\mathrm{co}H}M:=\{m\in M~|~\lambda_M(m)=1\otimes m\}$.
Morphisms of right $H$-comodules are right $H$-colinear maps $f\colon M\rightarrow N$, i.e., 
$$
\delta_N\circ f
=(f\otimes\mathrm{id}_H)\circ\delta_M~.
$$
Analogously one defines morphisms of left $H$-comodules as left $H$-colinear maps.
The category of right $H$-comodules is denoted by
$\mathcal{M}^H$, that of left $H$-comodules ${}^H\mathcal{M}$.
$H$-bicomodules are right and left 
$H$-comodules with commuting coactions, their
morphisms are $H$-bicolinear maps. We denote the category of $H$-bicomodules by
${}^H\mathcal{M}^H$. Since the right and left coactions on an $H$-bicomodule
$M$ commute, the notation
\begin{equation*}
    m_{-1}\otimes m_0\otimes m_1
    :=(\mathrm{id}_H\otimes\delta_M)(\lambda_M(m))
    =(\lambda_M\otimes\mathrm{id}_H)(\delta_M(m))
\end{equation*}
is well-defined for any $m\in M$. For right/left/bi $H$-comodules $M,N$ we denote
the vector space of their morphisms by $\mathrm{Hom}^H(M,N)$,
${}^H\mathrm{Hom}(M,N)$ and ${}^H\mathrm{Hom}^H(M,N)$,
respectively.

It is well-known that the categories of $H$-comodules are monoidal 
(see e.g. \cite{Ka95,Ma95}). For  right $H$-comodules $M,N$ 
the tensor product vector space $M\otimes N$ is a right $H$-comodule
via the diagonal right $H$-coaction
\begin{equation*}
    \delta_{M\otimes N}\colon M\otimes N\to  M\otimes N\otimes H~,~~
    m\otimes n\mapsto m_0\otimes n_0\otimes m_1n_1~.
\end{equation*}
The monoidal unit is the field $\Bbbk$ with the trivial right $H$-coaction
$\Bbbk\to \Bbbk\otimes H,~\chi\mapsto\chi\otimes 1_H$. 
Similarly  for left $H$-comodules $M,N$, $M\otimes N$ is a left
$H$-comodule with diagonal left $H$-coaction
\begin{equation*}
    \lambda_{M\otimes N}\colon
    M\otimes N\to H\otimes M\otimes N~,~~ m\otimes n\mapsto m_{-1}n_{-1}\otimes m_0\otimes n_0~.
\end{equation*}
Since the right and left diagonal coactions of $H$-bicomodules commute we obtain
the monoidal category $({}^H\mathcal{M}^H,\otimes,\Bbbk)$ of $H$-bicomodules.

Let $\mathrm{Hom}(M,N)$ be the vector space of linear maps
from the vector space $M$ to $N$.
When $M,N$ are right $H$-comodules we consider the linear map
$\delta^\mathrm{Ad}\colon\mathrm{Hom}(M,N)\rightarrow\mathrm{Hom}(M,N\otimes H)$
defined for any $\phi\in\mathrm{Hom}(M,N)$ by
\begin{equation}\label{coadr}
    \delta^\mathrm{Ad}(\phi)(m):=\phi(m_0)_0\otimes\phi(m_0)_1S(m_1)
\end{equation}
for all $m\in M$. Here, linearity of $\delta^\mathrm{Ad}(\phi)$
follows from 
linearity of $\delta_M$, $\delta_N$, $\phi$ and $S$. Thus $\delta^\mathrm{Ad}$ is well-defined. Similarly, for two left $H$-comodules
$M,N$ we define the linear map
$\lambda^\mathrm{Ad}\colon\mathrm{Hom}(M,N)\rightarrow\mathrm{Hom}(M,H\otimes N)$
for all $\phi\in\mathrm{Hom}(M,N)$ by
\begin{equation}\label{coadl}
    \lambda^\mathrm{Ad}(\phi)(m):=\phi(m_0)_{-1}\overline{S}(m_{-1})\otimes\phi(m_0)_0~,
\end{equation}
where $m\in M$.

Consider two vector spaces $M,N$. Since $H$ is a free 
module over $\Bbbk$ (it would suffice to consider $H$ projective over a commutative
unital ring $\Bbbk$) we can view $\mathrm{Hom}(M,N)\otimes H$
as a vector subspace of $\mathrm{Hom}(M,N\otimes H)$ via the inclusion
\begin{equation}\label{incinNH}
    \mathrm{Hom}(M,N)\otimes H
    \hookrightarrow\mathrm{Hom}(M,N\otimes H)
\end{equation}
defined by $(\phi\otimes h)(m):=\phi(m)\otimes h$ for all 
$\phi\in\mathrm{Hom}(M,N)$, $h\in H$ and $m\in M$
(cf. \cite{Bourbaki} Chapter~2, §4.2 Proposition~2). 
Similarly, $H\otimes\mathrm{Hom}(M,N)$ can be understood as a vector subspace
of $\mathrm{Hom}(M,H\otimes N)$ via the inclusion
$
    H\otimes\mathrm{Hom}(M,N)
    \hookrightarrow\mathrm{Hom}(M,H\otimes N)
$
defined by $(h\otimes\phi)(m):=h\otimes\phi(m)$
for all $\phi\in\mathrm{Hom}(M,N)$, $h\in H$ and $m\in M$.
We shall frequently identify $\mathrm{Hom}(M,N)\otimes H$ and $
H\otimes\mathrm{Hom}(M,N)$ with their images in 
    $\mathrm{Hom}(M,N\otimes H)$ and in $\mathrm{Hom}(M,H\otimes N)$.
\begin{definition}[Rational Morphisms \cite{Ulbrich}]
For right $H$-comodules $M,N$ we define the vector space of
right rational morphisms as the preimage
\begin{equation*}
    \mathrm{HOM}^{\mathrm{Ad}}(M,N)
    :=(\delta^\mathrm{Ad})^{-1}(\mathrm{Hom}(M,N)\otimes H~.
\end{equation*}
For left $H$-comodules $M,N$ we define the vector space of
left rational morphisms as the preimage
\begin{equation*}  
    {}^{\mathrm{Ad}}\mathrm{HOM}
    :=(\lambda^\mathrm{Ad})^{-1}(H\otimes\mathrm{Hom}(M,N))~.
\end{equation*}
On $H$-bicomodules $M,N$ we define the vector space of rational morphisms as the
intersection
\begin{equation*}
    \mathrm{HOM}(M,N):=\mathrm{HOM}^{\mathrm{Ad}}(M,N)\cap
    {}^{\mathrm{Ad}}\mathrm{HOM}(M,N)~.
\end{equation*}
\end{definition}
\noindent
For $\phi\in\mathrm{HOM}(M,N)$ we write
$\delta^\mathrm{Ad}(\phi)=\phi_0\otimes\phi_1\in\mathrm{Hom}(M,N)\otimes H$
and $\lambda^\mathrm{Ad}(\phi)=\phi_{-1}\otimes\phi_0\in H\otimes\mathrm{Hom}
(M,N)$.
In the following lemma 
we recall that the vector spaces of rational morphisms are $H$-comodules,
that $H$-colinear maps are equivalently $H$-coinvariant (coaction
invariant) rational morphisms and that evaluation and composition are in particular
$H$-colinear maps.
\begin{lemma}\label{lemma01}
Let $M,N$ be right $H$-comodules. 
\begin{enumerate}
\item[i.)] $\mathrm{HOM}^\mathrm{Ad}(M,N)$ is a right $H$-comodule.
\item[ii.)] The evaluation map $\mathrm{ev}\colon\mathrm{HOM}^\mathrm{Ad}(M,N)
  \otimes M\rightarrow N$  is right $H$-colinear, i.e. 
\begin{equation*}
\begin{tikzcd}
    \mathrm{HOM}^\mathrm{Ad}(M,N)\otimes M
    \arrow{rr}{\mathrm{ev}}
    \arrow{d}[swap]{\delta_{_{\mathrm{HOM}^\mathrm{Ad}(M,N)\otimes M}}}
    & & N
    \arrow{d}{\delta_N}\\
    \mathrm{HOM}^\mathrm{Ad}(M,N)\otimes M\otimes H
    \arrow{rr}{\mathrm{ev}\otimes\mathrm{id}_H}
    & & N\otimes H
\end{tikzcd}
\end{equation*}
commutes, that is,
$\phi(m)_0\otimes\phi(m)_1=\phi_0(m_0)\otimes\phi_1m_1$ for all
$\phi\in\mathrm{HOM}^\mathrm{Ad}(M,N)$, $m\in M$.

\item[iii.)]
  $\mathrm{Hom}^H(M,N)=\mathrm{HOM}^\mathrm{Ad}(M,N)^{\mathrm{co}H}$,
  i.e. for $\phi\in\mathrm{Hom}(M,N)$ the $H$-colinearity  property, for all $m\in M$, $\phi(m)_0\otimes\phi(m)_1
=\phi(m_0)\otimes m_1$  holds if and only if $\delta^\mathrm{Ad}(\phi)=\phi\otimes 1$.

\item[iv.)]  The composition of right rational morphisms gives 
right rational morphisms,
   $\circ\colon
  \mathrm{HOM}^\mathrm{Ad}(N,O)\otimes
  \mathrm{HOM}^\mathrm{Ad}(M,N)\to \mathrm{HOM}^\mathrm{Ad}(M,O)$,
and it is a right $H$-colinear map:
\begin{equation*}
    (\psi\circ\phi)_0\otimes(\psi\circ\phi)_1
    =\psi_0\circ\phi_0\otimes\psi_1\phi_1
\end{equation*}
for all $\phi\in\mathrm{HOM}^\mathrm{Ad}(M,N)$ and
$\psi\in\mathrm{HOM}^\mathrm{Ad}(N,O)$.
\end{enumerate}
Let $M,N$ be left $H$-comodules. 
\begin{enumerate}
\item[i'.)] ${}^\mathrm{Ad}\mathrm{HOM}(M,N)$ is a left $H$-comodule.

\item[ii'.)] The evaluation map $\mathrm{ev}\colon{}^\mathrm{Ad}\mathrm{HOM}
(M,N)\otimes M\rightarrow N$ if left $H$-colinear, i.e.
\begin{equation*}
\begin{tikzcd}
    {}^\mathrm{Ad}\mathrm{HOM}(M,N)\otimes M
    \arrow{rr}{\mathrm{ev}}
    \arrow{d}[swap]{\lambda_{_{{}^\mathrm{Ad}\mathrm{HOM}(M,N)\otimes M}}}
    & & N
    \arrow{d}{\lambda_N}\\
    H\otimes{}^\mathrm{Ad}\mathrm{HOM}(M,N)\otimes M
    \arrow{rr}{\mathrm{id}_H\otimes\mathrm{ev}}
    & & H\otimes N
\end{tikzcd}
\end{equation*}
commutes, that is,
$\phi(m)_{-1}\otimes\phi(m)_0=\phi_{-1}m_{-1}\otimes\phi_0(m_0)$ for
all $\phi\in  {}^\mathrm{Ad}\mathrm{HOM}(N,O)$ and  $m\in M$.

\item[iii'.)] ${}^H\mathrm{Hom}(M,N)={}^{\mathrm{co}H}({}^\mathrm{Ad}\mathrm{HOM}(M,N))$, i.e. for $\phi\in\mathrm{Hom}(M,N)$ we have $\phi(m)_{-1}\otimes\phi(m)_0
=m_{-1}\otimes\phi(m_0)$ for all $m\in M$ if and only if $\lambda^\mathrm{Ad}(\phi)=1\otimes\phi$.
\item[iv'.)] The composition of left rational morphisms gives a left rational 
morphism, $\circ\colon
  {}^\mathrm{Ad}\mathrm{HOM}(N,O)\otimes
  {}^\mathrm{Ad}\mathrm{HOM}(M,N)\to {}^\mathrm{Ad}\mathrm{HOM}(M,O)$,
and is a left $H$-colinear map:
\begin{equation*}
    (\psi\circ\phi)_{-1}\otimes(\psi\circ\phi)_0
    =\psi_{-1}\phi_{-1}\otimes\psi_0\circ\phi_0
\end{equation*}
for all $\phi\in{}^\mathrm{Ad}\mathrm{HOM}(M,N)$ and
$\psi\in{}^\mathrm{Ad}\mathrm{HOM}(N,O)$.
\end{enumerate}
Let $M,N$ be $H$-bicomodules.
\begin{enumerate}
\item[i''.)] $\mathrm{HOM}(M,N)$ is an $H$-bicomodule.
  
\item[ii''.)] The evaluation map $\mathrm{ev}\colon\mathrm{HOM}(M,N)
  \otimes M\rightarrow N$ is $H$-bicolinear.

\item[iii''.)]
  ${}^H\mathrm{Hom}^H(M,N)={}^{\mathrm{co}H}({}^\mathrm{Ad}\mathrm{HOM}(M,N))^{\mathrm{co}H}$,
  i.e., $H$-bicolinear maps are equivalent to $H$-bicoinvariant
  rational morphisms.

\item[iv''.)] The composition map
$\circ\colon
  \mathrm{HOM}(N,O)\otimes
  \mathrm{HOM}(M,N)\to \mathrm{HOM}(M,O)$
  is $H$-bicolinear. 
\end{enumerate}

\end{lemma}

\begin{proof} For i.) see \cite{Ulbrich} \S 2,  ii.) is
  straightforward and implies iii.) (just multiply the equality
  $\phi_0(m_0)\otimes\phi_1m_1=\phi(m_0)\otimes m_1$ by $S(m_2)$
  from the right). iv.)  is in
  \cite{CaenGu07}  Cor.~1.6;  i'.)-iv'.) are
  similarly proven. 
We prove  i''.), then $ii''.)$ - $iv''.)$ are straightforward.
In order to prove that $\delta^\mathrm{Ad}(\mathrm{HOM}(M,N))\subseteq\mathrm{HOM}(M,N)\otimes
H$
and $\lambda^\mathrm{Ad}(\mathrm{HOM}(M,N)) \subseteq H\otimes
\mathrm{HOM}(M,N)$,
we consider the diagram
\begin{equation}\label{eq01}
\begin{tikzcd}
\mathrm{HOM}(M,N) \arrow{rr}{\delta^\mathrm{Ad}}
\arrow{d}[swap]{\lambda^\mathrm{Ad}}
& & \mathrm{HOM}^\mathrm{Ad}(M,N)\otimes H
\arrow{rr}{\lambda^\mathrm{Ad}\otimes\mathrm{id}_H}
& &  \mathrm{Hom}(M,H\otimes N)\otimes H~~
\arrow[d, hook]\\
H\otimes{}^\mathrm{Ad}\mathrm{HOM}(M,N)
\arrow{rr}{\mathrm{id}_H\otimes\delta^\mathrm{Ad}}
& &
H\otimes\mathrm{Hom}(M,N\otimes H)
\arrow[rr, hook]
& &\mathrm{Hom}(M,H\otimes N\otimes H)~,
\end{tikzcd}
\end{equation}
where in $\mathrm{id}_H\otimes\delta^\mathrm{Ad}$ and
$\lambda^\mathrm{Ad}\otimes\mathrm{id}_H$ we used  definitions
\eqref{coadr} and \eqref{coadl}, while the hooked arrows denote the
canonical vector space inclusions, that will be implicitly understood. This diagram is  commutative:
\begin{align*}
    ((\mathrm{id}_H\otimes\delta^\mathrm{Ad})\circ\lambda^\mathrm{Ad})(\phi)(m)
     =&\phi_{-1}\otimes\phi_0(m_0)_0\otimes\phi_0(m_0)_1S(m_1)\\
  =&\phi(m_0)_{-1}\overline{S}(m_{-1})\otimes\phi(m_0)_0\otimes\phi(m_0)_1S(m_1)\\
       =&\phi_0(m_0)_{-1}\overline{S}(m_{-1})
    \otimes\phi_0(m_0)_0
    \otimes\phi_1\\
    =&((\lambda^\mathrm{Ad}\otimes\mathrm{id}_H)\circ\delta^\mathrm{Ad})(\phi)(m)
\end{align*}
for all $\phi\in\mathrm{HOM}(M,N)$ and $m\in M$.
Commutativity of (\ref{eq01}) implies that both
$((\lambda^\mathrm{Ad}\otimes\mathrm{id}_H)\circ\delta^\mathrm{Ad})(\mathrm{HOM}(M,N))$
and
$((\mathrm{id}_H\otimes\delta^\mathrm{Ad})\circ\lambda^\mathrm{Ad})(\mathrm{HOM}(M,N))$
are subspaces of 
$$
F:=(H\otimes\mathrm{Hom}(M,N\otimes H))\cap(\mathrm{Hom}(M,H\otimes N)\otimes H)~.
$$
Clearly $H\otimes\mathrm{Hom}(M,N)\otimes H\subseteq F$. 
We prove it is equal by showing the other inclusion: $F\subseteq H\otimes\mathrm{Hom}(M,N)\otimes H$. 
An arbitrary element of $F$ can be written as $\sum\nolimits_{i\in I}f^i\otimes h_i=\sum\nolimits_{j\in J}k^j\otimes g_j\in F$
with $h_i,k^j\in H$, $f^i\in\mathrm{Hom}(M,H\otimes N)$, $g_j\in\mathrm{Hom}(M,N\otimes H)$
and finite index sets $I,J$. Without loss of generality we can assume
that the elements $h_i, {i\in I}$
are linearly independent. Complete $\{h_i\}_{i\in I}$ to an  basis
$\{h_r\}_{r\in\mathcal{I}}$ of $H$ with (possibly infinite) index set 
$\mathcal{I}\supseteq I$. Consider linear functionals
$\{\alpha^i\}_{i\in I}\subseteq H^*$ such that
$\alpha^i(h_r)=\delta^i_r$  for all $i\in I,r\in\mathcal{I}$.
Let $\ell\in I$,  applying $\mathrm{id}_{H\otimes N}\otimes\alpha^\ell$
to $\sum\nolimits_{i\in I}f^i\otimes h_i=\sum\nolimits_{j\in J}k^j\otimes g_j$ gives
$f^\ell=\sum\nolimits_{j\in J}k^j\otimes(\mathrm{id}_N\otimes\alpha^\ell)\circ g_j
=\sum\nolimits_{j\in J}k^j\otimes\xi^\ell_j$, where we defined $\xi^\ell_j
:=(\mathrm{id}_N\otimes\alpha^\ell)\circ g_j\in\mathrm{Hom}(M,N)$ for all $\ell,j\in I$
(finitely many). Then $\sum\nolimits_{i\in I}f^i\otimes h_i
=\sum\nolimits_{i\in I}\sum\nolimits_{j\in J}k^j\otimes\xi^i_j\otimes h_i\in H\otimes
\mathrm{Hom}(M,N)\otimes H$, proving $F\subseteq
H\otimes\mathrm{Hom}(M,N)\otimes H$.

Thus, 
$((\lambda^\mathrm{Ad}\otimes\mathrm{id}_H)\circ\delta^\mathrm{Ad})(\mathrm{HOM}(M,N))$
and
$((\mathrm{id}_H\otimes\delta^\mathrm{Ad})\circ\lambda^\mathrm{Ad})(\mathrm{HOM}(M,N))$
are subspaces of  $H\otimes\mathrm{Hom}(M,N)\otimes H$, which implies that 
$\delta^\mathrm{Ad}(\mathrm{HOM}(M,N))\subseteq\mathrm{HOM}(M,N)\otimes H$
and $\lambda^\mathrm{Ad}(
\mathrm{HOM}(M,N))\subseteq H\otimes\mathrm{HOM}(M,N)$. In summary, $\mathrm{HOM}(M,N)$
is a right and left $H$-comodule with commuting coactions, i.e. an $H$-bicomodule.
\end{proof}
We call $\delta^\mathrm{Ad}\colon\mathrm{HOM}^\mathrm{Ad}(M,N)
\rightarrow\mathrm{HOM}^\mathrm{Ad}(M,N)\otimes H$
and $\lambda^\mathrm{Ad}\colon{}^\mathrm{Ad}\mathrm{HOM}(M,N)
\rightarrow H\otimes{}^\mathrm{Ad}\mathrm{HOM}(M,N)$
the \textit{right and left adjoint $H$-coaction} on rational morphisms, respectively.
In \cite{SVO99} it is shown that $\mathrm{HOM}^\mathrm{Ad}(M,N)$ is
the largest $H$-comodule in $\mathrm{Hom}(M,N)$. If $H$ is
finite-dimensional all morphisms are rational. 
The previous lemma allows to define the internal $\mathrm{Hom}$-functor
\begin{equation*}
    (\mathcal{M}^H)^\mathrm{op}\times\mathcal{M}^H\to \mathcal{M}^H~~,~~(M,N)\mapsto
    (\mathrm{HOM}^\mathrm{Ad}(M,N),\delta^\mathrm{Ad})~,
\end{equation*}
which on morphisms $(g^\mathrm{op}\colon M\rightarrow M' ,h\colon N\rightarrow N')$
in $(\mathcal{M}^H)^\mathrm{op}\times\mathcal{M}^H$
reads
\begin{equation*}
    \mathrm{HOM}^\mathrm{Ad}(g^\mathrm{op},h)
    \colon\mathrm{HOM}^\mathrm{Ad}(M,N)\to
    \mathrm{HOM}^\mathrm{Ad}(M',N')~,~~ \phi\mapsto h\circ \phi\circ g
\end{equation*}
(as usual ${\mathcal C}^\mathrm{op}$
denotes the opposite category of $\mathcal{C}$, hence
 $g^\mathrm{op}\colon M\rightarrow M'$ in ${\mathcal C}^\mathrm{op}$
 is $g\colon M'\rightarrow M$ in ${\mathcal C}$ and
 $g^\mathrm{op}\circ_{\mathcal{C}^\mathrm{op}}{g'}^\mathrm{op}=(g'\circ_{\mathcal{C}}
 g)^\mathrm{op}$
).
In particular,  $\mathrm{HOM}^\mathrm{Ad}(f^\mathrm{op},g)$ is $H$-colinear
because of iii.) and iv.). Similarly, we have the functors
\begin{equation*}
  \begin{split}
    ({}^H\mathcal{M})^\mathrm{op}\times{}^H\mathcal{M}\to{}^H\mathcal{M}~~,~~&(M,N)\mapsto
    ({}^\mathrm{Ad}\mathrm{HOM}(M,N),\lambda^\mathrm{Ad})~,\\
    ({}^H\mathcal{M}^H)^\mathrm{op}\times{}^H\mathcal{M}^H\to{}^H\mathcal{M}^H,~~&(M,N)\mapsto
    (\mathrm{HOM}(M,N),\delta^\mathrm{Ad},\lambda^\mathrm{Ad})~.
\end{split}
\end{equation*}

As for the category of vector spaces these functors are right adjoint
to the tensor product functors $\otimes$, thus we obtain closed monoidal categories.
For completeness we provide a proof of this known fact.
\begin{proposition}\label{prop05}
$(\mathcal{M}^H,\otimes,\mathrm{HOM}^\mathrm{Ad})$,
$({}^H\mathcal{M},\otimes,{}^\mathrm{Ad}\mathrm{HOM})$
and $({}^H\mathcal{M}^H,\otimes,\mathrm{HOM})$
are closed monoidal categories.
\end{proposition}
\begin{proof}
Let $M,N, O$ be modules in $\mathcal{M}^H$. Similarly to
\cite{CaenGu07}~Prop.~1.2, the isomorphism  of vector spaces $\hat{}: \mathrm{Hom}(M\otimes
N,O)\to\mathrm{Hom}(M,\mathrm{Hom}(N,O))$
given by $\hat f(m)(n)= f(m\otimes n)$ restricts to the isomorphism of vector spaces
\begin{equation}\label{hatiso}
\hat{}\,: \mathrm{Hom}^H(M\otimes
N,O)\to\mathrm{Hom}^H(M,\mathrm{HOM}^\mathrm{Ad}(N,O))~.
\end{equation}
Indeed $H$-colinearity of any $f\in \mathrm{Hom}^H(M\otimes
N,O)$
implies that, for all $m\in M$, $n\in N$,
\begin{align*}
    \delta^\mathrm{Ad}(\hat{f}(m))(n)
    &=\hat{f}(m)(n_0)^{}_0\otimes\hat{f}(m)(n_0)^{}_1S(n_1)\\
    &=f(m\otimes n_0)^{}_0\otimes f(m\otimes n_0)^{}_1S(n_1)\\
    &=f(m_0\otimes n_0)\otimes m_1n_1S(n_2)\\
    &=\hat{f}(m_0)(n)\otimes m_1~,
\end{align*}
that is,
$\delta^\mathrm{Ad}(\hat f(m))=\hat f(m_0)\otimes m_1\in
\mathrm{HOM}^\mathrm{Ad}(N,O)\otimes H$. Hence $\hat f(m)\in
\mathrm{HOM}^\mathrm{Ad}(N,O)$. Furthermore, the equality
$\delta^\mathrm{Ad}(\hat f(m))=\hat f(m_0)\otimes m_1=(\hat f\otimes \mathrm{id})\delta(m)$ shows that
$\hat f$ is an $H$-colinear map. This shows that the restriction in
\eqref{hatiso} is well defined. We are left to prove surjectivity. 
Since $\hat{}: \mathrm{Hom}(M\otimes
N,O)\to\mathrm{Hom}(M,\mathrm{Hom}(N,O))$ is surjective we show
that for any $\hat f\in\mathrm{Hom}^H(M,\mathrm{HOM}^\mathrm{Ad}(N,O))$ we have $f\in \mathrm{Hom}^H(M\otimes
N,O)$. Indeed $$\delta(f(m\otimes n))=\delta(\hat f(m)(n))=\hat
f(m)^{}_0(n_0)\otimes \hat f(m)^{}_1n_1
=\hat f(m_0)(n_0)\otimes
m_1n_1=f(m_0\otimes n_0)\otimes m_1n_1\,,$$
where in the second equality we used that $\hat
f(m)\in\mathrm{HOM}^\mathrm{Ad}(N,O)$ and in the  third equality
$H$-colinearity of $\hat f$. Naturality of the isomorphism in  \eqref{hatiso} under $H$-colinear maps
$M\to M'$, $N\to N'$ and $O\to O'$ follows from naturality
of the isomorphism of vector spaces
$\,\hat{}: \mathrm{Hom}(M\otimes
N,O)\to\mathrm{Hom}(M,\mathrm{Hom}(N,O))$ since the functor
$\mathrm{HOM}^\mathrm{Ad}$ on morphisms is the restriction to $H$-colinear maps of
the vector space functor $\mathrm{Hom}$ on linear maps. 
This shows that $(\mathcal{M}^H,\otimes,\mathrm{HOM}^\mathrm{Ad})$ is
closed monoidal. A similar proof holds for $({}^H\mathcal{M},\otimes,{}^\mathrm{Ad}\mathrm{HOM})$
and $({}^H\mathcal{M}^H,\otimes,\mathrm{HOM})$.

\end{proof}
There are two other right and left adjoint $H$-coactions on
$\mathrm{Hom}(M,N)$, that in turn define other
rational morphisms ``acting from the right''. Namely
\begin{equation*}
    \overleftarrow{\delta}^\mathrm{Ad}\colon\mathrm{Hom}(M,N)\rightarrow
    \mathrm{Hom}(M,N\otimes H),~~~~~
    \overleftarrow{\delta}^\mathrm{Ad}(\phi)(m)=\phi(m_0)_0\otimes\overline{S}(m_1)\phi(m_0)_1
\end{equation*}
for right $H$-comodules $M,N$ and
\begin{equation*}
    \overleftarrow{\lambda}^\mathrm{Ad}\colon\mathrm{Hom}(M,N)\rightarrow
    \mathrm{Hom}(M,H\otimes N),~~~~~
    \overleftarrow{\lambda}^\mathrm{Ad}(\phi)(m))=S(m_{-1})\phi(m_0)_{-1}\otimes\phi(m_0)_0
\end{equation*}
for left $H$-comodules $M$ and $N$.
Setting $\overleftarrow{\mathrm{HOM}}^\mathrm{Ad}(M,N)
=(\overleftarrow{\delta}^\mathrm{Ad})^{-1}(\mathrm{Hom}(M,N)\otimes H)$,
${}^\mathrm{Ad}\overleftarrow{\mathrm{HOM}}(M,N)
=(\overleftarrow{\lambda}^\mathrm{Ad})^{-1}(H\otimes\mathrm{Hom}(M,N)$ and
$\overleftarrow{\mathrm{HOM}}(M,N)
=\overleftarrow{\mathrm{HOM}}^\mathrm{Ad}(M,N)
\cap{}^\mathrm{Ad}\overleftarrow{\mathrm{HOM}}(M,N)$ for $M,N$ in 
$\mathcal{M}^H$,
${}^H\mathcal{M}$
and ${}^H\mathcal{M}^H$, respectively,
we obtain $H$-coactions via the restrictions
\begin{equation*}
    \overleftarrow\delta^{{\mathrm{Ad}}}\colon{\overleftarrow{\mathrm{HOM}}}^{{\mathrm{Ad}}}(M,N)\to
    \overleftarrow{\mathrm{HOM}}^{{\mathrm{Ad}}}(M,N)\otimes H~,~~\phi\mapsto
    \phi_0\otimes\phi_1
\end{equation*}
\begin{equation*}
    {\overleftarrow\lambda}^{{\mathrm{Ad}}}\colon {}^{{\mathrm{Ad}}}{\overleftarrow{\mathrm{HOM}}}(M,N)\to
    H\otimes {}^{{\mathrm{Ad}}}{\overleftarrow{\mathrm{HOM}}}(M,N)~,~~\phi\mapsto
    \phi_{-1}\otimes\phi_0~.
\end{equation*}
These coactions further restrict to
coactions on $\overleftarrow{\mathrm{HOM}}(M,N):=\overleftarrow{\mathrm{HOM}}^{{\mathrm{Ad}}}(M,N)\cap
\overleftarrow{\mathrm{HOM}}^{{\mathrm{Ad}}}(M,N)$ and commute. The corresponding $H$-colinear evaluation and composition are the restrictions to
$\overleftarrow{\mathrm{HOM}}^{{\mathrm{Ad}}}(M,N)$,
${}^{{\mathrm{Ad}}}{\overleftarrow{\mathrm{HOM}}}(M,N)$ and $\overleftarrow{\mathrm{HOM}}(M,N)$ of the morphisms
$\overleftarrow{\mathrm{ev}}_{M,N}\colon M\otimes{}\mathrm{Hom}(M,N)\to N$,
$(m\otimes f)\mapsto f(m)$ and
\begin{equation*}
    \overleftarrow{\circ}_{M,N,O}:=\circ^\mathrm{op}_{M,N,O}\colon{}\mathrm{Hom}(M,N)
    \otimes{}\mathrm{Hom}(N,O)\to
    {}\mathrm{Hom}(M,O)~,~~f\otimes g\mapsto f\circ^\mathrm{op}g~,
\end{equation*}
where $(f\circ^\mathrm{op}g)(m):=g(f(m))$ for all $m\in M$.  $H$-colinearity of
$\overleftarrow{\mathrm{ev}}_{M,N}$ determines the module of internal morphisms
$\overleftarrow{\mathrm{HOM}}^{{\mathrm{Ad}}}(M,N)$
to be on the right in the evaluation of morphisms
$\overleftarrow{\mathrm{ev}}_{M,N}:M\otimes
\overleftarrow{\mathrm{HOM}}^{{\mathrm{Ad}}}(M,N)\to N$. We thus see that these morphisms act
from the right; consistently, also the $H$-colinear composition is the
usual composition for morphisms acting from the right. 
In the paper we mainly focus on $\mathrm{HOM}$ (morphisms acting from
the left) and develop the theory in
this setting. However, in Section~\ref{Sec5} we have to
make this crucial distinction.
The convention $\overleftarrow{\mathrm{HOM}}$ can be studied in complete
analogy by mirroring all expressions (cf. \cite{AschieriCartanStructure} for its study in
the context of $H$-modules rather than $H$-comodules).

\subsection{Covariant Modules}\label{Sec2.2}

Considering vector spaces which are both,
$H$-modules and $H$-comodules, there is a natural compatibility
condition of those structures: we can ask the coaction to be a morphism
in the monoidal category of $H$-modules or equivalently the action to be
a morphism in the monoidal category of $H$-comodules. This leads us to the
notion of covariant modules, also known as Hopf modules.

\begin{definition}[Covariant Modules or Hopf Modules]
The category $\mathcal{M}_H^H$ of right covariant right
modules consists of
all right $H$-modules $M$ which are also right $H$-comodules
such that the right coaction
$\delta_M\colon M\rightarrow M\otimes H$ is right $H$-linear,
i.e. $\delta_M(ma)=\delta_M(m)\Delta(a)$ for all $m\in M$ and $a\in H$.
Morphisms are right $H$-linear and right $H$-colinear.
The categories ${}_H\mathcal{M}^H$,
${}^H\mathcal{M}_H$, ${}_H^H\mathcal{M}$, ${}_H^H\mathcal{M}_H$,
${}_H\mathcal{M}^H_H$, ${}_H^H\mathcal{M}^H$, ${}^H\mathcal{M}^H_H$
and ${}_H^H\mathcal{M}_H^H$ are similarly defined.
In particular ${}_H^H\mathcal{M}_H^H$ is the category of bicovariant bimodules. A vector space $M$ is in  ${}_H^H\mathcal{M}_H^H$ if it
is an $H$-bimodule and an $H$-bicomodule such that the right and left coactions $\delta_M$ and
$\lambda_M$ are $H$-bilinear, i.e.,
\begin{equation*}
    \delta_M(amb)
    =\Delta(a)\delta_M(m)\Delta(b)
    \text{ and }
    \lambda_M(amb)
    =\Delta(a)\lambda_M(m)\Delta(b)
\end{equation*}
for all $a,b\in H$ and $m\in M$.
The corresponding morphisms are $H$-bilinear and $H$-colinear maps.
The vector space of morphisms between two bicovariant 
bimodules $M,N$ is denoted by ${}^H_H\mathrm{Hom}^H_H(M,N)$
and similarly for the other types of covariant modules.
\end{definition}
Canonical examples of bicovariant bimodules of $H$ are the
ground field $\Bbbk$ (with right and left $H$-action given by 
$\epsilon\colon H\cong\Bbbk\otimes H\cong H\otimes\Bbbk
\rightarrow\Bbbk$ and right and left $H$-coaction
given by $\eta\colon\Bbbk\rightarrow H\cong\Bbbk\otimes H
\cong H\otimes\Bbbk$) and $H$ itself (with actions given by 
right and left multiplication and coaction $\Delta\colon
H\rightarrow H\otimes H$). Covariant modules arise in the study of
the differential calculus on quantum groups \cite{Woronowicz1989}.

Any covariant module is $H$-generated by its coinvariant
elements. The construction is as follows:
for a Hopf algebra $H$ and right covariant right module $M$
we can structure the $\Bbbk$-module $M^{\mathrm{co}H}\otimes H$ as a right
covariant right module such that it is isomorphic to $M$.
If $M$ is further a bicovariant bimodule this isomorphism extends
to an isomorphism of bicovariant bimodules,  see e.g. \cite{Mo93}.

\begin{theorem}[Fundamental Theorem of Hopf Modules]\label{FundThm}
For any module $M$ of $\mathcal{M}^H_H$ there is an isomorphism
\begin{equation*}
    M\rightarrow M^{\mathrm{co}H}\otimes H~,~~
    m\mapsto m_0S(m_1)\otimes m_2~
\end{equation*}
of right covariant right modules with inverse
$M^{\mathrm{co}H}\otimes H\rightarrow M$, $m\otimes h\mapsto m\cdot h$.
If $M$ is a bicovariant bimodule this is an isomorphism in
${}_H^H\mathcal{M}_H^H$, where $M^{\mathrm{co}H}\otimes H$ is structured as
a bicovariant bimodule via  the $H$-actions
$a\cdot(m\otimes h)\cdot b:=a_1mS(a_2)\otimes a_3hb$,
for all $a,b,h\in H$ and $m\in M^{\mathrm{co}H}$, and the diagonal $H$-coactions.

Analogous results can be formulated for modules of ${}_H^H\mathcal{M}$, ${}_H\mathcal{M}^H$
and ${}^H\mathcal{M}_H$, with isomorphisms:
\begin{equation*}
  \begin{split}
&   M\rightarrow H\otimes{}^{\,\mathrm{co}H\!}M~,~~
m\mapsto m_{-2}\otimes S(m_{-1})m_0~,\\
& M\rightarrow H\otimes M^{\mathrm{co}H}~,~~
m\mapsto m_2\otimes\overline{S}(m_1)m_0~,\\
& M\rightarrow{}^{\,\mathrm{co}H\!}M\otimes H~,~~
   m\mapsto m_0\overline{S}(m_{-1})\otimes m_{-2}~.
\end{split}
\end{equation*}
In particular, all covariant modules are free $H$-modules.
\end{theorem}
Given a bicovariant bimodule its right (or left) coinvariants form a so-called Yetter-Drinfel'd
module, i.e. a module and comodule with a given compatibility condition
(different from the one of covariant modules).
On the other hand, the tensor product of a Yetter-Drinfel'd module
with $H$ gives a bicovariant bimodule.
The fundamental theorem of Hopf modules shows that these processes are inverse
to each other, providing a (monoidal) equivalence of categories. We will use the
notion of bicovariant bimodule and omit discussing Yetter-Drinfel'd
modules, we refer the interested reader to \cite{Schauenburg94} for the categorical
equivalence.

We notice that on the tensor product $M\otimes N$ of modules $M,N$
of ${}_H^H\mathcal{M}_H^H$ we have the $H$-bimodule structure
\begin{equation*}
    a\cdot(m\otimes n)\cdot b
    =am\otimes nb,
    \text{ for }a,b\in H,~m\in M,~n\in N~,
\end{equation*}
as well as the $H$-bicomodule
structure given by the diagonal coactions. The covariance conditions
ensure that these structures are compatible.
From
Lemma~\ref{lemma01} i''.) we know that the adjoint coactions structure
the vector space of rational morphisms
$\mathrm{HOM}(M,N)=\mathrm{HOM}^\mathrm{Ad}(M,N)\cap{}^\mathrm{Ad}\mathrm{HOM}(M,N)$
as an $H$-bicomodule if $M,N$ are $H$-bicomodules. For bicovariant left modules
$M,N$ we furthermore have the following $H$-bimodule structure
\begin{equation}\label{eq27}
    (a\cdot\phi\cdot b)(m)=a\phi(bm)
    \text{ for }a,b\in H,~\phi\in\mathrm{HOM}(M,N),~m\in M
\end{equation}
on $\mathrm{HOM}(M,N)$. Covariance ensures that this
bimodule structure is compatible with $\delta^\mathrm{Ad}$ and $\lambda^\mathrm{Ad}$.
\begin{lemma}\label{lemma04}
Given modules $M,N$ of ${}_H^H\mathcal{M}^H$ the vector space of
rational morphisms $\mathrm{HOM}(M,N)$ is a bicovariant bimodule with
respect to the $H$-actions in (\ref{eq27}) and the adjoint coactions
$\delta^\mathrm{Ad}$ and $\lambda^\mathrm{Ad}$. The assignment
$ \mathrm{HOM}
    \colon({}_H^H\mathcal{M}^H)^\mathrm{op}\times{}_H^H\mathcal{M}^H
    \rightarrow{}_H^H\mathcal{M}_H^H$ is functorial.
\end{lemma}
\begin{proof}
In Proposition~\ref{prop05} we already noted that $\delta^\mathrm{Ad}$ and 
$\lambda^\mathrm{Ad}$ are commuting coactions on bicomodules.
The $H$-actions in (\ref{eq27}) are compatible with $\delta^\mathrm{Ad}$ since
\begin{align*}
    (a_1\cdot\phi_0\cdot b_1)(m)\otimes a_2\phi_1b_2
    &=a_1\phi_0(b_1m)\otimes a_2\phi_1b_2\\
    &=a_1\phi(b_1m_0)_0\otimes a_2\phi(b_1m_0)_1S(b_2m_1)b_3\\
    &=a_1\phi(bm_0)_0\otimes a_2\phi(bm_0)_1S(m_1)\\
    &=(a\phi(bm_0))_0\otimes(a\phi(bm_0))_1S(m_1)\\
    &=(a\cdot\phi\cdot b)(m_0)_0\otimes(a\cdot\phi\cdot b)(m_0)_1S(m_1)\\
    &=(a\cdot\phi\cdot b)_0(m)\otimes(a\cdot\phi\cdot b)_1
\end{align*}
for all $a,b\in H$, $\phi\in\mathrm{HOM}(M,N)$ and $m\in M$.
Similarly, the $H$-actions are compatible with $\lambda^\mathrm{Ad}$.
On morphisms 
$f\colon M'\rightarrow M$, $g\colon N\rightarrow N'$ in
${}_H^H\mathcal{M}^H$ we define
\begin{equation}\label{eq76}
    \mathrm{HOM}(f^\mathrm{op},g)\colon\mathrm{HOM}(M,N)\rightarrow\mathrm{HOM}(M',N')~,~~~~~
    \phi\mapsto g\circ\phi\circ f
\end{equation}
and show that it is a morphism in ${}_H^H\mathcal{M}_H^H$. It is well-defined and
$H$-bicolinear because of Proposition~\ref{prop05}. Moreover
(\ref{eq76}) is $H$-bilinear: for $a,b\in H$ and
$\phi\in\mathrm{HOM}(M,N)$ from left $H$-linearity of $f$ and $g$ it
follows that
\begin{align*}
    [\mathrm{HOM}(f^\mathrm{op},g)(a\cdot\phi\cdot b)](m)
    &=(g\circ(a\cdot\phi\cdot b)\circ f)(m)\\
    &=g(a\phi(bf(m)))\\
    &=ag(\phi(f(bm)))\\
    &=[a\cdot[\mathrm{HOM}(f^\mathrm{op},g)(\phi)]\cdot b](m)
\end{align*}
for all $m\in M$.
\end{proof}
Analogous statements hold for the functor $\overleftarrow{\mathrm{HOM}}$
of rational morphisms acting from the right: given modules $M,N$ of ${}^H\mathcal{M}^H_H$
we structure $\overleftarrow{\mathrm{HOM}}(M,N)$ as a bicovariant bimodule via the $H$-actions
\begin{equation}\label{eq02}
    (a\cdot\phi\cdot b)(m)=\phi(ma)b
    \text{ for all }a,b\in H,~\phi\in\overleftarrow{\mathrm{HOM}}(M,N),~m\in M
\end{equation}
and the adjoint coactions $\overleftarrow{\delta}^\mathrm{Ad}$, 
$\overleftarrow{\lambda}^\mathrm{Ad}$. Recall that the evaluation of
$\overleftarrow{\mathrm{HOM}}(M,N)$ is defined by $\overleftarrow{\mathrm{ev}}\colon
M\otimes\overleftarrow{\mathrm{HOM}}(M,N)\rightarrow N$, $m\otimes\phi\mapsto\phi(m)$,
thus explaining the counterintuitive definition (\ref{eq02}).

Utilizing the $H$-module structures on rational morphisms we 
refine Proposition~\ref{prop05}.
\begin{proposition}
$({}_H^H\mathcal{M}^H,\otimes,\mathrm{HOM})$ and 
$({}^H\mathcal{M}^H_H,\otimes,\overleftarrow{\mathrm{HOM}})$
are closed monoidal categories.
\end{proposition}
\begin{proof}
Let $M,N$ be modules of ${}_H^H\mathcal{M}^H$.
In Proposition~\ref{prop05} it is shown that
$({}^H\mathcal{M}^H,\otimes,\mathrm{HOM})$ is a closed monoidal category
and according to Lemma~\ref{lemma04} the vector space $\mathrm{HOM}(M,N)$ is a module of 
${}_H^H\mathcal{M}^H_H$, so in particular of ${}_H^H\mathcal{M}^H$.
Thus  it is left
to prove that $\otimes$ and $\mathrm{HOM}$ are adjoint via the external $\mathrm{Hom}$-functor
${}_H^H\mathrm{Hom}^H$. Indeed the isomorphism of vector spaces in  (\ref{hatiso})
restricts to a well-defined injective linear map 
\begin{equation*}
    {}_H^H\mathrm{Hom}^H(M\otimes N,O)
    \longrightarrow{}_H^H\mathrm{Hom}^H(M,\mathrm{HOM}(N,O))~,
\end{equation*}
where $O$ is another bicovariant left module. Its surjectivity is proven by showing that the preimage of
${}_H^H\mathrm{Hom}^H(M,\mathrm{HOM}(N,O))\subseteq\mathrm{Hom}^H(M,\mathrm{HOM}(N,O))$ under (\ref{hatiso}) is in  $ {}_H^H\mathrm{Hom}^H(M\otimes N,O)$.
The statement about $({}^H\mathcal{M}^H_H,\otimes,\overleftarrow{\mathrm{HOM}})$
is similarly proven.\end{proof}

Note that, even though $\otimes$ and $\mathrm{HOM}$ are functors
$({}_H^H\mathcal{M}_H^H)^\mathrm{op}\times{}_H^H\mathcal{M}_H^H
\rightarrow{}_H^H\mathcal{M}_H^H$, according to Lemma~\ref{lemma04}
they are not adjoint to each other in the category of bicovariant bimodules, since in general
\begin{equation*}
    {}_H^H\mathrm{Hom}_H^H(M\otimes N,O)
    \ncong{}_H^H\mathrm{Hom}_H^H(M,\mathrm{HOM}(N,O))
\end{equation*}
fails to be an isomorphism for $M,N,O$ bicovariant bimodules.
Thus, we are left with the following result.
\begin{corollary}\label{cor:closed-mon}
The category of bicovariant bimodules ${}_H^H\mathcal{M}_H^H$
\begin{enumerate}
\item[i.)] is a monoidal category $({}_H^H\mathcal{M}_H^H,\otimes)$
with respect to the tensor product of vector spaces, the monoidal unit is
$\Bbbk$;
\item[ii.)] is a biclosed category
  $({}_H^H\mathcal{M}_H^H,\mathrm{HOM}, \overleftarrow{\mathrm{HOM}})$
with the two internal $\mathrm{Hom}$-functors
\begin{equation*}
    \mathrm{HOM} \colon({}_H^H\mathcal{M}_H^H)^\mathrm{op}\times{}_H^H\mathcal{M}_H^H
    \rightarrow{}_H^H\mathcal{M}_H^H~,~~~\overleftarrow{\mathrm{HOM}}
    \colon({}_H^H\mathcal{M}_H^H)^\mathrm{op}\times{}_H^H\mathcal{M}_H^H
    \rightarrow{}_H^H\mathcal{M}_H^H~.
\end{equation*}
The first functor is defined in Lemma \ref{lemma04}. The second one
assigns to any pair $(M,N)$ of bicovariant bimodules the bicovariant
bimodule $\overleftarrow{\mathrm{HOM}}(M,N)$ with the adjoint coaction
$\overleftarrow{\delta}^\mathrm{Ad}$, $\overleftarrow{\lambda}^\mathrm{Ad}$ and the
$H$-actions in (\ref{eq02}).
\end{enumerate}
\end{corollary}

There is another monoidal structure on bicovariant bimodules:
Given two bicovariant bimodules $M,N$, their balanced tensor product over $H$,
$M\otimes_HN$,  
is a bicovariant bimodule  
via $a\cdot(m\otimes_Hn)\cdot b:=am\otimes_Hnb$
and
\begin{equation*}
    \delta_{M\otimes_HN}(m\otimes_Hn)
    :=(m_0\otimes_Hn_0)\otimes m_1n_1~,~~
    \lambda_{M\otimes_HN}(m\otimes_Hn)
    :=m_{-1}n_{-1}\otimes(m_0\otimes_Hn_0)
\end{equation*}
for all $a,b\in H$ and $m\otimes_Hn\in M\otimes_HN$. 
We briefly recall the construction of a braiding for the monoidal category $({}_H^H\mathcal{M}_H^H,\otimes_H)$.
Following Woronowicz \cite{Woronowicz1989} there is a unique $H$-bilinear map
$\sigma^\mathcal{W}_{M,N}\colon M\otimes_HN\rightarrow N\otimes_HM$ that 
on coinvariant (coaction invariant) elements
$\overline{m}\in{}^{\mathrm{co}H}M$, $\overline{n}\in
N^{\mathrm{co}H}$ acts as the flip operator 
$\sigma^\mathcal{W}_{M,N}(\overline{m}\otimes_H\overline{n})
=\overline{n}\otimes_H\overline{m}$.
From the identity 
$m\otimes n=m_{-2}S(m_{-1})m_0\otimes n_0S(n_1)n_2$ with 
$S(m_{-1})m_0\otimes n_0S(n_1)\in {}^{\mathrm{co}H}M\otimes
N^{\mathrm{co}H}$ and $H$-bilinearity we have that
$\sigma^\mathcal{W}_{M,N}\colon M\otimes_HN\rightarrow N\otimes_HM$ is
given by
\begin{equation}\label{WorBraiding}
    \sigma^\mathcal{W}_{M,N}(m\otimes_Hn)
    :=m_{-2}n_0S(n_1)\otimes_HS(m_{-1})m_0n_2~.
\end{equation}
In the following we shall frequently use the short notation ${}_\alpha n\otimes_H{}^\alpha m
:=\sigma^\mathcal{W}_{M,N}(m\otimes_Hn)$. 
This map is an isomorphism of bicovariant bimodules, the inverse
$\overline{\sigma}^\mathcal{W}_{N,M}\colon N\otimes_HM\rightarrow
M\otimes_HN$ is given by
\begin{equation*}
    \overline{\sigma}^\mathcal{W}_{N,M}(n\otimes_Hm)
    :=n_2m_0\overline{S}(m_{-1})\otimes_H\overline{S}(n_1)n_0m_{-2}~,
\end{equation*}
or equivalently on coinvariant elements $\overline{m}\in{}^{\mathrm{co}H}M$, 
$\overline{n}\in N^{\mathrm{co}H}$ by
$\overline{\sigma}^\mathcal{W}_{N,M}(\overline{n}\otimes_H\overline{m})
=\overline{m}\otimes_H\overline{n}$.
It turns out that for all bicovariant bimodules $M,N,O$ the \textit{hexagon equations}
\begin{equation*}
    \sigma^\mathcal{W}_{M\otimes_HN,O}
    =(\sigma^\mathcal{W}_{M,O}\otimes_H\mathrm{id}_N)
    \circ(\mathrm{id}_M\otimes_H\sigma^\mathcal{W}_{N,O})
    \quad\text{ ~and~ }\quad
    \sigma^\mathcal{W}_{M,N\otimes_HO}
    =(\mathrm{id}_N\otimes_H\sigma^\mathcal{W}_{M,O})
    \circ(\sigma^\mathcal{W}_{M,N}\otimes_H\mathrm{id}_O)~~
\end{equation*}
hold. In leg notation the above equalities read
\begin{equation*}
    \sigma^\mathcal{W}_{12,3}
    =\sigma^\mathcal{W}_{12}\circ\sigma^\mathcal{W}_{23}
    \quad\text{ ~and~ }\quad
    \sigma^\mathcal{W}_{1,23}
    =\sigma^\mathcal{W}_{23}\circ\sigma^\mathcal{W}_{12}~,
\end{equation*}
where $\sigma^\mathcal{W}_{12,3}=\sigma^\mathcal{W}_{M\otimes_HN,O}$,
$\sigma^\mathcal{W}_{12}=\sigma^\mathcal{W}\otimes_H\mathrm{id}$, et cetera.
As a result of the hexagon equations the \textit{quantum Yang-Baxter equation}
\begin{equation}\label{QYBE}
    \sigma^\mathcal{W}_{12}\circ\sigma^\mathcal{W}_{23}\circ\sigma^\mathcal{W}_{12}
    =\sigma^\mathcal{W}_{23}\circ\sigma^\mathcal{W}_{12}\circ\sigma^\mathcal{W}_{23}
\end{equation}
follows. In summary, $({}_H^H\mathcal{M}_H^H,\otimes_H,\sigma^\mathcal{W})$ is a braided
monoidal category, cf. for example \cite{Schauenburg94}~Theorem~6.3.

This category admits an internal $\mathrm{Hom}$-functor.
For modules $M,N$ of  ${}^H\mathcal{M}^H_H$   we define the vector space $$\mathrm{HOM}_H(M,N)
:=\mathrm{HOM}(M,N)\cap{}\mathrm{Hom}_H(M,N)
$$ 
of  right $H$-linear rational morphisms. It is easily seen to be a  sub bicovariant
bimodule of $\mathrm{HOM}(M,N)$.
\begin{proposition}\label{prop07}
$({}_H^H\mathcal{M}_H^H,\otimes_H,\mathrm{HOM}_H,\sigma^\mathcal{W})$ is a 
closed braided monoidal category. 
\end{proposition}
\begin{proof}
We prove that the internal $\mathrm{Hom}$-functor 
\begin{equation*}
    \mathrm{HOM}_H\colon({}_H^H\mathcal{M}_H^H)^\mathrm{op}\times{}_H^H\mathcal{M}_H^H
    \rightarrow{}_H^H\mathcal{M}_H^H~,~~~~~
    (M,N)\mapsto\mathrm{HOM}_H(M,N)
  \end{equation*}
  is adjoint to the tensor product functor $\otimes_H$, that is, for
  all bicovariant bimodules $M,N,O$ we have the bijection
\begin{equation*}
    {}_H^H\mathrm{Hom}_H^H(M\otimes_HN,O)\,\cong\,
    {}_H^H\mathrm{Hom}_H^H(M,\mathrm{HOM}_H(N,O))~,
  \end{equation*}
which is natural in all its arguments.
For any
$f\in{}_H^H\mathrm{Hom}_H^H(M\otimes_HN,O)$, define
$\hat{f}: M\to \mathrm{HOM}(N,O)$, $m\mapsto  \hat{f}(m)$
by $\hat{f}(m)(n)=f(m\otimes_Hn)$ for all $n\in N$.
The map $\hat{f}$ is
well defined because the inclusion 
${}_H^H\mathrm{Hom}_H^H(M\otimes_HN,O)\subseteq
{}^H\mathrm{Hom}^H(M\otimes N,O)$ (given by composing with the
projection $M\otimes N\to M\otimes_H N$) allows to apply
Proposition~\ref{prop05}. Moreover, $\hat{f}: M\to
\mathrm{HOM}_H(N,O)$ since
$\hat{\phi}(m)$ is right $H$-linear for all $m\in M$:  
$
\hat{\phi}(m)(na)
=\phi(m\otimes_Hna)
=\phi(m\otimes_Hn)a
=\hat{\phi}(m)(n)a
$
for all $a\in H$ and $n\in N$ by the right $H$-linearity of $\phi$.
The map $\,\hat{}\,: {}_H^H\mathrm{Hom}_H^H(M\otimes_HN,O)\,\to\,
    {}_H^H\mathrm{Hom}_H^H(M,\mathrm{HOM}_H(N,O))$, $f\mapsto \hat{f}$, is well defined
    since $\hat{f}$ is a
    morphism of bicovariant bimodules: left $H$-linearity and $H$-bicolinearity are
obvious, while right $H$-linearity holds since
$$
\hat{\phi}(m a)(n)
=\phi(ma\otimes_Hn)
=\phi(m\otimes_Han)
=(\hat{\phi}(m)\cdot a)(n)
$$
for all $m\in M$, $n\in N$ and $a\in H$. One proves that the map $\,\hat{}\,$ is invertible
and that it is a natural transformation in complete analogy with
Proposition~\ref{prop05}.
\end{proof}

In the closed braided monoidal category $({}_H^H\mathcal{M}_H^H,\otimes_H,
\mathrm{HOM}_H,\sigma^\mathcal{W})$
we define the evaluation $\mathrm{ev}_{}(\phi\otimes_Hm):=\phi(m)$ and
composition $\psi\circ\phi$ 
of rational morphisms
$\phi\in\mathrm{HOM}_H(M,N)$, $\psi\in\mathrm{HOM}_H(N,O)$, where
$(\psi\circ\phi)(m)=\psi(\phi(m))$ for all $m\in M$.
We further consider the \textit{braided tensor product} of rational morphisms
\begin{equation}\label{eq29}
    (\phi\otimes_{\sigma^\mathcal{W}}\psi)(m\otimes_Hn)
    :=\phi({}_\alpha m)\otimes_H({}^\alpha\psi)(n)
    =\phi(\psi_{-2}m_0S(m_1))\otimes_HS(\psi_{-1})\psi_0(m_2n)
\end{equation}
for $\phi\in\mathrm{HOM}_H(M,M')$, $\psi\in\mathrm{HOM}_H(N,N')$ and $m\otimes_Hn
\in M\otimes_HN$, where $M,M',N,N',O$ are bicovariant bimodules.
From \cite{Ma95}~Proposition~9.3.13 it follows:
\begin{corollary}\label{cor01}
Let $M,M',N,N',O$ be modules of ${}_H^H\mathcal{M}_H^H$. Then
\begin{equation*}
    \mathrm{ev}_{M,N}\colon\mathrm{HOM}_H(M,N)\otimes_H M\rightarrow N,
\end{equation*}
\begin{equation*}
    \circ_{M,N,O}\colon\mathrm{HOM}_H(N,O)\otimes_H\mathrm{HOM}_H(M,N)
    \rightarrow\mathrm{HOM}_H(M,O)
\end{equation*}
and
\begin{equation}\label{eq04}
    \otimes_{\sigma^\mathcal{W}}\colon
    \mathrm{HOM}_H(M,M')\otimes_H\mathrm{HOM}_H(N,N')
    \rightarrow\mathrm{HOM}_H(M\otimes_HN,M'\otimes_H,N')
\end{equation}
are morphisms of bicovariant bimodules. Furthermore, $\circ$ and
$\otimes_{\sigma^\mathcal{W}}$ are associative.
\end{corollary}

In the next section we shall consider evaluation, composition and
tensor product of rational morphisms that are not necessarily
$H$-linear, like for example connections.
\\

We conclude this section observing  that every right $H$-linear
map $\phi\in\mathrm{Hom}_H(M,N)$ between covariant modules is a 
rational morphism if $M$ is finitely generated  as an $H$-module. 
\begin{proposition}\label{prop100}
i.) Let $M,N$ be modules of $\mathcal{M}_H^H$, with $M$ finitely generated 
 as a right $H$-module, then $\mathrm{Hom}_H(M,N)\subseteq\mathrm{HOM}^\mathrm{Ad}(M,N)$.
ii.) Let $M,N$ be modules of ${}^H\mathcal{M}_H$ with $M$  finitely generated 
 as a left $H$-module, then
$\mathrm{Hom}_H(M,N)\subseteq {}^\mathrm{Ad}\mathrm{HOM}(M,N)$.
iii.)  Let $M,N$ be bicovariant bimodules, with $M$ finitely
generated, then  $\mathrm{Hom}_H(M,N)=\mathrm{HOM}_H(M,N)$.
\end{proposition}
\begin{proof}
\begin{enumerate}
\item[i.)] For $M,N$ in $\mathcal{M}_H^H$ 
define $N\otimes H$ as a right $H$-module via the right $H$-action
\begin{equation}\label{eq10'}
    (n\otimes k)\cdot h:=nh\otimes k
\end{equation}
for all $n\in N$, $k,h\in H$. The thesis follows from $\mathrm{Hom}_H(M,N\otimes H)=
\mathrm{Hom}_H(M,N)\otimes H$, which we now prove. We have already seen in \eqref{incinNH} that any $\phi\otimes
k\in\mathrm{Hom}_H(M,N)\otimes H$ is in $\mathrm{Hom}(M,N\otimes H)$;
we show that it is in $\mathrm{Hom}_H(M,N\otimes H)$. Indeed
$$
(\phi\otimes k)(mh)
=\phi(mh)\otimes k
=\phi(m)h\otimes k
=(\phi(m)\otimes k)\cdot h
=((\phi\otimes k)(m))\cdot h
$$
for all $m\in M$ and $h\in H$.
We are left to prove the other inclusion  $\mathrm{Hom}_H(M,N\otimes H)\subseteq
\mathrm{Hom}_H(M,N)\otimes H$.  We show  $\mathrm{Hom}_H(M,N\otimes H)\subseteq
\mathrm{Hom}(M,N)\otimes H$, then $\mathrm{Hom}_H(M,N\otimes H)\subseteq
\mathrm{Hom}_H(M,N)\otimes H$ easily follows.  Let $\{e_i\}$, $i=1,\ldots n$, be a generating set of the finitely generated right
$H$-module $M$. For any $m\in M$ we have $m=e^ih_i$ with $h_i\in H$ (sum
over $i=1,\ldots n$ understood). For any $\psi\in \mathrm{Hom}_H(M,N\otimes H)$ we have
$\psi(e^i)=n^{i \alpha}\otimes k^i_\alpha\in N\otimes H$
(finite sum over the index $\alpha$ understood) and $\psi(m)=\psi(e^ih_i)=
\psi(e^i)\cdot h_i=n^{i\, \alpha} h_i\otimes k^i_\alpha$. Thus the
$H$-linear map $\psi$ has image in
$N\otimes S_\psi$, where $$S_\psi:=\mathrm{span}\{k^i_\alpha\}\subseteq
H$$
is the finite
dimensional vector space given by the linear span of the elements
$k^i_\alpha\in H$, ($\{{}^i_\alpha\}$ is a finite  index
set). Hence $\psi\in\mathrm{Hom}_H(M,N\otimes
S_\psi)\subseteq\mathrm{Hom}(M,N)\otimes S_\psi\subseteq \mathrm{Hom}(M,N)\otimes H$,
where in the first inclusion we used $\mathrm{Hom}(M,N\otimes
S_\psi)=\mathrm{Hom}(M,N)\otimes S_\psi$ that  holds trivially since $S_\psi$ is
finite-dimensional.

\item[ii.)] Following a similar argumentation one proves that 
$\mathrm{Hom}_H(M,H\otimes N)= H\otimes\mathrm{Hom}_H(M,N)$ where 
 $H\otimes N$ is a right $H$-module via $(k\otimes n)\cdot h =k\otimes nh$ 
 for all $k,h\in H, n\in N$.
 
 \item[iii.)] This is straightforward from i.) and ii.).
\end{enumerate}
\end{proof}

\subsection{Tensor Product of Rational Morphisms}\label{Sec2.3}

Let us consider the monoidal category $({}_H^H\mathcal{M}_H^H,\otimes)$ with
the tensor product $\otimes$ of vector spaces. From Corollary~\ref{cor:closed-mon} this
category is also closed with internal Hom-functor  $\mathrm{HOM}$, but it is not closed monoidal for a generic Hopf algebra
$H$, nor is it braided. We shall introduce a linear map
$\sigma_{M,N}:M\otimes N\to
N\otimes M$ that is a lift of the
braiding $\sigma^\mathcal{W}_{M,N}: M\otimes_H N\to
N\otimes_H M$ of bicovariant bimodules and 
study  analogues of Corollary~\ref{cor01}, first by considering 
the tensor product $\otimes$ and the rational morphisms
$\mathrm{HOM}(M,N)$ and then by considering the balanced tensor
product $\otimes_H$ and the rational morphisms $\mathrm{HOM}(M,N)$.
Let $M, N, O$ be bicovariant bimodules. We begin by defining the evaluation \begin{equation}\label{eq05}
    \mathrm{ev}_{M,N}\colon\mathrm{HOM}(M,N)\otimes M\rightarrow N~,~~~~~
    \phi\otimes m\mapsto\phi(m)
\end{equation}
and the composition
\begin{equation}\label{eq06}
    \circ_{M,N,O}\colon\mathrm{HOM}(N,O)\otimes\mathrm{HOM}(M,N)
    \rightarrow\mathrm{HOM}(M,O)~,~~~~~
    \psi\otimes\phi\mapsto\psi\circ\phi
\end{equation}
of rational morphisms. Lemma~\ref{lemma01} implies that 
the evaluation in  (\ref{eq05}) is a morphism in
${}_H^H\mathcal{M}^H$, while the composition in (\ref{eq06}) is a
morphism in ${}_H^H\mathcal{M}^H_H$.
If and only if we consider $\mathrm{HOM}_H(M,N)$ the evaluation is right
$H$-linear and hence a morphism in ${}_H^H\mathcal{M}^H_H$.

In order to define a tensor product of rational morphisms we consider
the linear map:
\begin{equation}\label{sigma}
    \sigma_{M,N}\colon M\otimes N\rightarrow N\otimes M~,~~~~~
    m\otimes n\mapsto m_{-2}n_0S(m_{-1}n_1)\otimes m_0n_2~.
  \end{equation}
  This is a lift of the braiding $\sigma^\mathcal{W}_{M,N}\colon M\otimes_HN\rightarrow N\otimes_HM$
  since the diagram
\begin{equation*}
\begin{tikzcd}
M\otimes N
\arrow{rr}{\sigma_{M,N}}
\arrow{d}[swap]{\pi_H^{M,N}}
& & N\otimes M
\arrow{d}{\pi_H^{N,M}} \\
M\otimes_HN
\arrow{rr}{\sigma^\mathcal{W}_{M,N}}
& & N\otimes_HM
\end{tikzcd}
\end{equation*}
commutes. 
Accordingly, the short notation ${}_\alpha n\otimes_H{}^\alpha
m=\sigma^\mathcal{W}_{M,N}(m\otimes_H n)$  lifts to 
${}_\alpha n\otimes{}^\alpha m=\sigma_{M,N}(m\otimes n)$. 
A refinement of the above commutative diagram is given by
\begin{equation*}
\begin{tikzcd}
M\otimes N
\arrow{rr}{\sigma_{M,N}}
\arrow{d}[swap]{~~~~~~~~~~~\pi_H^{M,N}}
& & N^{\mathrm{co}H}\otimes M
\arrow{d}{\pi_H^{N,M}|^{}_{N^{\mathrm{co}H}\otimes M} } \\
M\otimes_HN
\arrow{rr}{\sigma^\mathcal{W}_{M,N}}
& & N\otimes_HM
\end{tikzcd}
\end{equation*}
where the restricted projection $\pi^{N,M}_H|^{}_{N^{\mathrm{co}H}\otimes M} : N^{\mathrm{co}H}\otimes M \to N\otimes_H
M$ is an isomorphism in ${}^H\mathcal{M}^H_H$ since,  according to the
Fundamental Theorem of Hopf modules, $N\otimes_H M\cong (N^{\mathrm{co}H}\otimes
H)\otimes_HM\cong N^{\mathrm{co}H}\otimes M$.
The inclusion $\sigma_{M, N}(M\otimes N)\subseteq N^{\mathrm{co}H}\otimes M$
is immediate from the definition of $\sigma_{M, N}$. We further have
$$ \sigma_{M, N}(M\otimes N)=N^{\mathrm{co}H}\otimes M~.$$ This
equality follows form commutativity of the diagram, i.e., $\sigma_{M, N}=
(\pi_H^{N,M}|^{}_{N^{\mathrm{co}H}\otimes M})^{-1}\circ
\sigma^\mathcal{W}_{M,N}\circ\pi_H^{N,M}$, observing that all the maps
on the right hand side are surjective.
\begin{lemma}\label{lemma02}
For any pair $(M,N)$ of bicovariant bimodules 
$\sigma_{M,N}\colon M\otimes N\rightarrow N\otimes M$ is a morphism in
${}^H\mathcal{M}^H_H$ and satisfies $\sigma_{M,N}(ma\otimes n)=\sigma_{M,N}(m\otimes an)$
for all $a\in H$, $m\in M$, $n\in N$. Furthermore, the hexagon
equation
\begin{equation}\label{eq19}
    \sigma_{M\otimes N,O}
    =(\sigma_{M,O}\otimes\mathrm{id}_N)\circ(\mathrm{id}_M\otimes\sigma_{N,O})
\end{equation}
holds for all modules $M,N,O$ of ${}_H^H\mathcal{M}_H^H$.
\end{lemma}
\begin{proof}
The (co-)linearity properties of $\sigma_{M,N}$ are obvious. We verify the
hexagon equation (\ref{eq19}). For $m\in M$, $n\in N$ and $o\in O$ the
expressions $\sigma_{M\otimes N,O}(m\otimes n\otimes o)
=m_{-2}n_{-2}o_0S(m_{-1}n_{-1}o_1)\otimes(m_0\otimes n_0)o_2$
and
\begin{align*}
    (\sigma_{M,O}\otimes\mathrm{id}_N)(m\otimes\sigma_{N,O}(n\otimes o))
    &=\sigma_{M,O}(m\otimes n_{-2}o_0S(n_{-1}o_1))\otimes n_0o_2\\
    &=m_{-2}n_{-2}o_0S(n_{-1}o_1)S(m_{-1})\otimes m_0\otimes n_0o_2
\end{align*}
coincide, where we used that $n_{-2}o_0S(n_{-1}o_1)\otimes(m_0\otimes n_0)o_2
\in O^{\mathrm{co}H}\otimes N$.
\end{proof}
Using the lifting $\sigma$,
for bicovariant bimodules $M,M',N,N'$ we define the {\it{tensor
    product of rational morphisms}} 
\begin{equation}\label{eq07}
    \otimes_\sigma\colon\mathrm{HOM}(M,M')\otimes\mathrm{HOM}(N,N')
    \rightarrow\mathrm{HOM}(M\otimes N,M'\otimes N')
\end{equation}
(for short we wrote $\otimes_\sigma$ rather than ${\otimes_\sigma}^{}_{M,M',N,N'}$) to be the linear map given by
\begin{equation}\label{eq36}
    (\phi\otimes_\sigma\psi)(m\otimes n)
    :=\phi({}_\alpha m)\otimes({}^\alpha\psi)(n)
    =\phi(\psi_{-2}m_0S(\psi_{-1}m_1))\otimes\psi_0(m_2n),
\end{equation}
for all $\phi\in\mathrm{HOM}(M,M')$, $\psi\in\mathrm{HOM}(N,N')$ and
$M\in M$, $n\in N$.  
\begin{theorem}\label{thm01}
For any bicovariant bimodule $M,M',N,N'$ the tensor product of
rational morphisms $\otimes_\sigma$ in \eqref{eq07}
is a morphism between the bicovariant bimodules
$\mathrm{HOM}(M,M')\otimes\mathrm{HOM}(N,N')$ and 
 $\mathrm{HOM}(M\otimes N,M'\otimes N')$.
Furthermore, it decomposes as
\begin{equation}\label{eq20}
    \phi\otimes_\sigma\psi
    =(\phi\otimes_\sigma\mathrm{id}_{N'})
    \circ(\mathrm{id}_M\otimes_\sigma\psi)~,
\end{equation}
for all $\phi\in\mathrm{HOM}(M,M')$, $\psi\in\mathrm{HOM}(N,N')$, and it is associative, i.e.,
\begin{equation*}
    (\phi\otimes_\sigma\psi)\otimes_\sigma\chi
    =\phi\otimes_\sigma(\psi\otimes_\sigma\chi)~,
\end{equation*}
where $\chi\in\mathrm{HOM}(M,M')$ is another rational morphism of bicovariant
bimodules.
\end{theorem}
\begin{proof}
For all
$\phi\in\mathrm{HOM}(M,M')$, $\psi\in\mathrm{HOM}(N,N')$
and $a,b\in H$, we show
\begin{equation*}
\begin{split}
    (\phi\otimes_\sigma\psi)_0\otimes(\phi\otimes_\sigma\psi)_1
    &=(\phi_0\otimes_\sigma\psi_0)\otimes\phi_1\psi_1~,\\
    (\phi\otimes_\sigma\psi)_{-1}\otimes(\phi\otimes_\sigma\psi)_0
    &=\phi_{-1}\psi_{-1}\otimes(\phi_0\otimes_\sigma\psi_0)~,\\
    a\cdot(\phi\otimes_\sigma\psi)\cdot b
    &=(a\cdot\phi)\otimes_\sigma(\psi\cdot b)~.
\end{split}
\end{equation*}

  Let $m\in M$, $n\in N$ and $o\in O$, then
\begin{align*}
    ((a\cdot\phi)\otimes_\sigma(\psi\cdot b))(m\otimes n)
    &=(a\cdot\phi)(\psi_{-2}b_1m_0S(\psi_{-1}b_2m_1))\otimes(\psi_0\cdot b_3)(m_2n)\\
    &=a\phi(\psi_{-2}(bm)_0S(\psi_{-1}(bm)_1))\otimes\psi_0((bm)_2n)\\
    &=a(\phi\otimes_\sigma\psi)(bm\otimes n)\\
    &=(a\cdot(\phi\otimes_\sigma\psi)\cdot b)(m\otimes n)
\end{align*}
proves that (\ref{eq07}) is $H$-bilinear. For right $H$-colinearity we verify that
\begin{align*}
    (\phi_0\otimes_\sigma\psi_0)(m\otimes n)\otimes\phi_1\psi_1
    &=\phi_0(\psi_{-2}m_0S(\psi_{-1}m_1))\otimes\psi_0(m_2n)\otimes\phi_1\psi_1\\
    &=\phi(\psi_{-2}m_0S(\psi_{-1}m_1))_0\otimes\psi_0(m_2n)
    \otimes\phi(\psi_{-2}m_0S(\psi_{-1}m_1))_1S(1)\psi_1\\
    &=\phi(\psi_{-2}m_0S(\psi_{-1}m_1))_0\otimes\psi_0(m_2n_0)_0
    \otimes\phi(\psi_{-2}m_0S(\psi_{-1}m_1))_1\psi_0(m_2n_0)_1S(m_3n_1)
\end{align*}
coincides with
\begin{align*}
    (\phi\otimes_\sigma\psi)_0(m\otimes n)\otimes(\phi\otimes_\sigma\psi)_1
    &=(\phi\otimes_\sigma\psi)(m_0\otimes n_0)_0
    \otimes(\phi\otimes_\sigma\psi)(m_0\otimes n_0)_1S(m_1n_1)\\
    &=\phi(\psi_{-2}m_0S(\psi_{-1}m_1))_0\otimes\psi(m_2n_0)_0
    \otimes\phi(\psi_{-2}m_0S(\psi_{-1}m_1))_1\psi(m_2n_0)_1S(m_3n_1)~,
\end{align*}
where we used that $\psi_{-2}m_0S(\psi_{-1}m_1)\otimes\psi_0
\in M^{\mathrm{co}H}\otimes\mathrm{HOM}(N,N')$. To check left
$H$-colinearity we first note that
$$
\lambda_M(\psi_{-2}m_0S(\psi_{-1}m_1))\otimes\psi_0
=\psi_{-4}m_{-1}S(\psi_{-1}m_2)
\otimes\psi_{-3}m_0S(\psi_{-2}m_1)\otimes\psi_0~.
$$
Using this equality we calculate
\begin{align*}
    \phi_{-1}\psi_{-1}&\otimes(\phi_0\otimes_\sigma\psi_0)(m\otimes n)
    =\phi_{-1}\psi_{-3}\otimes\phi_0(\psi_{-2}m_0S(\psi_{-1}m_1))\otimes\psi_0(m_2n)\\
    &=\phi(\psi_{-3}m_0S(\psi_{-2}m_1))_{-1}
    \overline{S}(\psi_{-4}m_{-1}S(\psi_{-1}m_2))\psi_{-5}
    \otimes\phi(\psi_{-3}m_0S(\psi_{-2}m_1))_0\otimes\psi_0(m_3n)\\
    &=\phi(\psi_{-3}m_0S(\psi_{-2}m_1))_{-1}
    \psi_{-1}m_2\overline{S}(m_{-1})
    \otimes\phi(\psi_{-3}m_0S(\psi_{-2}m_1))_0\otimes\psi_0(m_3n)\\
    &=\phi(\psi_{-2}m_0S(\psi_{-1}m_1))_{-1}
    \psi_0(m_4n_0)_{-1}\overline{S}(m_3n_{-1})m_2\overline{S}(m_{-1})
    \otimes\phi(\psi_{-2}m_0S(\psi_{-1}m_1))_0\otimes\psi_0(m_4n_0)_0\\
    &=\phi(\psi_{-2}m_0S(\psi_{-1}m_1))_{-1}
    \psi_0(m_2n_0)_{-1}\overline{S}(m_{-1}n_{-1})
    \otimes\phi(\psi_{-2}m_0S(\psi_{-1}m_1))_0\otimes\psi_0(m_2n_0)_0
\end{align*}
which turns out to coincide with
\begin{align*}
    (\phi\otimes_\sigma\psi)_{-1}&\otimes(\phi\otimes_\sigma\psi)_0(m\otimes n)
    =(\phi\otimes_\sigma\psi)(m_0\otimes n_0)_{-1}\overline{S}(m_{-1}n_{-1})
    \otimes(\phi\otimes_\sigma\psi)(m_0\otimes n_0)_0\\
    &=\phi(\psi_{-2}m_0S(\psi_{-1}m_1))_{-1}\psi_0(m_2n_0)_{-1}\overline{S}(m_{-1}n_{-1})
    \otimes\phi(\psi_{-2}m_0S(\psi_{-1}m_1))_0\otimes\psi_0(m_2n_0)_0~.
\end{align*}
Summarizing, $\otimes_\sigma$ is a morphism of bicovariant bimodules.

Equation (\ref{eq20}) is verified by the computation
\begin{align*}
    ((\phi\otimes_\sigma\mathrm{id}_{N'})
    \circ(\mathrm{id}_M\otimes_\sigma\psi))(m\otimes n)
    &=(\phi\otimes_\sigma\mathrm{id}_{N'})
    (\psi_{-2}m_0S(\psi_{-1}m_1)\otimes\psi_0(m_2n))\\
    &=\phi(\psi_{-2}m_0S(\psi_{-1}m_1))\otimes\psi_0(m_2n)\\
    &=(\phi\otimes_\sigma\psi)(m\otimes n)~,
\end{align*}
where we used again that $\psi_{-2}m_0S(\psi_{-1}m_1)\otimes\psi_0
\in M^{\mathrm{co}H}\otimes\mathrm{HOM}(N,N')$
and furthermore that $\lambda^\mathrm{Ad}(\mathrm{id}_{N'})=1_H\otimes\mathrm{id}_{N'}$.

Finally, associativity follows from the left $H$-colinearity of $\otimes_\sigma$. In fact
\begin{align*}
    ((\phi\otimes_\sigma\psi)\otimes_\sigma\chi)(m\otimes n\otimes o)
    &=(\phi\otimes_\sigma\psi)(\chi_{-2}(m_0\otimes n_0)S(\chi_{-1}m_1n_1))
    \otimes\chi_0(m_2n_2o)\\
    &=\phi(\psi_{-2}\chi_{-4}m_0S(\psi_{-1}\chi_{-3}m_1))
    \otimes\psi_0(\chi_{-2}m_2n_0S(\chi_{-1}m_3n_1))
    \otimes\chi_0(m_4n_2o)
\end{align*}
equals
\begin{align*}
    (\phi\otimes_\sigma(\psi\otimes_\sigma\chi))(m\otimes n\otimes o)
    &=\phi((\psi\otimes_\sigma\chi)_{-2}m_0S((\psi\otimes_\sigma\chi)_{-1}m_1))
    \otimes(\psi\otimes_\sigma\chi)_0(m_2(n\otimes o))\\
    &=\phi(\psi_{-2}\chi_{-2}m_0S(\psi_{-1}\chi_{-1}m_1))
    \otimes(\psi_0\otimes_\sigma\chi_0)(m_2n\otimes o)\\
    &=\phi(\psi_{-2}\chi_{-4}m_0S(\psi_{-1}\chi_{-3}m_1))
    \otimes\psi_0(\chi_{-2}m_2n_0S(\chi_{-1}m_3n_1))
    \otimes\chi_0(m_4n_2o)~.
\end{align*}
This concludes the proof of the theorem.
\end{proof}

The evaluation, composition and tensor product of rational morphisms
defined in \eqref{eq05}, \eqref{eq06} and \eqref{eq07}
descend to well defined maps when projecting the tensor product $\otimes$
to the balanced tensor product $\otimes_H$.
With slight abuse of notation we still denote by $\pi_H^{M',N'}\colon\mathrm{HOM}(M\otimes_HN,M'\otimes N')\to\mathrm{HOM}(M\otimes_HN,M'\otimes_HN')$ the projection on morphisms induced by the projection $\pi_H^{M',N'}\colon M'\otimes N'\to M'\otimes_HN'$.

\begin{theorem}\label{prop06}
 Let $M,M',N,N',O$ be modules of ${}_H^H\mathcal{M}_H^H$. Then
\begin{equation}
    \mathrm{ev}_{M,N}\colon\mathrm{HOM}(M,N)\otimes_H M\rightarrow
    N~,~~\phi\otimes_H m\mapsto \phi(m)~, 
\end{equation}
\begin{equation}\label{eq03hat}
    \circ_{M,N,O}\colon\mathrm{HOM}(N,O)\otimes_H\mathrm{HOM}(M,N)
    \rightarrow\mathrm{HOM}(M,O)~,~~(\phi\circ\psi)(m)=\phi(\psi(m))~,
\end{equation}
and
\begin{equation}\label{eq04hat}
    \,\hat{\otimes}\,:=\pi_H^{M',N'}\circ\otimes_\sigma \colon
    \mathrm{HOM}(M,M')\otimes\mathrm{HOM}(N,N')
    \rightarrow\mathrm{HOM}(M\otimes_HN,M'\otimes_HN')~,~~  \phi\otimes \psi\mapsto \phi\,\hat{\otimes}\,\psi 
\end{equation}
with
$$
~~~(\phi\,\hat{\otimes}\,\psi )(m\otimes_H n)
    =\phi({}_\alpha m)\otimes_H({}^\alpha\psi)(n)
    =\phi(\psi_{-2}m_0S(\psi_{-1}m_1))\otimes_H\psi_0(m_2n)~,
$$
are well defined morphisms in ${}_H^H\mathcal{M}^H$,
${}_H^H\mathcal{M}_H^H$ and ${}_H^H\mathcal{M}_H^H$
respectively. Furthermore, the restriction of $\,\hat{\otimes}\,$ to right
$H$-linear rational morphisms equals $\otimes_{\sigma^\mathcal{W}}$ as
defined in \eqref{eq29}:
 $\,\hat{\otimes}\,|^{}_{\mathrm{HOM}_H(M,M')\otimes\mathrm{HOM}_H(N,N')}=\otimes_{\sigma^\mathcal{W}}$.
 \end{theorem}
\begin{proof}
The linear maps $\mathrm{ev}_{M,N}$ and $\circ_{M,N,O}$ are easily
checked to be well defined. Their properties follow from those of the
evaluation and composition in  \eqref{eq06} and \eqref{eq07}.
For the tensor product $\,\hat{\otimes}\,$ we first notice that
$$\otimes_\sigma:
 \mathrm{HOM}(M,M')\otimes\mathrm{HOM}(N,N')
    \rightarrow\mathrm{HOM}(M\otimes_H N,M'\otimes N')
    \subseteq \mathrm{HOM}(M\otimes N,M'\otimes N')$$
    where
the inclusion is via composition with the projection
$\pi_H^{M,N}\colon M\otimes N\to M\otimes_H N$.
Indeed, for all $\phi\in\mathrm{HOM}(M,M')$, $\psi\in\mathrm{HOM}(N,N')$, 
$\phi\otimes_\sigma\psi$ is well defined on the quotient $M\otimes_HN$,
since
\begin{align*}
    (\phi\otimes_\sigma\psi)(ma\otimes n)
    &=\phi(\psi_{-2}m_0a_1S(\psi_{-1}m_1a_2))\otimes\psi_0(m_2a_3n)\\
    &=\phi(\psi_{-2}m_0S(\psi_{-1}m_1))\otimes\psi_0(m_2an)\\
    &=(\phi\otimes_\sigma\psi)(m\otimes an)
\end{align*}
for all $m\in M$, $n\in N$ and $a\in H$.

Since the
projection $\pi^{M',N'}_H\colon M\otimes N\rightarrow M\otimes_HN$
is a morphism in
${}_H^H\mathcal{M}_H^H$ and $\otimes_\sigma$ is a morphism in
${}_H^H\mathcal{M}_H^H$ by Theorem~\ref{thm01}, it follows that
$\,\hat{\otimes}\,=\pi^{M',N'}_H\circ\otimes_\sigma$ is a bicovariant bimodule
morphism as well.

Let $\phi\in\mathrm{HOM}_H(M,M')$, $\psi\in\mathrm{HOM}_H(N,N')$ be
right $H$-linear rational morphisms.
The equalities
\begin{align*}
    (\phi\,\hat{\otimes}\,\psi)(m\otimes_Hn)
    &=\phi(\psi_{-2}m_0S(\psi_{-1}m_1))\otimes_H\psi_0(m_2n)\\
    &=\phi(\psi_{-2}m_0S(m_1))\otimes_HS(\psi_{-1})\psi_0(m_2n)\\
    &=(\phi\otimes_{\sigma^\mathcal{W}}\psi)(m\otimes_Hn)
\end{align*}
prove that
$\,\hat{\otimes}\,|^{}_{\mathrm{HOM}_H(M,M')\otimes\mathrm{HOM}_H(N,N')}=\otimes_{\sigma^\mathcal{W}}$.
\end{proof}
The equality $\,\hat{\otimes}\,|^{}_{\mathrm{HOM}_H(M,M')\otimes\mathrm{HOM}_H(N,N')}=\otimes_{\sigma^\mathcal{W}}$ shows that the tensor product
$\,\hat{\otimes}\,$ extends to rational morphisms the tensor product
$\otimes_{\sigma^\mathcal{W}}$ of right $H$-linear rational morphisms. Associativity
of the composition in \eqref{eq06} and of the tensor product
$\otimes_\sigma$ in \eqref{eq07} implies associativity of the induced  composition \eqref{eq03hat} and of the
tensor product $\,\hat{\otimes}\,$.

\begin{remark} We list the different tensor products
  encountered, the first two into modules, the last three into modules
  of rational morphisms: i) $\otimes$ over the base field $\Bbbk$, ii)
  $\otimes_H$ over the Hopf algebra $H$, iii)
  $\otimes_{\sigma^\mathcal{W}}$ of right $H$-linear rational
  morphisms, iv) $\otimes_{\sigma}$ of rational morphisms
 and v) its projection $\,\hat{\otimes}\,$. In particular, $\phi \,\hat{\otimes}\,
 \mathrm{id}_N=\phi\otimes_{\sigma^\mathcal{W}}\mathrm{id}_N$ for
 $\phi\in \mathrm{HOM}_H(M,M')$ a right $H$-linear map,  explicitly, on elements $m\in M, n\in N$,
$(\phi \,\hat{\otimes}\,
 \mathrm{id}_N)(m\otimes_H n)=(\phi\otimes_{\sigma^\mathcal{W}}\mathrm{id}_N)(m\otimes_H
 n)=\phi(m) \otimes_H n$.
  \end{remark}

  \section{Bicovariant Bimodules and Differential Calculi}\label{Sec3}
  
We review the construction of bicovariant differential calculi on
quantum groups \cite{Woronowicz1989} and study in detail braided
symmetric tensors and their duals, since these are relevant for pseudo-Riemannian
structures.
We follow the convention of \cite{AsSchupp96}, where vector fields are paired with
differential forms from the left, rather than from the right as in \cite{Woronowicz1989} (this explains the inverse braiding defining the braided exterior algebra). 
See \cite{AsSchupp96} for the relation between these two conventions.
A comprehensive study of bicovariant differential calculi, together with numerous
examples and applications, can be found in the textbooks
\cite{KlimykSchmudgen, MajidBeggsBook}.

\subsection{Braided Exterior Algebra and Duality}\label{Sec3.1}

Fix a bicovariant bimodule $M$. For any positive integer $k>0$ we
define
\begin{equation*}
    M^k:=\underbrace{M\otimes^{}_H\ldots\otimes^{}_HM}_{k-\text{times}}.
\end{equation*}
This tensor product of bicovariant bimodules  is again a bicovariant
bimodule with respect to the $H$-(co)actions
\begin{equation*}
\begin{split}
    a\cdot m\cdot b
    &=(am^1)\otimes_Hm^2\ldots\otimes_Hm^{k-1}\otimes_H(m^kb)~,\\
    \delta_{M^k}(m)
    &=(m^1_0\otimes_H\ldots\otimes_Hm^k_0)\otimes m^1_1\ldots m^k_1~,\\
    \lambda_{M^k}(m)
    &=m^1_{-1}\ldots m^k_{-1}\otimes(m^1_0\otimes_H\ldots\otimes_Hm^k_0)
\end{split}
\end{equation*}
where $m=m^1\otimes_H\ldots\otimes_Hm^k\in M^k$,
 $a,b\in H$, cf. Proposition~\ref{prop07}.
We further set $M^0:=H$ which is a bicovariant
bimodule with actions and coactions given by the multiplication and the coproduct.
Thus, the \textit{tensor algebra} $(M^\bullet,\otimes_H,1_H)$, where
\begin{equation*}
    M^\bullet
    :=\bigoplus\nolimits_{k\geq 0}M^k
    =H\oplus M\oplus(M\otimes_HM)\oplus\cdots~,
\end{equation*}
is a graded $H$-bicomodule algebra, 
i.e., the $H$-actions and coaction commute, are compatible with the tensor product $\otimes_H$ and respect the natural grading of $M^\bullet$. In particular, $M^\bullet$ is a graded bicovariant bimodule.

Following  \cite{Woronowicz1989}, the braiding $\sigma^\mathcal{W}$ of the braided monoidal category of
bicovariant bimodules (cf. Proposition~\ref{prop07}) leads to a canonical quotient $\Lambda^\bullet M$
of $M^\bullet$, the \textit{braided exterior algebra}.

Let $n\in\mathbb{N}$ be a natural number $n>0$. We denote the nearest neighbour
transpositions of the set $\{1,\ldots,n\}$ by $\{t_1,\ldots,t_{n-1}\}$, meaning
that $t_i$, $1\leq i<n$, fixes the elements of $\{1,\ldots,n\}$ but flips $i$ and
$i+1$. Every permutation $p\in P(n)$ of $\{1,\ldots,n\}$ can be written as
a composition 
of nearest neighbour transpositions: $p=t_{k_1}\cdots t_{k_N}$, with
$k_1,\ldots,k_N\in\{1,\ldots,n-1\}$ for some $N\in\mathbb{N}$ (if
$N=0$, $p$ is the identity). For a fixed
$p\in P(n)$ its length $I(p)\in\mathbb{N}$ is the (well defined)
minimal number of nearest neighbour transpositions such that
\begin{equation}\label{permutation}
    p=t_{k_1}\cdots t_{k_{I(p)}}~.
\end{equation}
For $k\in\{1,\ldots,n-1\}$ we
define 
$$
\overline{\sigma}_k:=\mathrm{id}_{M^{k-1}}\otimes_H\overline{\sigma}^\mathcal{W}_{M,M}\otimes_H
\mathrm{id}_{M^{n-k-1}}\colon M^n
\rightarrow M^n
$$
and for every permutation $p\in P(n)$ with decomposition (\ref{permutation}) we set
$\Pi_p:=\overline{\sigma}_{k_1}\cdots\overline{\sigma}_{k_{I(p)}}\colon M^n
\rightarrow M^n$.
Because of the quantum Yang-Baxter equation (\ref{QYBE}),
$\Pi_p$ is independent of the choice of the decomposition (\ref{permutation}).
Define now the \textit{braided anti-symmetrization} maps
\begin{equation*}
    A_n:=\sum_{p\in P(n)}\mathrm{sign}(p)\Pi_p\colon M^n
    \rightarrow M^n~,~~~
    \text{ for }n\geq 1~,
\end{equation*}
$A_0={\rm id}_{M^0}={\rm id}_H$,  where $\mathrm{sign}(p)=(-1)^{I(p)}$ denotes the sign
of the permutation $p\in P(n)$. It immediately follows that $A_n$ is a
bicovariant bimodule morphism, since so are
${\rm{id}}_{M^n}$ and $\overline{\sigma}_k$ for any $k$.
We further define $S^n:=\ker A_n\subseteq M^n$, and
$S^\bullet:=\bigoplus_{n\geq 0}S^n\subseteq M^\bullet$.
The vector space $S^\bullet$ is a graded $H$-bicomodule subalgebra of
$M^\bullet$
and can be shown to be a graded algebra ideal of the tensor algebra 
$(M^\bullet,\otimes_H,1_H)$.
Then, the quotient 
\begin{equation}\label{QA}
  M^\bullet/S^\bullet=  \bigoplus_{n\geq 0}M^n/S^n~.
\end{equation}
is a graded $H$-bicomodule algebra.
Using the bicovariant bimodule  isomorphisms
$M^n/S^n\cong{\rm{Im}}A_n$ we define the \textit{braided
exterior algebra} as the graded bicovariant bimodule 
\begin{equation}\label{BEA1}
    \Lambda^\bullet M:=
   \bigoplus_{n\geq 0}{\rm{Im}}A_n~,
   \end{equation}
with multiplication the \textit{wedge product} $\wedge$ which is
induced on $\Lambda^\bullet M$ from the quotient multiplication on
$M^\bullet/S^\bullet$. The linear span of all wedge
products with $n$ elements in $M$ is denoted 
$M^{\wedge n}:=\underbrace{M\wedge\ldots\wedge
  M}_{n-\text{times}}$ and we have
\begin{equation}
M^{\wedge 0}=H~~,~~~M^{\wedge 1}=M~~,~~~
M^{\wedge n}={\rm{Im}}A_n~~,~~~\Lambda^\bullet M:=
  \bigoplus_{n\geq 0}M^{\wedge n}~.\label{BEA}
  \end{equation}
We have given this definition of braided exterior algebra, rather than the
equivalent one as the quotient $
M^\bullet/S^\bullet$, so that $M^{\wedge n}$ is a
bicovariant subbimodule of $M^n$. We shall later study
$M^2$ as a direct sum of braided symmetric tensors and
the braided antisymmetric tensors $M^{\wedge 2}$. In summary, we have
\\
\begin{proposition}\label{prop111}
The braided exterior algebra $(\Lambda^\bullet M,\wedge,1_H)$ is an associative graded $H$-bicomodule algebra, and, in particular, a graded bicovariant bimodule.
\end{proposition}
The $H$-actions and coactions on $\Lambda^\bullet M$ are induced from the ones on $M^\bullet$, explicitly
\begin{equation*}
\begin{split}
    a\cdot(m^1\wedge\ldots\wedge m^k)\cdot b
    &=(a\cdot m^1)\wedge\ldots\wedge(m^k\cdot b)~,\\
    (m^1\wedge\ldots\wedge m^k)_0\otimes
    (m^1\wedge\ldots\wedge m^k)_1
    &=(m^1_0\wedge\ldots\wedge m^k_0)\otimes
    (m^1_1\cdots m^k_1)~,\\
    (m^1\wedge\ldots\wedge m^k)_{-1}\otimes
    (m^1\wedge\ldots\wedge m^k)_0
    &=(m^1_{-1}\cdots m^k_{-1})
    \otimes(m^1_0\wedge\ldots\wedge m^k_0)~,
\end{split}
\end{equation*}
for all $a,b\in H$ and $m^1,\ldots,m^k\in M$.
\\\medskip

For any bicovariant bimodule $M$ which is finitely generated as a
right  $H$-module,
its dual $M^*:=\mathrm{Hom}_H(M,H)$
is a bicovariant bimodule according to Propositions~\ref{prop07} and \ref{prop100}.
Explicitly, actions and coactions read $(a\cdot X\cdot b)(m)=aX(bm)$ and
\begin{equation*}
    X_0(m)\otimes X_1=X(m_0)_0\otimes X(m_0)_1S(m_1)
    ~~~\text{ ~,~~ }~~~
    X_{-1}\otimes X_0(m)=X(m_0)_{-1}\overline{S}(m_{-1})\otimes X(m_0)_0
\end{equation*}
for all $a,b\in H$, $X\in M^*$ and $m\in M$. Furthermore, according to
Corollary~\ref{cor01}, the evaluation is a morphism of bicovariant bimodules.
In the context of dual modules we denote the evaluation by
\begin{equation}\label{eq54}
    \langle\cdot,\cdot\rangle\colon
    M^*\otimes_HM\rightarrow H
\end{equation}
and call it pairing of $M$ and $M^*$. Since
$M$ is the free and  finitely generated  right $H$-module  $M\cong M^{\mathrm{co}H}\otimes
H$, $M^*$ is the free and finitely generated left $H$-module $M^*\cong
H\otimes (M^{\mathrm{co}H})^*$. Nondegeneracy of the vector space pairing  $
\langle\cdot,\cdot\rangle\colon    (M^{\mathrm{co}H})^*\otimes M^{\mathrm{co}H}\rightarrow \Bbbk$
implies nondegeneracy of the pairing in \eqref{eq54}, i.e.  $\langle X,m\rangle=0$ for all $m\in M$ implies
$X=0$ and $\langle X,m\rangle=0$ for all $X\in M^*$ implies
$m=0$.   From the fundamental theorem for the bicovariant bimodule $M^*$
we also have $M^*\cong H\otimes  (M^*)^{\mathrm{co}H}$, comparing with  $M^*\cong
H\otimes (M^{\mathrm{co}H})^*$ we obtain  $(M^*)^{\mathrm{co}H}\cong  (M^{\mathrm{co}H})^*$.
For any $k>1$ we extend the pairing in \eqref{eq54} to  the non-degenerate pairing of bicovariant bimodules 
\begin{equation}\label{onion}
    \langle\cdot,\cdot\rangle
    \colon (M^*)^k\otimes_HM^k\rightarrow H~,
\end{equation}
via the \textit{onion like structure} 
\begin{equation*}
    \langle X_1\otimes_H\ldots\otimes_HX_k,
    \omega_1\otimes_H\ldots\otimes_H\omega_k\rangle
    :=\langle X_1,\ldots\langle X_{k-1},
    \langle X_k,\omega_1\rangle\omega_2\rangle\ldots\omega_k\rangle~,
  \end{equation*}
for any 
$X_1\otimes_H\ldots\otimes_HX_k\in(M^*)^k$ and
$\omega_1\otimes_H\ldots\otimes_H\omega_k\in M^k$.
Nondegeneracy can be easily proven using a basis of $M^{\mathrm{co}H}$ and of
$(M^*)^{\mathrm{co}H}\cong  (M^{\mathrm{co}H})^*$. Uniqueness up to isomorphism of the
dual implies the bicovariant bimodule isomorphism $M^*\otimes_H\ldots \otimes_H
M^*\cong(M\otimes_H \ldots \otimes_H M)^*$.
(For a more general discussion see for example
\cite{AschieriCartanStructure}~Theorem~2.5.3).
For every morphism $\Phi\colon M^k\rightarrow M^k$
of bicovariant bimodules the  pairing induces a 
\textit{transpose morphism} $\Phi^*\colon(M^*)^{k}\rightarrow(M^*)^{k}$
via
\begin{equation*}
    \langle\Phi^*(X),m\rangle
    =\langle X,\Phi(m)\rangle
\end{equation*}
for all $X\in(M^*)^k$ and $m\in M^k$. It follows
that $\Phi^*$ is a morphism of bicovariant bimodules, as well.
\begin{lemma}\label{lemma03}
The transpose of $\overline{\sigma}^\mathcal{W}_{M,M}\colon M^2\rightarrow
M^2$
is $(\overline{\sigma}_{M,M}^\mathcal{W})^*=\overline{\sigma}^\mathcal{W}_{M^*,M^*}\colon
(M^*)^2\rightarrow(M^*)^2$.
\end{lemma}
\begin{proof}
By Theorem~\ref{FundThm} and the bilinearity of $\overline{\sigma}^\mathcal{W}$ it is
sufficient to verify the claim on coinvariant elements.
Let $\{m_R^i\}_{i\in I}$ and $\{m_L^i\}_{i\in I}$ be a basis of right and left
coinvariant elements of $M$, respectively and consider the dual bases $\{\alpha^R_i\}_{i\in I}$
and $\{\alpha^L_i\}_{i\in I}$. Then
\begin{align*}
    \langle(\overline{\sigma}^\mathcal{W}_{M,M})^*(\alpha^R_i\otimes_H\alpha^L_j),
    m_R^p\otimes_Hm_L^q\rangle
    &=\langle\alpha^R_i\otimes_H\alpha^L_j,
    \overline{\sigma}^\mathcal{W}_{M,M}(m_R^p\otimes_Hm_L^q)\rangle\\
    &=\langle\alpha^R_i\otimes_H\alpha^L_j,
    m_L^q\otimes_Hm_R^p\rangle\\
    &=\delta_i^p\delta_j^q\\
    &=\langle\alpha^L_j\otimes_H\alpha^R_i,
    m_R^p\otimes_Hm_L^q\rangle\\
    &=\langle\overline{\sigma}^\mathcal{W}_{M^*,M^*}(\alpha^R_i\otimes_H\alpha^L_j),
    m_R^p\otimes_Hm_L^q\rangle
\end{align*}
implies
$(\overline{\sigma}_{M,M}^\mathcal{W})^*=\overline{\sigma}^\mathcal{W}_{M^*,M^*}$
by $H$-bilinearity and 
non-degeneracy of the pairing.
\end{proof}

We now assume that $\overline{\sigma}^\mathcal{W}_{M,M}\colon
M\otimes_HM\rightarrow M\otimes_HM$ is diagonalisable and that has
$1$ among its eigenvalues.
This is for example the case for all
quantum groups of the $A$, $B$, $C$, $D$ series, with $M$ the
bicovariant bimodule of $1$-forms, cf. \cite[Sec. 3]{Jurco},
\cite[Prop. 3.1]{BhowmickCov}.
 We denote by $V_\lambda\subseteq M$  the eigenspace of
$\overline{\sigma}^\mathcal{W}_{M,M}$ with eigenvalue $\lambda\in
\Lambda$.  The
set $\Lambda$ of eigenvalues is finite since $M^2$ is a free and finitely
generated right $H$-module and $\overline{\sigma}^\mathcal{W}_{M,M}$
is in particular right $H$-linear. We have
\begin{equation}\label{eq81}
    M^2=\bigoplus\nolimits_{\lambda\in\Lambda}V_\lambda
    ~~~\text{ and }~~~    \overline{\sigma}^\mathcal{W}_{M,M}=\sum\nolimits_{\lambda\in\Lambda}\lambda P_\lambda~,
\end{equation}
where $P_\lambda\colon M^2\rightarrow M^2$, $P_\lambda^2=P_\lambda$
is the projector to the eigenspace $V_\lambda$ corresponding to
the eigenvalue $\lambda$.
The explicit expression of the projectors is 
\begin{equation}\label{Pnu}
  P_\nu
    =\prod_{\lambda\in\Lambda\setminus\{\nu\}}\frac{\overline{\sigma}^\mathcal{W}_{M,M}
    -\lambda\mathbb{1}}{\nu-\lambda}~,\qquad\nu\in\Lambda~,
\end{equation}
as is easily seen evaluating them on vectors $\omega_\rho\in V_\rho$
for the two cases $\rho=\nu$ and $\rho\not=\nu$. They are bicovariant
bimodules morphisms because compositions of bicovariant bimodule
morphisms.  
The eigenspaces $V_\lambda=P_\lambda(M^2)$ are bicovariant
bimodules since images of bicovariant bimodule morphisms, therefore $
M^2=\bigoplus\nolimits_{\lambda\in\Lambda}V_\lambda$ is a direct
sum of bicovariant bimodules.

We define the vector subspaces
$$
V^*_\lambda
:=\{X\in(M^*)^2~|~\langle X,V_\rho\rangle=0\text{ for all }\rho\neq\lambda\}
$$
and
$$
V_\lambda^\perp
:=\{X\in (M^*)^2~|~ \langle X, V_\lambda\rangle=0\}
$$
for all eigenvalues $\lambda\in\Lambda$. The vector subspace
$V^*_\lambda\subseteq(M^*)^2$ is a subbicovariant bimodule, indeed,
since
 $V_\lambda$ is a  bicovariant bimodule
and the pairing is a morphism of bicovariant bimodules,
$V^*_\lambda$  is closed  under the $H$-actions and the $H$-coactions.
Similarly $V_\lambda^\perp\subseteq(M^*)^2$ is a subbicovariant bimodule.
\begin{lemma}\label{lemma15}
Let $\overline{\sigma}^\mathcal{W}_{M,M}$ be diagonalisable with eigenvalues
$\Lambda$, eigenspaces $V_\lambda$ and projectors $P_\lambda$
as in (\ref{eq81}).
Then the following statements hold:
\begin{enumerate}
\item[i.)]
$(M^*)^2=\bigoplus\nolimits_{\lambda\in\Lambda}V^*_\lambda$
as bicovariant bimodules.

\item[ii.)]  The transpose of the braiding,
$\overline{\sigma}^\mathcal{W}_{M^*,M^*}\colon (M^*)^2\to (M^*)^2$,
has the projector decomposition
$$\overline{\sigma}^\mathcal{W}_{M^*,M^*}=\sum\nolimits_{\lambda\in\Lambda}\lambda
P^*_\lambda$$
where the projectors $P^*_\lambda:  (M^*)^{2} \to
(M^*)^{2} $ to the
eigenspaces  $V^*_\lambda$ are the transpose of the projectors
$P_\lambda$, i.e., $P^*_\lambda=(P_\lambda)^*$.
\item[iii.)] For any $\lambda\in\Lambda$ there is an isomorphism
$V^*_\lambda\cong(V_\lambda)^*:=\mathrm{Hom}_H(V_\lambda,H)$
of bicovariant bimodules.

\item[iv.)] For any $\lambda\in\Lambda$ there is an isomorphism
$V^*_\lambda\cong(M^*)^2/V_\lambda^\perp$ of bicovariant bimodules.

\end{enumerate}
\end{lemma}
\begin{proof}
\begin{enumerate}
\item[i.)]  $V^*_\lambda\cap V^*_\rho=\{0\}$ if $\lambda\neq\rho$,
by non-degeneracy of the evaluation. Therefore
$\bigoplus\nolimits_{\lambda\in\Lambda}V^*_\lambda\subseteq(M^*)^2$. To
show the other inclusion let $X\in(M^*)^2$ and
$\omega=\sum\nolimits_{\lambda\in\Lambda}\omega_\lambda$ be arbitrary,
where $\omega_\lambda\in V_\lambda$. Define $X_\lambda\in V^*_\lambda$ by
$X_\lambda(\omega_\lambda)=X(\omega_\lambda)$ and $X_\lambda(\omega_\rho)=0$
for $\rho\neq\lambda$. This shows that $X$ decomposes into a direct sum
$X=\sum\nolimits_{\lambda\in\Lambda}X_\lambda$.

\item[ii.)]  The transpose $(P_\nu)^*:(M^*)^2\to
  (M^*)^2$ of the projector $P_\nu$ satisfies, for any
  $X_\nu\in V^*_\nu$ and $\omega_\rho\in V_\rho$ with $\rho\not=\nu$, 
  $\langle (P_\nu)^* X_\nu,\omega_\rho\rangle=\langle X_\nu,P_\nu \omega_\rho\rangle=0$. Moreover,
  for any
   $X_\nu\in V^*_\nu$ and   $\omega_\nu\in V_\nu$,
  $\langle (P_\nu)^* X_\nu,v_\nu\rangle=\langle X_\nu,P_\nu
  \omega_\nu\rangle=\langle X_\nu, \omega_\nu\rangle$ shows that $(P_\nu)^*$
  is the identity on $V_\nu^*$. Hence  $(P_\nu)^*$ equals the
  projector $P_\nu^*\colon (M^*)^2\to
  (M^*)^2$ to the subspace  $V_\nu^*$.
  From Lemma~\ref{lemma03} we have 
$$
\overline{\sigma}^\mathcal{W}_{M^*,M^*}
=(\overline{\sigma}^\mathcal{W}_{M,M})^*
=\big(\sum\nolimits_{\lambda\in\Lambda}\lambda P_\lambda\big)^*
=\sum\nolimits_{\lambda\in\Lambda}\lambda P^*_\lambda
$$
as claimed.
\item[iii.)] Every element of $(M^*)^2\cong (M^2)^*$ 
can be restricted to an $H$-linear map $V_\lambda\rightarrow H$, so in particular $V^*_\lambda
\hookrightarrow(V_\lambda)^*$, $\omega_\lambda\mapsto
\langle\omega_\lambda , \cdot\rangle$ is an injection of  bicovariant
bimodules. On the other hand, every
$X\in(V_\lambda)^*$ can be extended to $(M^2)^*\cong (M^*)^2$ by setting it
equal zero on all $V_\rho$ for $\rho\neq\lambda$. By definition this extension is
in $V^*_\lambda\subseteq  (M^*)^2$ and inverse to the
previous restriction, proving that the bicovariant bimodule injection
$V^*_\lambda
\hookrightarrow(V_\lambda)^*$, $\omega_\lambda\mapsto
\langle\omega_\lambda , \cdot\rangle$ is an isomorphism.

\item[iv.)] The quotient $(M^2)^*/
V_\lambda^\perp$ is a
bicovariant bimodule, since it is a quotient of bicovariant bimodules. We prove that there is a non-degenerate pairing
\begin{equation}\label{pairinglambda}
    \langle\cdot,\cdot\rangle_\lambda\colon(M^*)^2/V_\lambda^\perp
    \otimes_HV_\lambda\rightarrow H~,~~~~~
    [X]\otimes_H\omega\mapsto\langle[X],\omega\rangle_\lambda
    :=\langle X , \omega\rangle
\end{equation}
of bicovariant bimodules.
First of all, the $H$-bimodule map in (\ref{pairinglambda}) is well-defined since, if
$[X]=0$ any representative $X$ is in $V_\lambda^\perp$ and so
$\langle X,V_\lambda\rangle=0$. The pairing is a morphism of bicovariant
bimodules since it is induced from the bicovariant bimodules pairing
$\langle\cdot,\cdot\rangle$ via a quotient of bicovariant
bimodules. Nondegeneracy follows from that of
$\langle\cdot,\cdot\rangle$.
This nondegenerate pairing implies that 
$(M^2)^*/V_\lambda^\perp\cong(V_\lambda)^*$ and thus the quotient
is isomorphic to $V_\lambda^*$ by $iii.)$.
\end{enumerate}
\end{proof}

Recalling that the braided anti-symmetrization on $M^2$ is
$A_2={\rm id}_{M^2}-\overline{\sigma}^\mathcal{W}_{M,M}$ we
see that the eigenspace with eigenvalue $\lambda=1$ is $V_1=S^2={\rm ker}A_2$, then  ${\rm Im}A_2
 =\bigoplus\nolimits_{\lambda\in\Lambda\setminus\{1\}}V_\lambda$. It
 is convenient to introduce the  notation
\begin{equation*}
  M^{\vee2}:=M\vee M:={\rm ker}A_2=S^2=V_1~.
\end{equation*}
We call $M^{\vee2}$ the \textit{braided symmetric tensors} of
$M^2$, and  $M^{\wedge 2}=M\wedge M={\rm
  Im}A_2$ the \textit{braided antisymmetric tensors}, cf. \eqref{BEA}.
From \eqref{eq81} and \eqref{Pnu} we have in particular the direct sum of bicovariant bimodules
\begin{equation}
    M^2=M^{\vee2}\oplus M^{\wedge2}
\end{equation}
 with the bicovariant bimodule morphisms
$$
P_\vee:=P_1=\prod_{\lambda\in\Lambda\setminus\{1\}}\frac{\overline{\sigma}^\mathcal{W}_{M,M}
    -\lambda\mathbb{1}}{1-\lambda}\colon M^2\rightarrow M^2
\quad\text{ and }\quad
P_\wedge=\mathrm{id}_{M^2}-P_\vee\colon M^2\rightarrow M^2~,
$$ 
that are the
projectors to $M^{\vee 2}$ and $M^{\wedge 2}$, respectively.

Dually, recalling Lemma \eqref{lemma15} ii.) we set
\begin{equation*}
  (M^*)^{\vee2}:=  M^*\vee M^*:={\rm ker} A_2^*=V_1^*~,~~
    (M^*)^{\wedge 2}:=M^*\wedge M^*:={\rm
    Im}A^*_2=\bigoplus\nolimits_{\lambda\in\Lambda\setminus\{1\}}V^*_\lambda
\end{equation*}
and consider the direct sum decomposition
$(M^*)^2=(M^*)^{\vee2}\oplus (M^*)^{\wedge 2}$.
Choosing $\lambda=1$ in \eqref{pairinglambda} we obtain the
nondegenerate pairing
$
    \langle\cdot,\cdot\rangle_1
    \colon(M^*)^2/V_1^\perp\otimes_HM^{\vee2}
    \rightarrow H
$ that composed with the isomorphism $V_1^*\cong
(M^*)^2/V_1^\perp$, cf. Lemma \eqref{lemma15} iv.), gives
the nondegenerate pairing of braided symmetric tensors
\begin{equation}\label{eq80}
    \langle\cdot,\cdot\rangle_\vee
    \colon(M^*)^{\vee2}\otimes_HM^{\vee2}
    \rightarrow H~.
\end{equation}
\subsection{Bicovariant Differential Calculi and Structure Constants}\label{Sec3.2}

Recall that a differential graded algebra $(\Omega^\bullet,\wedge,\mathrm{d})$
is a ($\mathbb{N}_0$-)graded algebra
$\Omega^\bullet=\bigoplus\nolimits_{k\in\mathbb{N}_0}\Omega^k$
with graded product 
$\wedge\colon\Omega^k\otimes\Omega^\ell\rightarrow\Omega^{k+\ell}$,
$k,\ell\in\mathbb{N}_0$ and a graded map
$\mathrm{d}\colon\Omega^\bullet\rightarrow\Omega^{\bullet+1}$ of degree $+1$
(the \textit{differential}), which
squares to zero ($\mathrm{d}^2=0$) and satisfies the
\textit{graded Leibniz rule}
$$
\mathrm{d}(\omega\wedge\eta)=\mathrm{d}\omega\wedge\eta
+(-1)^{|\omega|}\omega\wedge\mathrm{d}\eta
$$
for all homogeneous elements $\omega,\eta\in\Omega^\bullet$, 
where $|\omega|\in\mathbb{N}_0$ denotes the degree of $\omega$ \footnote{
The graded algebra $\Omega^\bullet$ is not necessarily 
the braided exterior algebra
$\Lambda^\bullet M$ of a bicovariant bimodule $M$ defined in
\eqref{BEA}. Furthermore, the notation $\Omega^\bullet$ and $\Omega^k$ should not
be confused with the tensor product algebra $M^\bullet$, $M^k$.}.
A \textit{differential calculus} on an algebra $A$ is a differential graded
algebra $(\Omega^\bullet,\wedge,\mathrm{d})$ such that $\Omega^0=A$ and 
$\Omega^k=\mathrm{span}_\Bbbk\{a^0\mathrm{d}a^1\wedge\ldots\wedge\mathrm{d}a^k~|~
a^0,\ldots,a^k\in A\}$ for all $k>0$.
Note that the graded product is by assumption associative
but not graded commutative in general.
The data $(\Gamma:=\Omega^1,\mathrm{d}\colon A\rightarrow\Gamma)$ constitutes
a so-called \textit{first order differential calculus} (for short FODC) on $A$, i.e.
an $A$-bimodule $\Gamma$ together with a linear map
$\mathrm{d}\colon A\rightarrow\Gamma$ such that the \textit{Leibniz rule}
$\mathrm{d}(ab)=(\mathrm{d}a)b+a\mathrm{d}b$ holds for all $a,b\in A$ and
the \textit{surjectivity condition} 
$A=\mathrm{span}_\Bbbk\{a\mathrm{d}b~|~a,b\in A\}$ is satisfied.
The prototypical example of a differential calculus is the differential
graded algebra of differential forms $(\Omega^\bullet(M),\wedge,\mathrm{d})$
on a smooth manifold $M$, with $\wedge$ the wedge product and $\mathrm{d}$ the
de Rham differential. The corresponding FODC is given by the smooth sections
$\Gamma=\Gamma^\infty(T^*M)$ of the cotangent bundle and the differential in
zero degree. 
For a Lie group $G$ we can use the pullback of the left and right multiplication on $G$
to induce additional $\mathscr{C}^\infty(G)$-coactions on $\Omega^1(G)$.
This example has the following noncommutative generalization. As before, we denote
by $H$ a Hopf algebra (with invertible antipode).
\begin{definition}\label{FODC}
A differential calculus $(\Omega^\bullet,\wedge,\mathrm{d})$ on $H$
is said to be bicovariant if for all $k>0$
and $a^{i,0},\ldots,a^{i,k}\in H$ with
$\sum_ia^{i,0}\mathrm{d}a^{i,1}\wedge\ldots\wedge\mathrm{d}a^{i,k}=0$
(finite sum over $i$) we have
\begin{equation*}
    \sum\nolimits_i\Delta(a^{i,0})(\mathrm{d}\otimes\mathrm{id})(\Delta(a^{i,1}))\ldots(\mathrm{d}\otimes\mathrm{id})(\Delta(a^{i,k}))
    =0
    =\sum\nolimits_i\Delta(a^{i,0})(\mathrm{id}\otimes\mathrm{d})(\Delta(a^{i,1}))\ldots(\mathrm{id}\otimes\mathrm{d})(\Delta(a^{i,k}))~.
\end{equation*}
It is said to be right covariant if it satisfies the first equation and
left covariant if it satisfies the second equation.
A FODC is said to be bicovariant (right, left covariant) if the above conditions are satisfied in first degree.
\end{definition}
The rationale of Definition~\ref{FODC} is the following.
To any bicovariant differential calculus $(\Omega^\bullet,\wedge,\mathrm{d})$ on $H$
we associate two linear maps
$\delta_\Omega\colon\Omega^\bullet\rightarrow\Omega^\bullet\otimes H$ and
$\lambda_\Omega\colon\Omega^\bullet\rightarrow H\otimes\Omega^\bullet$
by
\begin{equation}\label{eq08}
\begin{split}
    \delta_\Omega(\omega)
    &=\sum\nolimits_i\Delta(a^{i,0})(\mathrm{d}\otimes\mathrm{id})(\Delta(a^{i,1}))\ldots(\mathrm{d}\otimes\mathrm{id})(\Delta(a^{i,k}))~,\\
    \lambda_\Omega(\omega)
    &=\sum\nolimits_i\Delta(a^{i,0})(\mathrm{id}\otimes\mathrm{d})(\Delta(a^{i,1}))\ldots(\mathrm{id}\otimes\mathrm{d})(\Delta(a^{i,k}))
\end{split}
\end{equation}
for all $k>0$ and
$\omega=\sum_ia^{i,0}\mathrm{d}a^{i,1}\wedge\ldots\wedge\mathrm{d}a^{i,k}
\in\Omega^k$, where $a^{i,0},\ldots,a^{i,k}\in H$.
Note that $\delta_\Omega$ and $\lambda_\Omega$ are well-defined through the
right and left covariance property of Definition~\ref{FODC}, respectively.
\begin{proposition}\label{prop02}
Let $(\Omega^\bullet,\wedge,\mathrm{d})$ be a differential calculus on $H$.
It is bicovariant if and only if $\Omega^\bullet$ is a graded $H$-bicomodule algebra
and $\mathrm{d}$ is $H$-bicolinear. In this case
the graded $H$-coactions are determined by (\ref{eq08}).
\end{proposition}
See e.g. \cite[Sect. 12.2.3]{KlimykSchmudgen} for a proof of 
Proposition~\ref{prop02}.
As a corollary it follows that a FODC 
$(\Gamma,\mathrm{d})$ on $H$ is bicovariant if and only if $\Gamma$ is a bicovariant
bimodule and $\mathrm{d}\colon H\rightarrow\Gamma$ is an $H$-bicolinear map.
Given a bicovariant FODC $(\Gamma,\mathrm{d})$ on $H$ we can consider
the braided exterior algebra $(\Lambda^\bullet\Gamma,\wedge,1_H)$ of 
Proposition~\ref{prop111}.

\begin{proposition}[\cite{Woronowicz1989}~Thm.~4.1]\label{propBEA}
Every bicovariant FODC $(\Gamma,\mathrm{d})$ on $H$ has a unique extension to a
bicovariant differential calculus $(\Lambda^\bullet\Gamma,\wedge,\mathrm{d})$.
\end{proposition}
A bicovariant differential calculus $(\Omega^\bullet,\wedge,\mathrm{d})$ will be called \textit{Woronowicz differential calculus} if it equals the extension of of the bicovariant FODC $(\Gamma,\mathrm{d})$ as in the above proposition. This will be the main examples of differential calculi in this paper.

We next recall the Cartan-Maurer formula associated with $(\Lambda^\bullet\Gamma,\wedge,\mathrm{d})$ \cite[Sect. 4]{PaoloLeonardo} (we use the conventions of vector fields acting from the left, cf. \cite{AsSchupp96}). It will be used to define torsion-free connections in the next section.
By the Fundamental Theorem of Hopf Modules 
(see Theorem~\ref{FundThm}) it follows that the $H$-bimodule $\Gamma$ of a bicovariant
FODC $(\Gamma,\mathrm{d})$ on $H$ is $H$-generated by
its coinvariant elements. Assume that ${}^{\mathrm{co}H}\Gamma$ is  finite-dimensional and choose a basis $\{\omega^i\}_{i\in I}$ of ${}^{\mathrm{co}H}\Gamma$.
There are linear maps $f_j{}^i\colon H\rightarrow\Bbbk$
and elements $M_j{}^i\in H$ such that
\begin{equation}\label{eq55}
    a\omega^i
    =\omega^j(f_j{}^i*a)~~~
    \text{ for all }a\in H~~~
    \text{ and }~~~
    \delta_\Gamma(\omega^i)
    =\omega^j\otimes M_j{}^i~,
\end{equation}
where we used Einstein sum convention and the convolution action
$f_j{}^i*a:=a_1f_j{}^i(a_2)$. It is easy to verify that
$\omega^ia=(S(f_j{}^i)*a)\omega^j$, where
$S(f_j{}^i)*a=a_1f_j{}^i(S(a_2))$,
and
\begin{equation*}
\begin{split}
    \Delta(f_j{}^i)
    &=f_j{}^k\otimes f_k{}^i~,~~~~
    \epsilon(f_j{}^i)
    =\delta_j{}^i~,\\
    \Delta(M_j{}^i)
    &=M_j{}^k\otimes M_k{}^i~,~
    \epsilon(M_j{}^i)
    =\delta_j{}^i
\end{split}
\end{equation*}
for all $i,j\in I$. Moreover, we have the compatibility condition
$(a*f_i{}^j)M_j{}^k=M_i{}^j(f_j{}^k*a)$
for all $a\in H$. We obtain a basis $\{\eta^i\}_{i\in I}$ of $\Gamma^{\mathrm{co}H}$
by defining $\eta^i=\omega^jS(M_j{}^i)$. The braiding isomorphism
$\sigma^\mathcal{W}_{\Gamma,\Gamma}\colon\Gamma\otimes_H\Gamma
\rightarrow\Gamma\otimes_H\Gamma$ transforms basis elements as follows
$\sigma^\mathcal{W}_{\Gamma,\Gamma}(\omega^i\otimes_H\eta^j)
=\eta^j\otimes_H\omega^i$, since it is $H$-colinear it closes in
${}^{\mathrm{co}H}\Gamma\otimes{}^{\mathrm{co}H}\Gamma$,
$\sigma^\mathcal{W}_{\Gamma,\Gamma}(\omega^i\otimes_H\omega^j)
=\omega^k\otimes_H\omega^\ell\Sigma_{k\ell}{}^{ij}$.
From
$$
\sigma^\mathcal{W}_{\Gamma,\Gamma}(\omega^i\otimes_H\omega^j)
=\sigma^\mathcal{W}_{\Gamma,\Gamma}(\omega^i\otimes_H\eta^rM_r{}^j)
=\eta^r\otimes_H(S(f_k{}^i)*M_r{}^j)\omega^k
=\omega^\ell\otimes_H\omega^kf_k{}^iS(M_\ell{}^j)
$$
we have $\Sigma_{k\ell}{}^{ij}=f_\ell{}^i(S(M_k{}^j))\in\Bbbk$.
Moreover,
\begin{equation}\label{eq34}
    \omega^i\wedge\omega^j
    =\omega^i\otimes_H\omega^j
    -\omega^k\otimes_H\omega^\ell\Lambda_{k\ell}{}^{ij}~,
\end{equation}
where $\Lambda=\Sigma^{-1}$,
hence $\Lambda_{k\ell}{}^{ij}:=f_k{}^j(M_\ell{}^i)\in\Bbbk$.
Recall from Proposition~\ref{prop100} that $\mathrm{Hom}_H(\Gamma,H)$ is
a bicovariant bimodule since we assume $\Gamma$ to be finitely generated
as a right $H$-module. We call it the module of 
vector fields
$$
\mathfrak{X}:=\mathrm{Hom}_H(\Gamma,H)
$$
and consider the
basis $\{\chi_i\}_{i\in I}\subseteq{}^{\mathrm{co}H}\mathfrak{X}$
of left invariant elements dual to $\{\omega^i\}_{i\in I}$, i.e. such that
$\langle\chi_i,\omega^j\rangle=\delta_i^j$. We obtain
\begin{equation*}
    \chi_ia
    =(f_i{}^j*a)\chi_j,~~~~~
    a\chi_i
    =\chi_j(\overline{S}(f_i{}^j)*a),~~~~~
    \delta_{\mathfrak{X}}(\chi_i)
    =\chi_j\otimes S(M_i{}^j),~~~~~
    \lambda_{\mathfrak{X}}(\chi_i)=1\otimes\chi_i
\end{equation*}
for all $a\in H$ and $i\in I$. In terms of the left invariant vector fields
and one forms we have
$\mathrm{d}a=\omega^i(\chi_i*a)$.
We define the \textit{structure constants}
$$
C^i_{kj}:=(\chi_k*\chi_j)(x^i)\in\Bbbk~,
$$ where
$x^k\in\ker(\epsilon)$ such that $\chi_i(x^k)=\delta_i^k$. From
\begin{equation*}
    0=\mathrm{d}^2a
    =\mathrm{d}(\omega^i(\chi_i*a))
    =\mathrm{d}(\omega^i)(\chi*a)
    -\omega^i\wedge\omega^j(\chi_j*(\chi_i*a))~,
\end{equation*}
setting $a=x^i$ in $\mathrm{d}^2(a_2)\overline{S}(a_1)$
we obtain the \textit{Cartan-Maurer formula}
\begin{equation}\label{CartanMaurer}
    \mathrm{d}\omega^i
    =C^i_{kj}\omega^j\wedge\omega^k~.
\end{equation}
Applying (\ref{eq34}) we further deduce
\begin{equation}\label{CartanMaurerII}
    \mathrm{d}\omega^i
    =C^i_{kj}(\omega^j\otimes_H\omega^k
    -\Lambda_{mn}{}^{jk}\omega^m\otimes_H\omega^n)
    =\mathbf{C}^i_{kj}\omega^j\otimes_H\omega^k~,
\end{equation}
where we defined the braided antisymmetric structure constants $\mathbf{C}^i_{kj}:=C^i_{kj}-\Lambda_{jk}{}^{mn}C^i_{nm}\in\Bbbk$.

\section{Connections on Bicovariant Bimodules}\label{Sec4}

We give a general introduction to connections on Hopf algebra modules
for covariant first order differential calculi (Section~\ref{Sec4.1}),
their curvature and torsion (Section~\ref{Sec4.2}) and duality of connections
(Section~\ref{Sec4.3}). We illustrate all notions via concrete canonical examples
on bicovariant bimodules. In particular, we prove that there are
flat connections on any bicovariant bimodule and torsion-free connections on those of one forms and vector fields. Utilizing the tensor product $\,\hat{\otimes}\,$ of rational morphisms,
we define the sum of connections in Section~\ref{Sec4.4}, their curvatures are studied in Section~\ref{Sec4.5}. 

\subsection{Definitions and Canonical Examples}\label{Sec4.1}

Fix
a bicovariant differential calculus $(\Omega^\bullet,\wedge,\mathrm{d})$
on $H$.

\begin{definition}
A \textit{right connection} on a right $H$-module $M$ is a linear map
$\dst\colon M\rightarrow M\otimes_H\Omega^1$ satisfying
\begin{equation*}
    \dst(ma)=\dst(m)a+m\otimes_H\mathrm{d}a
\end{equation*}
for all $m\in M$ and $a\in H$. The set of all right connections on $M$ is denoted by
$\mathrm{Con}_H(M)$.

A \textit{left connection} on a left $H$-module $M$ is a linear map
$\std\colon M\rightarrow\Omega^1\otimes_HM$ satisfying
\begin{equation*}
    \std(am)=a\std(m)+\mathrm{d}a\otimes_Hm
\end{equation*}
for all $m\in M$ and $a\in H$. The set of all left connections on $M$ is denoted by
${}_H\mathrm{Con}(M)$.
\end{definition}
Given a right $H$-module $M$ there is a free and transitive right action
\begin{equation*}
    \mathrm{Con}_H(M)\times\mathrm{Hom}_H(M,M\otimes_H\Omega^1)
    \rightarrow\mathrm{Con}_H(M)~,~~~~~(\dst,\phi)\mapsto\dst+\phi~.
\end{equation*}
of $\mathrm{Hom}_H(M,M\otimes_H\Omega^1)$ on $\mathrm{Con}_H(M)$
turning $\mathrm{Con}_H(M)$ into
an affine space over $\mathrm{Hom}_H(M,M\otimes_H\Omega^1)$.
Similarly, for $M$ a left $H$-module, ${}_H\mathrm{Con}(M)$ becomes an affine space over ${}_H\mathrm{Hom}(M,
\Omega^1\otimes_HM)$.
On covariant modules there always exist connections. We give their explicit form in the
following lemma.
\begin{lemma}\label{prop03}\phantom{cc}
\begin{enumerate}
\item[i.)] 
For any module $M$ of $\mathcal{M}_H^H$ there is a
right connection
$\mathrm{d}^M
\in\mathrm{Con}_H(M)
$
defined by
\begin{equation}\label{eq33}
    \mathrm{d}^M(m)=m_0S(m_1)\otimes_H\mathrm{d}m_2
\end{equation}
for all $m\in M$.

\item[ii.)] For any module $M$ of ${}_H\mathcal{M}^H$ there is a
left connection
${}^{\overline{M}}\mathrm{d}
\in{}_H\mathrm{Con}(M)$
defined by ${}^{\overline{M}}\mathrm{d}(m)=\mathrm{d}m_2\otimes_H\overline{S}(m_1)m_0$
for all $m\in M$.
\end{enumerate}
\begin{enumerate}
\item[iii.)] 
For any module $M$ of ${}^H\mathcal{M}_H$ there is a
right connection
$\mathrm{d}^{\overline{M}}
\in\mathrm{Con}_H(M)$
defined by $\mathrm{d}^{\overline{M}}(m)=m_0\overline{S}(m_{-1})\otimes_H\mathrm{d}m_{-2}$
for all $m\in M$. 

\item[iv.)] For any module $M$ of ${}_H^H\mathcal{M}$ there is a
left connection
${}^M\mathrm{d}
\in{}_H\mathrm{Con}(M)$
defined by ${}^M\mathrm{d}(m)=\mathrm{d}m_{-2}\otimes_HS(m_{-1})m_0$
for all $m\in M$.
\end{enumerate}
\end{lemma}
\begin{proof}
We prove that (\ref{eq33}) defines a right connection on $M$. In fact
$$
\mathrm{d}^M(ma)
=m_0a_1S(m_1a_2)\otimes_H\mathrm{d}(m_2a_3)
=m_0S(m_1)\otimes\mathrm{d}(m_2)a
+m_0S(m_1)\otimes_Hm_2\mathrm{d}a
=\mathrm{d}^M(m)a
+m\otimes_H\mathrm{d}a
$$
for all $m\in M$ and $a\in H$
by the Leibniz rule of $\mathrm{d}$ and the right covariance of $M$.
The statements about
${}^M\mathrm{d}$, $\mathrm{d}^{\overline{M}}$ and
${}^{\overline{M}}\mathrm{d}$ follow in complete analogy.
\end{proof}
The connection $\mathrm{d}^M$ is the unique right connection that vanishes
on $M^{\mathrm{co}H}$. Similarly, ${}^{\overline{M}}\mathrm{d}$ is the unique
left connection that vanishes on $M^{\mathrm{co}H}$. Using ${}^{\mathrm{co}H}M$
instead of $M^{\mathrm{co}H}$ leads to the connections $\mathrm{d}^{\overline{M}}$
and ${}^M\mathrm{d}$.
We sometimes refer to the connections of Lemma~\ref{prop03} as the
\textit{canonical connections}. These are easily proven to be $H$-bicolinear maps
and so in particular rational morphisms according to Lemma~\ref{lemma01}.
Together with the fact that connections form an affine space over right
$H$-linear maps (which we proved in Proposition~\ref{prop100} to be rational)
this implies that any connection on a finitely generated
covariant module is rational. 
\begin{proposition}\label{PropConRat}
Let $M$ be a finitely generated right $H$-module.
\begin{enumerate}
\item[i.)] If $M$ is a module of $\mathcal{M}_H^H$ then there is the set inclusion
$\mathrm{Con}_H(M)\subset\mathrm{HOM}^\mathrm{Ad}(M,M\otimes_H\Omega^1)$,
i.e. every right connection on $M$ is a right rational morphism.

\item[ii.)] If $M$ is a module of ${}^H\mathcal{M}_H$ then there is the set inclusion
$\mathrm{Con}_H(M)\subset{}^\mathrm{Ad}\mathrm{HOM}(M,M\otimes_H\Omega^1)$,
i.e. every right connection on $M$ is a left rational morphism.

\item[iii.)] If $M$ is a module of ${}^H\mathcal{M}_H^H$ then there is the set inclusion
$\mathrm{Con}_H(M)\subset\mathrm{HOM}(M,M\otimes_H\Omega^1)$,
i.e. every right connection on $M$ is a rational morphism.
\end{enumerate}
\end{proposition}
\begin{proof}
Consider a right covariant right module $M$ and the right connection
$\mathrm{d}^M$ on $M$ presented in Lemma~\ref{prop03}~i.). Recalling the
definition of $\delta^\mathrm{Ad}$ in (\ref{coadr}) we observe that
\begin{align*}
    \delta^\mathrm{Ad}(\mathrm{d}^M)(m)
    &=\mathrm{d}^M(m_0)_0\otimes\mathrm{d}^M(m_0)_1S(m_1)\\
    &=(m_0S(m_1)\otimes_H\mathrm{d}m_2)_0
    \otimes(m_0S(m_1)\otimes_H\mathrm{d}m_2)_1S(m_3)\\
    &=m_0S(m_2)_1\otimes_H\mathrm{d}m_3\otimes m_1S(m_2)_2m_4S(m_5)\\
    &=m_0S(m_1)\otimes_H\mathrm{d}m_2\otimes 1_H\\
    &=\mathrm{d}^M(m)\otimes 1_H\\
    &=(\mathrm{d}^M\otimes 1_H)(m)
\end{align*}
for all $m\in M$, that is, 
$\delta^\mathrm{Ad}(\mathrm{d}^M)=\mathrm{d}^M\otimes 1_H$. Hence
$\mathrm{d}^M\in\mathrm{HOM}^H(M,M\otimes_H\Omega^1)$, in particular it
is right rational.
Since any right connection $\dst$ on $M$ equals $\mathrm{d}^M+\phi$ for
a unique right $H$-linear map $\phi\in\mathrm{Hom}_H(M,M\otimes_H\Omega^1)$
the result $i.)$ follows from Proposition~\ref{prop100}~$i.)$, which states
that right rationality of $\phi$ holds if
$M$ is a finitely generated right $H$-module.
Similarly, the right connection $\mathrm{d}^{\overline{M}}$ of
Lemma~\ref{prop03}~iii.) satisfies $\lambda^\mathrm{Ad}(\mathrm{d}^{\overline{M}})
=1_H\otimes\mathrm{d}^{\overline{M}}$ and from Proposition~\ref{prop100}~$ii.)$
we conclude that any right connection is left rational, i.e. $ii.)$.
The last statement of the proposition follows from $i.)$ and $ii.)$.
\end{proof}
There is an analogue version of Proposition~\ref{PropConRat} for 
the left connections of Lemma~\ref{prop03}~ii.) and iv.).
In this case rationality is with respect to the internal $\mathrm{Hom}$-functors
$\overleftarrow{\mathrm{HOM}}^\mathrm{Ad}$, 
${}^\mathrm{Ad}\overleftarrow{\mathrm{HOM}}$ and
$\overleftarrow{\mathrm{HOM}}$ of morphisms 'acting from the right' defined after
Proposition~\ref{prop05}.
\begin{remark}
The coactions $\delta^\mathrm{Ad}$ and $\lambda^\mathrm{Ad}$ do
\textit{not} restrict to coactions on $\mathrm{Con}_H(M)$,
which is an affine space, not a vector space.
Let $M$ be a finitely generated right $H$-module.
If $M$ is a module of $\mathcal{M}_H^H$ and $\dst\in\mathrm{Con}_H(M)$ then
$\delta^\mathrm{Ad}(\dst)=\dst{}_0\otimes\dst{}_1
\in\mathrm{HOM}^\mathrm{Ad}(M,M\otimes_H\Omega^1)\otimes H$ satisfies the
\textit{pseudo Leibniz rule}
\begin{equation*}
    \dst{}_0(ma)\otimes\dst{}_1
    =\dst{}_0(m)a\otimes\dst{}_1
    +m\otimes_H\mathrm{d}a\otimes 1_H
\end{equation*}
for all $m\in M$ and $a\in H$. This is the case since
\begin{align*}
    \dst{}_0(ma)\otimes\dst{}_1
    &=\dst(m_0a_1)_0\otimes\dst(m_0a_1)_1S(m_1a_2)\\
    &=(\dst(m_0)a_1)_0\otimes(\dst(m_0)a_1)_1S(m_1a_2)
    +(m_0\otimes_H\mathrm{d}a_1)_0\otimes(m_0\otimes_H\mathrm{d}a_1)_1S(m_1a_2)\\
    &=\dst(m_0)_0a_1\otimes\dst(m_0)_1a_2S(m_1a_3)
    +m_0\otimes_H\mathrm{d}a_1\otimes m_1a_2S(m_1a_3)\\
    &=\dst{}_0(m)a\otimes\dst{}_1
    +m_0\otimes_H\mathrm{d}a\otimes 1_H
\end{align*}
by the Leibniz rule of $\dst$. Similarly one proves that if
$M$ is a module of ${}^H\mathcal{M}_H$ finitely generated as a right $H$-module
and $\dst\in\mathrm{Con}_H(M)$ then
$\lambda^\mathrm{Ad}(\dst)=\dst{}_{-1}\otimes\dst{}_0$
satisfies the pseudo Leibniz rule
\begin{equation}\label{eq38}
    \dst{}_{-1}\otimes\dst{}_0(ma)
    =\dst{}_{-1}\otimes\dst{}_0(m)a
    +1_H\otimes m\otimes_H\mathrm{d}a
\end{equation}
for all $m\in M$ and $a\in H$.
\end{remark}
\medskip

We fix a bicovariant bimodule $M$ for the rest of this section.
We observe that the right connection
$\mathrm{d}^M
\in\mathrm{Con}_H(M)$ of Lemma~\ref{prop03}~$i.)$ satisfies
the additional Leibniz rule
\begin{equation*}
    \mathrm{d}^M(am)
    =a_1m_0S(a_2m_1)\otimes_H\mathrm{d}(a_3m_2)
    =a\cdot\mathrm{d}^M(m)+\sigma^\mathcal{W}_{\Omega^1,M}(\mathrm{d}a\otimes_Hm)
\end{equation*}
for all $a\in H$ and $m\in M$, where $\sigma^\mathcal{W}$ denotes the braiding of
bicovariant bimodules defined in (\ref{WorBraiding}). It is therefore a 
\textit{bimodule connection} (cf. \cite{D-VMichor,MajidBeggs2011}).
\begin{definition}
A right connection 
$\dst
\in\mathrm{Con}_H(M)$ is said to be a bimodule right connection
if
\begin{equation}\label{eq64}
    \dst(am)=a\dst(m)+\sigma^\mathcal{W}_{\Omega^1,M}(\mathrm{d}a\otimes_Hm)
\end{equation}
for all $a\in H$ and $m\in M$, while a left connection 
$\std
\in{}_H\mathrm{Con}(M)$ is said to be a bimodule left connection
if
\begin{equation}\label{eq65}
    \std(ma)=\dst(m)a+\sigma^\mathcal{W}_{M,\Omega^1}(m\otimes_H\mathrm{d}a)
\end{equation}
for all $a\in H$ and $m\in M$. There are obvious versions of 
$\overline{\sigma}^\mathcal{W}$-bimodule connections if we replace 
$\sigma^\mathcal{W}$
by $\overline{\sigma}^\mathcal{W}$ in (\ref{eq64}) and (\ref{eq65}).
\end{definition}
For bimodule connections on bicovariant bimodules there is an immediate
correspondence between right and left connections.
\begin{remark}
If $\dst\colon M\rightarrow M\otimes_H\Omega^1$ is a bimodule right connection, then
\begin{equation*}
    \dst{}^{\overline{\sigma}}:=\overline{\sigma}^\mathcal{W}_{M,\Omega^1}\circ\dst
    \colon M\rightarrow\Omega^1\otimes_HM
\end{equation*}
is a $\overline{\sigma}^\mathcal{W}$-bimodule left connection. Namely,
\begin{equation*}
    \dst{}^{\overline{\sigma}}(am)=a\dst{}^{\overline{\sigma}}(m)
    +\mathrm{d}a\otimes_Hm
    \quad\text{ and }\quad
    \dst{}^{\overline{\sigma}}(ma)=\dst{}^{\overline{\sigma}}(m)a
    +\overline{\sigma}_{M,\Omega^1}^\mathcal{W}(m\otimes_H\mathrm{d}a)
\end{equation*}
for all $a\in H$ and $m\in M$. Conversely, a 
$\overline{\sigma}^\mathcal{W}$-bimodule left connection
$\std\colon M\rightarrow\Omega^1\otimes_HM$ gives rise to a bimodule right connection
$\std{}^\sigma:=\sigma^\mathcal{W}_{\Omega^1,M}\circ\std
\colon M\rightarrow M\otimes_H\Omega^1$. Those constructions are inverse to each other,
i.e. $(\dst{}^{\overline{\sigma}})^\sigma=\dst$ and
$(\std{}^\sigma)^{\overline{\sigma}}=\std$, as one easily verifies.
There is an analogous correspondence between $\overline{\sigma}^\mathcal{W}$-bimodule
right connections and bimodule left connections.
\end{remark}
\begin{example}\label{example02}
The canonical connections $\mathrm{d}^M
\in\mathrm{Con}_H(M)$ and
${}^M\mathrm{d}
\in{}_H\mathrm{Con}(M)$ are bimodule
right and left connections, respectively. The corresponding
$\overline{\sigma}^\mathcal{W}$-bimodule left and right connections are
\begin{equation*}
    {}^{\overline{M}}\mathrm{d}=(\mathrm{d}^M)^{\overline{\sigma}}
    \colon M\rightarrow\Omega^1\otimes_HM
    \quad\text{ and }\quad
    \mathrm{d}^{\overline{M}}=({}^M\mathrm{d})^\sigma
    \colon M\rightarrow M\otimes_H\Omega^1.
\end{equation*}
\end{example}

If $(\Omega^\bullet,\wedge,\mathrm{d})$ is a Woronowicz differential
calculus on a braided exterior algebra of Proposition~\ref{propBEA}
we can use the structure constants to construct a right connection
$\Omega^1\rightarrow\Omega^1\otimes_H\Omega^1$.
Recall that, given a finite-dimensional basis $\{\omega^i\}_{i\in I}$ of 
${}^{\mathrm{co}H}\Omega^1$, the structure constants $C^i_{kj}$, 
$\mathbf{C}^i_{kj}\in\Bbbk$ are determined via the Cartan-Maurer formulas
(\ref{CartanMaurer}) and (\ref{CartanMaurerII}).
We define two linear maps 
$\dst{}^c,\widetilde{\dst{}^c}\colon\Omega^1\rightarrow\Omega^1\otimes_H\Omega^1$
by
\begin{equation}\label{eq67}
    \dst{}^c(\omega^i)
    =-\omega^j\otimes_H\omega^kC^i_{kj}
    \quad\text{ and }\quad
    \widetilde{\dst{}^c}(\omega^i)
    =-\omega^j\otimes_H\omega^k\mathbf{C}^i_{kj}
    =-\mathrm{d}\omega^i
\end{equation}
for all $i\in I$ and extend those maps as right connections
on the free right $H$-module $\Omega^1={}^{\mathrm{co}H}\Omega^1\otimes H$.
Namely, $\dst{}^c(\omega^ia_i)
=-\omega^j\otimes_H\omega^kC^i_{kj}a_i
+\omega^i\otimes_H\mathrm{d}a_i$
for all $a_i\in H$ and similarly for $\widetilde{\dst{}^c}$.
In the following we call $\dst{}^c$ and $\widetilde{\dst{}^c}$ the
\textit{structure constants connections}.
Whenever discussing them we implicitly assume a Woronowicz differential calculus.
\begin{lemma}\label{prop04}
The structure constants connections are left $H$-coinvariant elements, i.e.,
$\dst{}^c,\widetilde{\dst{}^c}
\in{}^H\mathrm{Hom}(\Omega^1,\Omega^1\otimes_H\Omega^1)$. 
\end{lemma}
\begin{proof}
First note that the structure constant connections are left rational by
Proposition~\ref{PropConRat}.
We prove that $\dst{}^c$ is left $H$-coinvariant. Then it is clear that
$\widetilde{\dst{}^c}$ is left $H$-coinvariant, as well.
For any $\omega\in\Omega^1$ and a basis $\{\omega^i\}_{i\in I}$ of 
${}^{\mathrm{co}H}\Omega^1$ there are unique $a_i\in H$ such that
$\omega=\omega^ia_i$ by the Fundamental Theorem of Hopf Modules. Then
\begin{align*}
    \dst{}^c_{-1}\otimes\dst{}^c_0(\omega)
    &=\dst{}^c(\omega_0)_{-1}\overline{S}(\omega_{-1})
    \otimes\dst{}^c(\omega_0)_0\\
    &=\dst{}^c(\omega^i(a_i)_2)_{-1}\overline{S}((a_i)_1)
    \otimes\dst{}^c(\omega^i(a_i)_2)_0\\
    &=-(\omega^j\otimes_H\omega^kC^i_{kj}(a_i)_2)_{-1}\overline{S}((a_i)_1)
    \otimes(\omega^j\otimes_H\omega^kC^i_{kj}(a_i)_2)_0\\
    &\quad+(\omega^i\otimes_H\mathrm{d}(a_i)_2)_{-1}\overline{S}((a_i)_1)
    \otimes(\omega^i\otimes_H\mathrm{d}(a_i)_2)_0\\
    &=-(a_i)_2\overline{S}((a_i)_1)
    \otimes\omega^j\otimes_H\omega^kC^i_{kj}(a_i)_3
    +(a_i)_2\overline{S}((a_i)_1)
    \otimes\omega^i\otimes_H\mathrm{d}(a_i)_3\\
    &=1_H\otimes(-\omega^j\otimes_H\omega^kC^i_{kj}a_i
    +\omega^i\otimes_H\mathrm{d}a_i)\\
    &=1_H\otimes\dst{}^c(\omega)
\end{align*}
follows.
\end{proof}
If we extend the linear maps in (\ref{eq67}) as left connections we similarly obtain two
left $H$-coinvariant left connections.

Considering structure constants with respect to a
basis $\{\eta^i\}_{i\in I}$ of right $H$-coinvariant elements of $\Omega^1$
leads to yet another pair of right/left connections, this time right $H$-coinvariant.
In the following we mainly focus on $\dst{}^c$ and $\widetilde{\dst{}^c}$.

\subsection{Curvature and Torsion}\label{Sec4.2}

Let $(\Omega^\bullet,\wedge,\mathrm{d})$ be a bicovariant differential calculus
over $H$. Using the differential $\mathrm{d}$ we extend any right
connection $\dst
\in\mathrm{Con}_H(M)$ on a right $H$-module
$M$ to a linear map $\dst{}^\bullet\colon M\otimes_H\Omega^\bullet
\rightarrow M\otimes_H\Omega^{\bullet+1}$ by
\begin{equation}\label{eq68}
  \dst{}^\bullet(m\otimes_H\omega)=\dst(m)\wedge\omega+m\otimes_H\mathrm{d}\omega
  \end{equation}
for $m\in M$ and $\omega\in\Omega^\bullet$.
Note that $\dst{}^\bullet$
is well-defined on the balanced tensor product $M\otimes_H\Omega^\bullet$ since
\begin{align*}
    \dst(ma)\wedge\omega+ma\otimes_H\mathrm{d}\omega
    =(\dst(m)a+m\otimes_H\mathrm{d}a)\wedge\omega
    +m\otimes_Ha\mathrm{d}\omega
    =\dst(m)\wedge a\omega+m\otimes_H\mathrm{d}(a\omega)
\end{align*}
for $a\in H$ by the Leibniz rules of $\dst$ and $\mathrm{d}$.
For a left connection $\std
\in{}_H\mathrm{Con}(M)$
on a left $H$-module $M$ we define a linear map
$\std{}^\bullet\colon\Omega^\bullet\otimes_HM\rightarrow\Omega^{\bullet+1}\otimes_HM$
by
\begin{equation}\label{eq70}
\std{}^\bullet(\omega\otimes_Hm)
=\omega\wedge\std(m)+(-1)^{|\omega|}\mathrm{d}\omega\otimes_Hm
\end{equation}
where
$\omega\in\Omega^k$, $|\omega|=k$ and
$m\in M$. From
\begin{align*}
    \omega a\wedge\std(m)+(-1)^{|\omega a|}\mathrm{d}(\omega a)\otimes_Hm
    &=\omega\wedge a\std(m)+(-1)^{|\omega|}\mathrm{d}(\omega)a\otimes_Hm
    +(-1)^{2|\omega|}\omega\wedge\mathrm{d}a\otimes_Hm\\
    &=\omega\wedge\std(am)+(-1)^{|\omega|}\mathrm{d}\omega\otimes_Ham~,
\end{align*}
where $a\in H$,
we see that $\std{}^\bullet$ is well-defined on the balanced tensor product
$\Omega^\bullet\otimes_HM$.
\begin{lemma}\label{lemma10}
The extension $\dst{}^\bullet\colon M\otimes_H\Omega^\bullet
\rightarrow M\otimes_H\Omega^{\bullet+1}$ of a right connection
$\dst\in\mathrm{Con}_H(M)$ on a right $H$-module $M$ satisfies the
graded Leibniz rule
\begin{equation*}
    \dst^\bullet((m\otimes_H\omega)\cdot a)
    =(\dst^\bullet(m\otimes_H\omega))\cdot a
    +(-1)^{|\omega|}m\otimes_H\omega\wedge\mathrm{d}a~.
\end{equation*}
The extension $\std{}^\bullet\colon\Omega^\bullet\otimes_H M
\rightarrow\Omega^{\bullet+1}\otimes_HM$ of a left connection
$\std\in{}_H\mathrm{Con}(M)$ on a left $H$-module $M$ satisfies the
graded Leibniz rule
\begin{equation*}
    \std^\bullet(a\cdot(\omega\otimes_Hm))
    =a\cdot\std^\bullet(\omega\otimes_Hm)
    +(-1)^{|\omega|}\mathrm{d}a\wedge\omega\otimes_Hm~,
\end{equation*}
for all $m\in M$, $\omega\in\Omega^\bullet$ and $a\in H$.
\end{lemma}
\begin{remark}
Let $M$ be a module of ${}_H^H\mathcal{M}_H^H$, finitely generated as a 
right $H$-module and consider a  connection $\dst\in\mathrm{Con}_H(M)$.
Using the tensor product $\,\hat{\otimes}\,$ introduced in Theorem~\ref{prop06} we have two well-defined linear maps
$\wedge_{23}\circ(\dst\,\hat{\otimes}\,\mathrm{id}_{\Omega^\bullet}),
\mathrm{id}_M\,\hat{\otimes}\,\mathrm{d}\colon M\otimes_H\Omega^\bullet\rightarrow
M\otimes_H\Omega^{\bullet+1}$, and their sum equals $\dst{}^\bullet\colon M\otimes_H\Omega^\bullet\rightarrow M\otimes_H\Omega^{\bullet+1}$.
Indeed 
\begin{equation}\label{eq69}
    \dst{}^\bullet=\wedge_{23}\circ(\dst\,\hat{\otimes}\,\mathrm{id}_{\Omega^\bullet})
    +\mathrm{id}_M\,\hat{\otimes}\,\mathrm{d}
\end{equation}
is an equation of linear maps
$M\otimes_H\Omega^\bullet\rightarrow M\otimes_H\Omega^{\bullet+1}$,
where all individual terms are well-defined.

Now consider a left connection $\std\in{}_H\mathrm{Con}(M)$ 
on $M$.
Since left connection are understood as maps ``acting from the right",
by employing the analogue of $\,\hat{\otimes}\,$ in $\overleftarrow{\mathrm{HOM}}$ we can define $\std{}^\bullet$ as a sum of well-defined linear maps $\Omega^\bullet\otimes_HM\to\Omega^{\bullet+1}\otimes_HM$.
Since we are mainly interested in the theory of right connections we are not going
to spell this out in detail.
\end{remark}

The extensions of a connection are used to define the curvature
as the square of the connection.
\begin{definition}[Curvature]
The curvature of a
right connection $\dst
\in\mathrm{Con}_H(M)$ on a right $H$-module $M$ is
by definition the linear map
\begin{equation}\label{defcurvatureR}
    \mathrm{R}^{\dst}=\dst^\bullet\circ\dst\colon M
    \rightarrow M\otimes_H\Omega^2~,
\end{equation}
while the curvature of a
left connection $\std
\in{}_H\mathrm{Con}(M)$ on a left $H$-module $M$ is
by definition the linear map
\begin{equation*}
    \mathrm{R}^\std=\std^\bullet\circ\std\colon M
    \rightarrow\Omega^2\otimes_HM~.
\end{equation*}
A connection with zero curvature is said to be flat.
\end{definition}
The following lemma shows that these curvatures are right and left $H$-linear maps, respectively. Hence, they determine assignments
$\mathrm{R}\colon\mathrm{Con}_H(M)\rightarrow\mathrm{Hom}_H(M,M\otimes_H\Omega^2)$
between the affine spaces of connections and the
vector spaces $\mathrm{Hom}_H(M,M\otimes_H\Omega^2)$ and ${}_H\mathrm{Hom}(M,\Omega^2\otimes_HM)$.
\begin{lemma}\label{lemma12}
The curvature 
$\mathrm{R}^{\dst}\in\mathrm{Hom}_H(M,M\otimes_H\Omega^2)$ of a right connection
is right $H$-linear and the
curvature $\mathrm{R}^\std\in{}_H\mathrm{Hom}(M,\Omega^2\otimes_HM)$ 
of a left connection is left $H$-linear.
\end{lemma}
For connections on $\Omega^1$ or the dual bicovariant bimodule $\mathfrak{X}$
there is a natural definition of torsion. In the following
$\{\omega^i\}_{i\in I}$ denotes a basis of the vector space
${}^{\mathrm{co}H}\Omega^1$ of left coinvariant elements, that we call
left-invariant following \cite{Woronowicz1989}. Let
$\{\chi_i\}_{i\in I}$ be the corresponding dual basis of ${}^{\mathrm{co}H}\mathfrak{X}$.
We further abbreviate $\mathrm{I}:=\omega^i\otimes_H\chi_i
\in\Omega^1\otimes_H\mathfrak{X}$, where the sum over repeated indices is understood.
Recall that $\wedge\colon\Omega^1\otimes_H\Omega^1\rightarrow\Omega^2$, 
$\omega\otimes_H\eta\mapsto\omega\wedge\eta$ is the product of the bicovariant differential calculus
$(\Omega^\bullet,\wedge,\mathrm{d})$.
\begin{definition}[Torsion]
The torsion of a right connection $\dst
\in\mathrm{Con}_H(\Omega^1)$
on $\Omega^1$ is the right $H$-linear map
\begin{equation*}
    \mathrm{Tor}^{\dst}
    :=\wedge\circ\dst+\mathrm{d}
    \colon\Omega^1\rightarrow\Omega^2~.
\end{equation*}
Using the dual module $\mathfrak{X}$, $\mathrm{Tor}^{\dst}\in\Omega^2\otimes_H\mathfrak{X}$.
The torsion of a left connection
$\std
\in{}_H\mathrm{Con}(\mathfrak{X})$ on
$\mathfrak{X}$ is the $\mathfrak{X}$-valued $2$-form
\begin{equation*}
    \mathrm{Tor}^{\std}
    :=\std{}^\bullet(\mathrm{I})
    \in\Omega^2\otimes_H\mathfrak{X}~.
\end{equation*}
A connection with zero torsion is said to be torsion-free.
\end{definition}
In analogy to Lemma~\ref{lemma12} we have that the torsion of a right connection
corresponds to an assignment $\mathrm{Tor}\colon\mathrm{Con}_H(\Omega^1)
\rightarrow\mathrm{Hom}_H(\Omega^1,\Omega^2)$.

As an example we calculate curvature and torsion of the connections
introduced in the previous section. It turns out that the canonical connections
are all flat with non-trivial torsion that is easily computable, while
the structure
constants connection $\dst{}^c$ is torsion-free and non-flat.
\begin{proposition}
\begin{enumerate}
\item[i.)]
The canonical connections defined in Proposition~\ref{prop03} are flat.

\item[ii.)]
The structure constants connection $\dst{}^c$ of
Proposition~\ref{prop04} is torsion-free with curvature
$$
\mathrm{R}^{\dst{}^c}(\omega^i)
=(C^i_{kj}C^j_{nm}-C^i_{\ell m}C^\ell_{kn})
\omega^m\otimes_H\omega^n\wedge\omega^k~.
$$
\end{enumerate}
\end{proposition}
\begin{proof}
Since $M\cong M^{\mathrm{co}H}\otimes H$, $\mathrm{d}^M(M^{\mathrm{co}H})=0$
and $\mathrm{R}^{\mathrm{d}^M}$ is right $H$-linear we have
$\mathrm{R}^{\mathrm{d}^M}=0$.
Similarly for the other canonical connections.

Next we consider the right connection $\dst{}^c$ 
on $\Omega^1$. By the Cartan-Maurer formula (\ref{CartanMaurer})
$
\mathrm{Tor}^{\dst{}^c}(\omega^i)
=-\omega^j\wedge\omega^kC^i_{kj}+\mathrm{d}\omega^i
=0,
$
which implies $\mathrm{Tor}^{\dst{}^c}=0$ since $\mathrm{Tor}^{\dst{}^c}$ is
right $H$-linear.
The curvature of $\dst{}^c$ is given by
\begin{align*}
    \mathrm{R}^{\dst{}^c}(\omega^i)
    &=(\dst{}^c)^\bullet(\dst{}^c(\omega^i))\\
    &=-C^i_{kj}(\dst{}^c)^\bullet(\omega^j\otimes_H\omega^k)\\
    &=-C^i_{kj}\dst{}^c(\omega^j)\wedge\omega^k
    -C^i_{kj}\omega^j\otimes_H\mathrm{d}\omega^k\\
    &=C^i_{kj}C^j_{nm}\omega^m\otimes_H\omega^n\wedge\omega^k
    -C^i_{kj}C^k_{sr}\omega^j\otimes_H\omega^r\wedge\omega^s~.
\end{align*}
This concludes the proof of the proposition.
\end{proof}

\subsection{Dual Connection}\label{Sec4.3}

We recall the duality of left and right modules in a more general setting than that in
Section~\ref{Sec3.1} on bicovariant bimodules. We then study dual connections.
\begin{definition}\label{def02}
A right $H$-module $M$ and a left $H$-module $N$
are said to be dual if there is an $H$-bilinear map
$$
\langle\cdot,\cdot\rangle\colon N\otimes M\rightarrow H
$$
which is non-degenerate, i.e. $\langle n,m\rangle=0$
for all $n\in N$ implies $m=0$ and
$\langle n,m\rangle=0$ for all $m\in M$ implies $n=0$.
\end{definition}
If $M$ is a finitely generated projective right $H$-module,
the vector space $M^*=\mathrm{Hom}_H(M,H)$ of right
$H$-linear maps $M\rightarrow H$ becomes a left $H$-module.
The left $H$-action on
$M^*$ is defined by $(a\cdot\alpha)(m)=a\cdot\alpha(m)$
for all $a\in H$, $m\in M$ and $\alpha\in M^*$.
The pairing $\langle\cdot,\cdot\rangle\colon M^*\otimes M\rightarrow H$ is the
evaluation.
\begin{proposition}\label{prop06'}
Consider a dual pair $(M,N,\langle\cdot,\cdot\rangle)$,
as given in Definition~\ref{def02} and a differential calculus 
$(\Omega^\bullet,\wedge,\mathrm{d})$ on $H$.
\begin{enumerate}
\item[i.)] Any right connection $\dst
\in\mathrm{Con}_H(M)$ induces a left connection $\std
\in{}_H\mathrm{Con}(N)$, called the dual connection, via the formula
\begin{equation}\label{eq47}
    \mathrm{d}\langle n,m\rangle
    =\langle\std n,m\rangle
    +\langle n,\dst m\rangle
\end{equation}
for all $m\in M$ and $n\in N$.

\item[ii.)] Conversely, any left connection 
$\std
\in{}_H\mathrm{Con}(N)$
induces a dual right connection $\dst
\in\mathrm{Con}_H(M)$ via eq.(\ref{eq47}).

\item[iii.)] If $\dst
\in\mathrm{Con}_H(M)$ and
$\std
\in{}_H\mathrm{Con}(N)$ are connections dual to each other, i.e.,
such that eq.(\ref{eq47}) holds, the curvatures
$\mathrm{R}^{\dst}\colon M\rightarrow M\otimes_H\Omega^2$ and
$\mathrm{R}^\std\colon N\rightarrow\Omega^2\otimes_HN$
are dual to each other in the sense that
\begin{equation*}
    \langle n,\mathrm{R}^{\dst}m\rangle
    =\langle\mathrm{R}^\std n,m\rangle
\end{equation*}
holds for all $m\in M$ and $n\in N$.

\item[iv.)] If $\dst
\in\mathrm{Con}_H(\Omega^1)$
is a right connection and $\std
\in{}_H\mathrm{Con}(\mathfrak{X})$ is the dual left connection,
the torsions
$\mathrm{Tor}^{\dst}\colon\Omega^1\rightarrow\Omega^2$ and
$\mathrm{Tor}^\std\in\Omega^2\otimes_H\mathfrak{X}$
are dual to each other in the sense that
\begin{equation}\label{eq71}
    -\mathrm{Tor}^{\dst}(\omega)
    =\langle\mathrm{Tor}^{\std},\omega\rangle
\end{equation}
holds for all $\omega\in\Omega^1$.
\end{enumerate}
\end{proposition}
\begin{proof}
Let $m\in M$, $n\in N$ and $a\in H$.
\begin{enumerate}
\item[i.)] Assume that $\dst$ is a right connection on $M$. Then
$\std\colon N\rightarrow\Omega\otimes_HN$ is well-defined
via the pairing, since
\begin{align*}
    \langle\std n,m\cdot a\rangle
    &=\mathrm{d}\langle n,m\cdot a\rangle
    -\langle n,\dst(m\cdot a)\rangle\\
    &=\mathrm{d}(\langle n,m\rangle a)
    -\langle n,(\dst m)\cdot a\rangle
    -\langle n,m \otimes_H\mathrm{d}a\rangle\\
    &=(\mathrm{d}\langle n,m\rangle)\cdot a
    +\langle n,m\rangle\otimes_H\mathrm{d}a
    -\langle n,\dst m\rangle\cdot a
    -\langle n,m\rangle\otimes_H\mathrm{d}a\\
    &=\langle\std n,m\rangle\cdot a~.
\end{align*}
It is a left connection on $N$, since
\begin{align*}
    \langle\std(a\cdot n),m\rangle
    &=\mathrm{d}\langle a\cdot n,m\rangle
    -\langle a\cdot n,\dst m\rangle\\
    &=\mathrm{d}(a\langle n,m\rangle)
    -a\langle n,\dst m\rangle\\
    &=\mathrm{d}a\otimes_H\langle n,m\rangle
    +a\cdot\mathrm{d}\langle n,m\rangle
    -a\langle n,\dst m\rangle\\
    &=\langle\mathrm{d}a\otimes_Hn,m\rangle
    +a\langle\std n,m\rangle\\
    &=\langle\mathrm{d}a\otimes_Hn,m\rangle
    +\langle a\cdot\std n,m\rangle~.
\end{align*}

\item[ii.)] This follows in complete analogy to $i.)$.

\item[iii.)] If $\dst$ and $\std$ are right and left connections, respectively,
such that eq.(\ref{eq47}) is satisfied, then
\begin{align*}
    \langle n,\mathrm{R}^{\dst}m\rangle
    &=\langle n,\dst^\bullet(\dst(m))\rangle\\
    &=\langle n,\dst^\bullet(m_M\otimes_Hm_\Omega)\\
    &=\langle n,\dst(m_M)\wedge m_\Omega\rangle
    +\langle n,m_M\otimes_H\mathrm{d}m_\Omega\rangle\\
    &=\langle n,\dst m_M\rangle\wedge m_\Omega
    +\langle n,m_M\rangle\cdot\mathrm{d}m_\Omega\\
    &=\mathrm{d}\langle n,m_M\rangle\wedge m_\Omega
    -\langle\std n,m_M\rangle\wedge m_\Omega
    +\mathrm{d}(\langle n,m_M\rangle\cdot m_\Omega)
    -\mathrm{d}(\langle n,m_M\rangle)\wedge m_\Omega\\
    &=-\langle\std n,m_M\rangle\wedge m_\Omega
    +\mathrm{d}(\langle n,m_M\rangle\cdot m_\Omega)\\
    &=-n_\Omega\wedge\langle n_N,\dst m\rangle
    +\mathrm{d}\langle n,\dst m\rangle\\
    &=-n_\Omega\wedge\mathrm{d}\langle n_N,m\rangle
    +n_\Omega\wedge\langle\std n_N,m\rangle
    +\mathrm{d}^2\langle n,m\rangle
    -\mathrm{d}\langle\std n,m\rangle\\
    &=-n_\Omega\wedge\mathrm{d}\langle n_N,m\rangle
    +n_\Omega\wedge\langle\std n_N,m\rangle
    -\mathrm{d}(n_\Omega\cdot\langle n_N,m\rangle)\\
    &=n_\Omega\wedge\langle\std n_\mathcal{N},m\rangle
    -\mathrm{d}(n_\Omega)\cdot\langle n_N,m\rangle\\
    &=\langle n_\Omega\wedge\std n_N
    +(-1)^1\mathrm{d}n_\Omega\otimes_Hn_N,m\rangle\\
    &=\langle\mathrm{R}^\std(n),m\rangle~,
\end{align*}
where we used the short notation $\dst m=m_M\otimes_Hm_\Omega$ and
$\std n=n_\Omega\otimes_Hn_N$.

\item[iv.)] Let $\omega=\omega^ia_i\in\Omega^1$ and
$\dst\omega=\omega^i\otimes_H\omega^j\dst{}_{ij}\in\Omega^1\otimes_H\Omega^1$,
where $a_i,\dst{}_{ij}\in H$ are suitable coefficients with respect to the bases
$\omega^i$ and $\omega^i\otimes\omega^j$ of left-invariant one forms. Then
\begin{align*}
    \langle\mathrm{Tor}^{\std},\omega\rangle
    &=\langle\std{}^\bullet(\omega^i\otimes\chi_i),\omega\rangle\\
    &=\langle\omega^i\wedge\std(\chi_i)
    +(-1)^{|\omega^i|}\mathrm{d}\omega^i\otimes_H\chi_i,\omega\rangle\\
    &=\omega^i\wedge\langle\std(\chi_i),\omega\rangle
    -\mathrm{d}\omega^i\langle\chi_i,\omega^ja_j\rangle\\
    &=\omega^i\wedge(\mathrm{d}\langle\chi_i,\omega^ja_j\rangle
    -\langle\chi_i,\dst\omega\rangle)
    -\mathrm{d}(\omega^i)a_i\\
    &=\omega^i\wedge\mathrm{d}a_i
    -\omega^i\wedge\langle\chi_i,\omega^j\otimes_H\omega^k\dst{}_{jk}\rangle
    -\mathrm{d}(\omega^i)a_i\\
    &=-\omega^i\wedge\omega^k\dst{}_{ik}-\mathrm{d}(\omega^ia_i)\\
    &=-\mathrm{Tor}^{\dst}(\omega)~.
\end{align*}
\end{enumerate}
This concludes the proof of the proposition.
\end{proof}
Note that the minus sign in (\ref{eq71}) appears because torsion is of
odd degree, while curvature is of even degree.

We already noticed in Example~\ref{example02} that ${}^{\overline{M}}\mathrm{d}$
is the left connection corresponding to the right connection $\mathrm{d}^M$
and ${}^M\mathrm{d}$ is the left connection corresponding to the right
connection $\mathrm{d}^{\overline{M}}$ via composition with the braiding
$\overline{\sigma}^\mathcal{W}$. In the following proposition we prove that
these connections are furthermore dual to each other.
We also construct the connections dual to the structure constants connections of
Proposition~\ref{prop04}.
\begin{proposition}
Let $M$ be a bicovariant bimodule and
$(\Omega^\bullet,\wedge,\mathrm{d})$ a bicovariant
differential calculus on $H$.
Then
\begin{equation*}
    \mathrm{d}\langle\alpha,m\rangle
    =\langle{}^{\overline{M}}\mathrm{d}(\alpha),m\rangle
    +\langle\alpha,\mathrm{d}^M(m)\rangle~~,~~~
    \mathrm{d}\langle\alpha,m\rangle
    =\langle{}^M\mathrm{d}(\alpha),m\rangle
    +\langle\alpha,\mathrm{d}^{\overline{M}}(m)\rangle
\end{equation*}
holds for all $m\in M$ and $\alpha\in M^*$, i.e.
$\mathrm{d}^M$ and ${}^{\overline{M}}\mathrm{d}$ are dual connections,
and $\mathrm{d}^{\overline{M}}$ and ${}^M\mathrm{d}$ are dual connections.

Furthermore, the left connection
$\std{}^c\in
{}_H\mathrm{Con}(\mathfrak{X})$ dual to
$\dst{}^c$ is given by
\begin{equation}\label{eq58}
    \std{}^c(a^i\chi_i)
    =a^iC^k_{ni}\omega^n\otimes_H\chi_k
    +\mathrm{d}a^i\otimes_H\chi_i
\end{equation}
and the left connection
$\widetilde{\std}{}^c
\in{}_H\mathrm{Con}(\mathfrak{X})$ dual to
$\widetilde{\dst}{}^c$ is given by
\begin{equation*}
    \widetilde{\std}{}^c(a^i\chi_i)
    =a^i\mathbf{C}^k_{ni}\omega^n\otimes_H\chi_k
    +\mathrm{d}a^i\otimes_H\chi_i
\end{equation*}
for all $a^i\in H$, with $\{\chi_i\}$ the basis dual to $\{\omega^i\}$.
\end{proposition}
\begin{proof}
Recall that $\mathrm{d}^M$ and ${}^{\overline{M}}\mathrm{d}$ have the defining property of vanishing on right invariant elements. Then $\mathrm{d}\langle\alpha,m\rangle=\langle{}^{\overline{M}}\mathrm{d}(\alpha),m\rangle+\langle\alpha,\mathrm{d}^M(m)\rangle$ trivially holds if $\alpha$ and $m$ are right invariant, since all terms vanish. The duality then follows recalling that $M\cong M^{\mathrm{co}H}\otimes H$ and $M^*\cong H\otimes(M^*)^{\mathrm{co}H}$.
Similarly, using $M\cong H\otimes{}^{\mathrm{co}H}M$ and $M^*\cong{}^{\mathrm{co}H}(M^*)\otimes H$,
one proves that $\mathrm{d}^{\overline{M}}$ and ${}^M\mathrm{d}$ are dual.
Moreover,
$\mathrm{d}\langle\chi_i,\omega^j\rangle-\langle\chi_i,\dst{}^c(\omega^j)\rangle=\langle\std{}^c(\chi_i),\omega^j\rangle$
similarly implies that $\std{}^c$ defined in (\ref{eq58}) is the dual connection of $\dst{}^c$.
The dual connection of $\widetilde{\dst}^c$ is obtained by substituting
$C^i_{jk}$ with $\mathbf{C}^i_{jk}$.
\end{proof}

  \subsection{Sum of Connections}\label{Sec4.4}

Given two right connections on modules $M$, $N$ we construct their sum
which is a
a right connection on the tensor product module $M\otimes N$.
In the following $M$ and $N$ denote bicovariant bimodules.
We first study the relation between the canonical connections on $M$, $N$ and
$M\otimes_HN$.
\begin{lemma}
The canonical connections on $M$, $N$ and $M\otimes_HN$ are related via
\begin{equation}\label{eq12}
    \mathrm{d}^{M\otimes_HN}(m\otimes_Hn)
    =\sigma^\mathcal{W}_{23}(\mathrm{d}^M(m)\otimes_Hn)
    +m\otimes_H\mathrm{d}^N(n),
\end{equation}
for all $m\in M$ and $n\in N$.
\end{lemma}
\begin{proof}
Let $m\in M$ and $n\in N$. We have
\begin{align*}
    \sigma^\mathcal{W}_{23}(\mathrm{d}^M(m)\otimes_Hn)
    &=m_0S(m_1)\otimes_H\sigma^\mathcal{W}(\mathrm{d}m_2\otimes_Hn)\\
    &=m_0S(m_1)\otimes_Hm_2n_0S(n_1)\otimes_HS(m_3)\mathrm{d}m_4n_2\\
    &=m_0\otimes_Hn_0S(n_1)\otimes_HS(m_1)\mathrm{d}m_2n_2~,
\end{align*}
which implies
\begin{align*}
    \mathrm{d}^{M\otimes_HN}(m\otimes_Hn)
    &=(m_0\otimes_Hn_0)S(m_1n_1)\otimes_H\mathrm{d}(m_2n_2)\\
    &=m_0\otimes_Hn_0S(m_1n_1)\otimes_H\mathrm{d}(m_2)n_2
    +m_0\otimes_Hn_0S(m_1n_1)\otimes_Hm_2\mathrm{d}(n_2)\\
    &=\sigma^\mathcal{W}_{23}(\mathrm{d}^M(m)\otimes_Hn)
    +m\otimes_H\mathrm{d}^N(n)~.
\end{align*}
\end{proof}
The right hand side of equation (\ref{eq12}) can be understood as defining the
sum of the connections $\mathrm{d}^M$ and $\mathrm{d}^N$; this equals the canonical
connection on $M\otimes_HN$. Since any right connection
$\dst{}^M
\in\mathrm{Con}_H(M)$ on $M$ is equal to the sum of $\mathrm{d}^M$ and a
right $H$-linear map $\phi^M\in\mathrm{Hom}_H(M,M\otimes_H\Omega^1)$
the sum of two arbitrary right
connections $\dst{}^M
\in\mathrm{Con}_H(M)$ and
$\dst{}^N
\in\mathrm{Con}_H(N)$ 
on the bicovariant bimodules $M,N$ that we rewrite as
$\dst{}^M=\mathrm{d}^M+(\dst{}^M-\mathrm{d}^M)=\mathrm{d}^M+\phi^M$,
$\dst{}^N=\mathrm{d}^N+(\dst{}^N-\mathrm{d}^N)=\mathrm{d}^N+\phi^N$,
is then given by
\begin{equation}\label{eq13}
    \dst{}^{M\otimes_HN}
    :=\mathrm{d}^{M\otimes_HN}
    +\sigma^\mathcal{W}_{23}
    \circ((\dst{}^M-\mathrm{d}^M)\otimes_{\sigma^\mathcal{W}}\mathrm{id}_N)
    +(\mathrm{id}_M\otimes_{\sigma^\mathcal{W}}(\dst{}^N-\mathrm{d}^N))~.
\end{equation}
Since $\phi^M=\dst{}^M-\mathrm{d}^M$ and $\phi^N=\dst{}^N-\mathrm{d}^N$
are right $H$-linear maps, the tensor products $\otimes_{\sigma^\mathcal{W}}$
are well-defined and give right $H$-linear maps. It is immediate that
(\ref{eq13}) defines a right $H$-connection $M\otimes_HN\rightarrow
M\otimes_HN\otimes_H\Omega^1$ and generalizes the sum (\ref{eq12}) 
of the canonical connection.
In order to formulate a sum of connections solely from the
initial data $(\dst{}^M,\dst{}^N)$ and
without reference to the canonical connections
we employ the tensor product $\,\hat{\otimes}\,$ of rational morphisms
defined in Proposition~\ref{prop06} via the lifting $\sigma$ of $\sigma^\mathcal{W}$.
\begin{theorem}\label{thm02}
Let $M,N$ be modules of ${}_H^H\mathcal{M}_H^H$, finitely generated as right 
$H$-modules and consider right connections 
$\dst{}^M
\in\mathrm{Con}_H(M)$ 
and
$\dst{}^N
\in\mathrm{Con}_H(N)$.
Then
\begin{equation}\label{eq15}
    \dst{}^M\oplus\dst{}^N
    :=\sigma^\mathcal{W}_{23}\circ(\dst{}^M\,\hat{\otimes}\,\mathrm{id}_N)
    +\mathrm{id}_M\,\hat{\otimes}\,\dst{}^N
\end{equation}
defines a right connection $\dst{}^M\oplus\dst{}^N
\in\mathrm{Con}_H(M\otimes_HN)$.
On elements $m\otimes_Hn\in M\otimes_HN$ this reads
\begin{equation}\label{eq37}
    (\dst{}^M\oplus\dst{}^N)(m\otimes_Hn)
    =\sigma^\mathcal{W}_{23}(\dst{}^M(m_0S(m_1))\otimes_Hm_2n)
    +\dst{}^N_{-2}m_0S(\dst{}^N_{-1}m_1)\otimes_H\dst{}^N_0(m_2n)~.
\end{equation}
It follows that
\begin{equation}\label{eq16}
    \dst{}^M\oplus\dst{}^N    =\dst^{M\otimes_HN}~,
\end{equation}
i.e. definition (\ref{eq15}) coincides with definition (\ref{eq13})
of the sum of connections.
\end{theorem}
\begin{proof}
According to Proposition~\ref{prop100} $\dst{}^M$ and $\dst{}^N$ are rational
morphisms, since $M,N$ are finitely generated $H$-modules.
Then Theorem~\ref{prop06} implies that
$\dst{}^M\oplus\dst{}^N
=\sigma^\mathcal{W}_{23}\circ(\dst{}^M\,\hat{\otimes}\,\mathrm{id}_N)
+\mathrm{id}_M\,\hat{\otimes}\,\dst{}^N$
is well-defined as a linear map
$M\otimes_HN\rightarrow M\otimes_HN\otimes_H\Omega^1$.
Employing (\ref{eq36})
the explicit expression (\ref{eq37}) follows. Then, for all $m\otimes_Hn\in M\otimes_HN$
and $a\in H$,
\begin{align*}
    (\dst{}^M\oplus\dst{}^N)((m\otimes_Hn)a)
    &=\sigma^\mathcal{W}_{23}(\dst{}^M(m_0S(m_1))\otimes_Hm_2na)
    +\dst{}^N_{-2}m_0S(\dst{}^N_{-1}m_1)\otimes_H\dst{}^N_0(m_2na)\\
    &=\sigma^\mathcal{W}_{23}(\dst{}^M(m_0S(m_1))\otimes_Hm_2n)a
    +\dst{}^N_{-2}m_0S(\dst{}^N_{-1}m_1)\otimes_H\dst{}^N_0(m_2n)a\\
    &\quad+m_0S(m_1)\otimes_Hm_2n\otimes_H\mathrm{d}a\\
    &=((\dst{}^M\oplus\dst{}^N)(m\otimes_Hn))a
    +m\otimes_Hn\otimes_H\mathrm{d}a
\end{align*}
by (\ref{eq38}) and the right $H$-linearity of $\sigma^\mathcal{W}$.
This proves that $\dst{}^M\oplus\dst{}^N$ is a right connection on $M\otimes_HN$.
For the last statement we first prove that
\begin{align*}
    (\mathrm{d}^M\oplus\mathrm{d}^N)(m\otimes_Hn)
    &=(\sigma^\mathcal{W}_{23}\circ(\mathrm{d}^M\,\hat{\otimes}\,\mathrm{id}_N)
    +(\mathrm{id}_M\,\hat{\otimes}\,\mathrm{d}^N))(m\otimes_Hn)\\
    &=\sigma^\mathcal{W}_{23}(\mathrm{d}^M(m_0S(m_1))\otimes_Hm_2n)
    +\mathrm{d}^N_{-2}m_0S(\mathrm{d}^N_{-1}m_1)\otimes_H\mathrm{d}^N_0(m_2n)\\
    &=0+m_0S(m_1)\otimes_H\mathrm{d}^N(m_2n)\\
    &=m_0S(m_1)\otimes_Hm_2n_0S(m_3n_1)\otimes\mathrm{d}(m_4n_2)\\
    &=(m_0\otimes_Hn_0)S(m_1n_1)\otimes\mathrm{d}(m_2n_2)\\
    &=\mathrm{d}^{M\otimes_HN}(m\otimes_Hn)
\end{align*}
for all $m\in M$ and $n\in N$, using $\mathrm{d}^M(M^{\mathrm{co}H})=0$
and that $\mathrm{d}^N$ is left coinvariant
(cf. Proposition~\ref{prop03}).
Then (\ref{eq16}) follows from the linearity of
$\,\hat{\otimes}\,$ and Theorem~\ref{prop06}.
This concludes the proof of the theorem.
\end{proof}
For two connections $\dst{}^M\in\mathrm{Con}_H(M)$ and $\dst{}^N\in\mathrm{Con}_H(N)$
we call (\ref{eq15}) or equivalently (\ref{eq13}) the \textit{sum of $\dst{}^M$
and $\dst{}^N$}, written $\dst{}^M\oplus\dst{}^N\in\mathrm{Con}_H(M\otimes_HN)$.
If $M=N$ and $\dst{}^M=\dst{}^N$ we abbreviate the sum of connections by
\begin{equation*}
    \dst{}^M(\cdot\otimes_H\cdot):=\dst{}^M\oplus\dst{}^M
    \in\mathrm{Con}_H(M^{\otimes_H2})
\end{equation*}
and similarly for higher tensor products.
In analogy to (\ref{eq15})  for
right $H$-linear maps $\phi\in\mathrm{Hom}_H(M,M\otimes_H\Omega^1)$
and $\psi\in\mathrm{Hom}_H(N,N\otimes_H\Omega^1)$
we define the \textit{sum of $H$-linear maps} $\phi\oplus\psi\in\mathrm{Hom}_H(M\otimes_HN,
M\otimes_HN\otimes_H\Omega^1)$
by
\begin{equation}\label{eq90}
    \phi\oplus\psi
    :=\sigma^\mathcal{W}_{23}\circ(\phi\otimes_{\sigma^\mathcal{W}}\mathrm{id}_N)
    +\mathrm{id}_M\otimes_{\sigma^\mathcal{W}}\psi~.
\end{equation}
If $M=N$ and $\phi=\psi$, we simply write $\phi(\cdot\otimes_H\cdot)
:=\phi\oplus\phi\colon M^{\otimes_H2}\rightarrow M^{\otimes_H2}\otimes\Omega^1$.

\begin{remark}[Bimodule Connections]\label{rem01}
  Usually for the sum of connections extra properties are required,
  typically bimodule connections are considered  \cite{D-VMichor}.
The sum of two 
bimodule right connections $\dst{}^M\in\mathrm{Con}_H(M)$ and
$\dst{}^N\in\mathrm{Con}_H(N)$ is the 
 bimodule right connection
$$
\dst{}^{M\otimes_HN}_{BM}(m\otimes_Hn)
:=\sigma^\mathcal{W}_{23}(\dst{}^M(m)\otimes_Hn)
+m\otimes_H\dst{}^N(n)~,
$$
where $m\in M$ and $n\in N$.
 We
  have seen that in the context of bicovariant bimodules arbitrary
  right connections can be summed. 
If $M, N$ are bicovariant bimodules and $\dst{}^N$ is furthermore left
$H$-colinear the connection $
\dst{}^{M\otimes_HN}_{BM}$
coincides with the sum of connections $\dst{}^{M\otimes_HN}$
introduced in (\ref{eq15}) since
  \begin{align}
    \dst{}^{M\otimes_HN}(m\otimes_Hn)
    &=(\sigma^\mathcal{W}_{23}\circ(\dst{}^M\,\hat{\otimes}\,\mathrm{id}_N)
    +\mathrm{id}_M\,\hat{\otimes}\,\dst{}^N)(m\otimes_Hn)\nonumber\\
    &=\sigma^\mathcal{W}_{23}(\dst{}^M(m_0S(m_1))\otimes_Hm_2n)
    +\dst{}^N_{-2}m_0S(\dst{}^N_{-1}m_1)\otimes_H\dst{}^N_0(m_2n)\nonumber\\
    &=\sigma^\mathcal{W}_{23}(\dst{}^M(m)\otimes_Hn)
    +\sigma^\mathcal{W}_{23}(m_0\otimes_H\mathrm{d}S(m_1)\otimes_Hm_2n)\nonumber\\
    &~~+m_0S(m_1)m_2\otimes_H\dst{}^N(n)
    +\sigma^\mathcal{W}_{23}(m_0S(m_1)\otimes_H
    \mathrm{d}m_2\otimes_Hn)\nonumber\\
    &=\sigma^\mathcal{W}_{23}(\dst{}^M(m)\otimes_Hn)
    +m\otimes_H\dst{}^N(n)\nonumber\\
    &= 
      \dst{}^{M\otimes_HN}_{BM}(m\otimes_H n)\label{sumBM}
 \end{align}
for all $m\in M$ and $n\in N$ by the right Leibniz rule of $\dst{}^M$, the
additional left Leibniz rule (\ref{eq64}) of the bimodule right connection
$\dst{}^N$ and $H$-colinearity: $\dst{}^N_{-1}\otimes\dst{}^N_0=1\otimes\dst{}^N$.

In particular, since the canonical connections are $H$-bicolinear, the sum of connections in (\ref{eq12}) equals the sum of bimodule connections.
\end{remark}

\begin{remark}[Classical sum of connections]
  Let $H$ be a commutative Hopf algebra (e.g.
   the coordinate algebra of a compact Lie group $G$) and $M,N$ be \textit{symmetric}
bicovariant bimodules, i.e. $am=ma$ and $an=na$ for all $a\in H$, $m\in M$ and
$n\in N$. In particular Woronowicz's braiding $\sigma^\mathcal{W}_{M,N}$ coincides with the usual flip operator
$\sigma^\mathrm{flip}\colon M\otimes_HN\rightarrow N\otimes_HM$. Any right connection $\dst{}^M\in\mathrm{Con}_H(M)$
is a bimodule connection in this case, since
$    \dst{}^M(am)
    =\dst{}^M(ma)
    =\dst{}^M(m)a+m\otimes_H\mathrm{d}a
    =a\dst{}^M(m)+\sigma^\mathcal{W}_{\Omega^1,M}(\mathrm{d}a\otimes_Hm)
    $
    for all $a\in H$ and $m\in M$.
    
Furthermore, when considering the sum
$\dst{}^{M\otimes_HN}(m\otimes_Hn)$ in equation \eqref{sumBM} we have
that the equality
$\dst{}^N_{-2}m_0S(\dst{}^N_{-1}m_1)\otimes_H\dst{}^N_0(m_2n)=
   m_0S(m_1)\otimes_H\dst{}^N(m_2n)
$ holds because of commutativity, so that  there is no need to require
$H$-colinearity of $ \dst{}^N$. Therefore the sum of connections $\dst{}^{M\otimes_HN}$
becomes the usual sum of connections of classical differential geometry:
$ (\dst{}^M\oplus\dst{}^N)(m\otimes_Hn)=
\sigma^\mathrm{flip}_{23}(\dst{}^M(m)\otimes_Hn)
    +m\otimes_H\dst{}^N(n)$.

  \end{remark}

\subsection{Sum of Curvatures}\label{Sec4.5}

In this section we give an explicit expression of the curvature of the sum of
connections in terms of the curvature of the initial connections.
We fix a bicovariant differential calculus $(\Omega^\bullet,\wedge,\mathrm{d})$ on
$H$ and consider bicovariant bimodules $M,M',\ldots,N,N',\ldots$, which are 
finitely generated as right $H$-modules.
Recalling (\ref{eq69}) and (\ref{defcurvatureR}) the curvature
of a right connection $\dst\in\mathrm{Con}_H(M)$ on $M$ reads
\begin{equation}\label{eq61}
    \mathrm{R}^{\dst}
    =\dst{}^\bullet\circ\dst
    =\wedge_{23}\circ(\dst\,\hat{\otimes}\,\mathrm{id}_{\Omega^1})
    \circ\dst
    +(\mathrm{id}_M\,\hat{\otimes}\,\mathrm{d})\circ\dst~.
\end{equation}

Before we attack the sum of curvatures we study the compatibility
properties of the tensor product $\,\hat{\otimes}\,$ with the composition
of rational morphisms.
Recall that for right $H$-linear maps $\phi\in\mathrm{Hom}_H(M',M'')$,
$\phi'\in\mathrm{Hom}_H(M,M')$, $\psi\in\mathrm{Hom}_H(N',N'')$ and
$\psi'\in\mathrm{Hom}_H(N,N')$ the equality 
\begin{equation*}
    (\phi\otimes_{\sigma^\mathcal{W}}\psi)\circ(\phi'\otimes_{\sigma^\mathcal{W}}\psi')
    =(\phi\circ{}_\alpha\phi')\otimes_{\sigma^\mathcal{W}}({}^\alpha\psi\circ\psi')
\end{equation*}
holds. More in general, recalling the lifting $\sigma$ defined in (\ref{sigma}) we have
\begin{lemma}\label{lemma05}
Let $\phi\in\mathrm{HOM}(M',M'')$, $\phi'\in\mathrm{HOM}(M,M')$,
$\psi\in\mathrm{HOM}(N',N'')$ and $\psi'\in\mathrm{HOM}(N,N')$.
Then
\begin{equation*}
    (\phi\,\hat{\otimes}\,\mathrm{id}_{N'})\circ (\mathrm{id}_{M'}\,\hat{\otimes}\,\psi')=(\phi\,\hat{\otimes}\,\psi')
     ~~\mbox{
      \it{and} }~~
     (\mathrm{id}_M\,\hat{\otimes}\,\psi)\circ(\mathrm{id}_M\,\hat{\otimes}\,\psi')
    = (\mathrm{id}_M\,\hat{\otimes}\,\psi\circ \psi')~.
\end{equation*}
If  $\phi'$ is right $H$-linear or right $H$-colinear we
further have
\begin{equation*}
    (\phi\,\hat{\otimes}\,\mathrm{id}_{N})\circ(\phi'\,\hat{\otimes}\,\mathrm{id}_{N})
    =(\phi\circ\phi')\,\hat{\otimes}\,\mathrm{id}_{N}~.
\end{equation*}
If  $\phi'$ is right $H$-linear or, if $\phi'$ is right $H$-colinear
and $\psi$ is left $H$-colinear, we have
\begin{equation*}
    (\mathrm{id}_{M'}\,\hat{\otimes}\,\psi)\circ(\phi'\,\hat{\otimes}\,\mathrm{id}_{N'})
    ={}_\alpha\phi'\,\hat{\otimes}\,{}^\alpha\psi~.
\end{equation*}
Finally, the last two equalities hold also if 
$M'=M\otimes_H\Omega^1$,
$\phi'=\dst\colon M\rightarrow M\otimes_H\Omega^1$ is a right connection
and, for the last one, if $\psi$ is left $H$-colinear.
\end{lemma}
\begin{proof}
The first equality easily follows from that for $\otimes_\sigma$ in equation
\eqref{eq20} and recalling that $\phi\otimes_\sigma
\mathrm{id}_{N'}=(\phi\otimes_\sigma \mathrm{id}_{N'})\circ
\pi_H^{M,N'}$. We prove at once this and all the other equalities for
$\phi,\phi',\psi,\psi'$ rational morphisms  by
computing
\begin{align*}
    ((\phi\,\hat{\otimes}\,\psi)\circ(\phi'\,\hat{\otimes}\,\psi'))(m\otimes_Hn)
    &=(\phi\,\hat{\otimes}\,\psi)
    (\phi'(\psi'_{-2}m_0S(\psi'_{-1}m_1))\otimes_H\psi'_0(m_2n))\\
    &=\phi(\psi_{-2}\phi'(\psi'_{-2}m_0S(\psi'_{-1}m_1))_0
    S(\psi_{-1}\phi'(\psi'_{-2}m_0S(\psi'_{-1}m_1))_1))\\
    &\quad\otimes_H\psi_0(\phi'(\psi'_{-2}m_0S(\psi'_{-1}m_1))_2\psi'_0(m_2n))\\
    &=\phi\bigg(\psi_{-2}\phi'_0\bigg(\psi'_{-2}m_0S(\psi'_{-1}m_1)\bigg)
    S(\psi_{-1}\phi'_1)\bigg)
    \otimes_H\psi_0(\phi'_2\psi'_0(m_2n))~,
\end{align*}
and
\begin{align*}
    ((\phi\circ{}_\alpha\phi')&\,\hat{\otimes}\,({}^\alpha\psi\circ\psi'))(m\otimes_Hn)\\
    &=\big((\phi\circ(\psi_{-2}\phi'_0S(\psi_{-1}\phi'_1)))
    \,\hat{\otimes}\,((\psi_0\phi'_2)\circ\psi')\big)(m\otimes_Hn)\\
    &=\phi((\psi_{-2}\phi'_0S(\psi_{-1}\phi'_1))(
    [(\psi_0\phi'_2)\circ\psi']_{-2}m_0S([(\psi_0\phi'_2)\circ\psi']_{-1}m_1)))
    \otimes_H[(\psi_0\phi'_2)\circ\psi']_0(m_2n)\\
    &=\phi((\psi_{-4}\phi'_0S(\psi_{-3}\phi'_1))(
    \psi_{-2}\phi'_2\psi'_{-2}m_0S(\psi_{-1}\phi'_3\psi'_{-1}m_1)))
    \otimes_H[(\psi_0\phi'_4)\circ\psi'_0](m_2n)\\
    &=\phi(\psi_{-4}\phi'_0(S(\psi_{-3}\phi'_1)
    \psi_{-2}\phi'_2\psi'_{-2}m_0S(\psi_{-1}\phi'_3\psi'_{-1}m_1)))
    \otimes_H[(\psi_0\phi'_4)\circ\psi'_0](m_2n)\\
    &=\phi\bigg(\psi_{-2}\phi'_0\bigg(
    \psi'_{-2}m_0S(\psi'_{-1}m_1)S(\psi_{-1}\phi'_1)\bigg)\bigg)
    \otimes_H\psi_0(\phi'_2\psi'_0(m_2n))
\end{align*}
for all $m\in M$ and $n\in N$. We see that these expressions coincide
if $\phi'$ is right $H$-linear or if $\phi'$ is right $H$-colinear and
$\psi$ is left $H$-colinear. As special cases we obtain the equalities
displayed in the lemma.

Assume that $M'=M\otimes_H\Omega^1$ and
$\phi'=\dst\colon M\rightarrow M\otimes_H\Omega^1$ is a right connection.
Note that $\mathrm{d}^M\,\hat{\otimes}\,\psi'=0$ for all linear maps
$\psi'$ because the canonical connection $\mathrm{d}^M$ vanishes on $M^{\mathrm{co}H}$.
Then
\begin{align*}
    (\phi\,\hat{\otimes}\,\psi)\circ(\dst\,\hat{\otimes}\,\mathrm{id}_{N'})
    &=(\phi\,\hat{\otimes}\,\psi)\circ((\dst-\mathrm{d}^M)\,\hat{\otimes}\,\mathrm{id}_{N'})\\
    &=(\phi\circ{}_\alpha(\dst-\mathrm{d}^M))\,\hat{\otimes}\,{}^\alpha\psi\\
    &=(\phi\circ{}_\alpha\dst)\,\hat{\otimes}\,{}^\alpha\psi
    -(\phi\circ{}_\alpha\mathrm{d}^M)\,\hat{\otimes}\,{}^\alpha\psi\\
    &=(\phi\circ{}_\alpha\dst)\,\hat{\otimes}\,{}^\alpha\psi
    -(\phi\circ(\psi_{-2}\cdot\mathrm{d}^M\cdot S(\psi_{-1})))
    \,\hat{\otimes}\,\psi_0\\
    &=(\phi\circ{}_\alpha\dst)\,\hat{\otimes}\,{}^\alpha\psi~,
\end{align*}
where in the last two equations we have used that $\mathrm{d}^M$ is right $H$-colinear and
assumed that $\psi$ is left $H$-colinear. The final part of the lemma follows again by
specializing this equality.
\end{proof}

\begin{theorem}\label{SOC}
Consider two right connections $\dst{}^M
\in\mathrm{Con}_H(M)$ and
$\dst{}^N
\in\mathrm{Con}_H(N)$.
The curvature of their sum $\dst{}^M\oplus\dst{}^N
\in\mathrm{Con}_H(M\otimes_HN)$ defined in Theorem~\ref{thm02}
is given by
\begin{equation}\label{eq63}
\begin{split}
    \mathrm{R}^{\dst{}^M\oplus\dst{}^N}
    &=\sigma^\mathcal{W}_{23}\circ(\mathrm{R}^{\dst{}^M}\otimes_{\sigma^\mathcal{W}}\mathrm{id}_M)
    +\mathrm{id}_M\otimes_{\sigma^\mathcal{W}}\mathrm{R}^{\dst{}^N}\\
    &\quad+\wedge_{34}\circ\sigma^\mathcal{W}_{23}
    \circ \big((\dst{}^M-\mathrm{d}^M)\otimes_{\sigma^\mathcal{W}}(\dst{}^N-\mathrm{d}^N)\big)\\
    &\quad+\wedge_{34}\circ\big(\mathrm{id}_M
    \otimes_{\sigma^\mathcal{W}}(\dst{}^N-\mathrm{d}^N)\otimes_{\sigma^\mathcal{W}}\mathrm{id}_{\Omega^1}\big)
    \circ\sigma^\mathcal{W}_{23}
    \circ((\dst{}^M-\mathrm{d}^M)\otimes_{\sigma^\mathcal{W}}\mathrm{id}_N)~,
\end{split}    
\end{equation}
where all addends are right $H$-linear maps 
$M\otimes_HN\rightarrow M\otimes_HN\otimes_H\Omega^2$.
\end{theorem}
\begin{proof}
We split the proof into seven parts. 
\begin{enumerate}
\item[i.)] Firstly, by (\ref{eq61}) and observing that
  $\dst{}^M\,\hat{\otimes}\,\mathrm{id}$ is a right $H$-linear map we obtain
\allowdisplaybreaks
\begin{align}
    \mathrm{R}^{\dst{}^M\oplus\dst{}^N}
    &=\wedge_{34}
    \circ[[\sigma^\mathcal{W}_{23}\circ(\dst{}^M\,\hat{\otimes}\,\mathrm{id}_N)]
    \otimes_{\sigma^\mathcal{W}}\mathrm{id}_{\Omega^1}]
    \circ\sigma^\mathcal{W}_{23}
    \circ(\dst{}^M\,\hat{\otimes}\,\mathrm{id}_N) \label{SOC1} \\
    &\quad+\wedge_{34}
    \circ[[\sigma^\mathcal{W}_{23}\circ(\dst{}^M\,\hat{\otimes}\,\mathrm{id}_N)]
    \otimes_{\sigma^\mathcal{W}}\mathrm{id}_{\Omega^1}]
    \circ(\mathrm{id}_M\,\hat{\otimes}\,\dst{}^N) \label{SOC2} \\ 
    &\quad+\wedge_{34}
    \circ[(\mathrm{id}_M\,\hat{\otimes}\,\dst{}^N)\,\hat{\otimes}\,\mathrm{id}_{\Omega^1}]
    \circ\sigma^\mathcal{W}_{23}
    \circ(\dst{}^M\,\hat{\otimes}\,\mathrm{id}_N) \label{SOC3} \\
    &\quad+\wedge_{34}
    \circ[(\mathrm{id}_M\,\hat{\otimes}\,\dst{}^N)\,\hat{\otimes}\,\mathrm{id}_{\Omega^1}]
    \circ(\mathrm{id}_M\,\hat{\otimes}\,\dst{}^N) \label{SOC4} \\
    &\quad+(\mathrm{id}_{M\otimes_HN}\,\hat{\otimes}\,\mathrm{d})
    \circ\sigma^\mathcal{W}_{23}
    \circ(\dst{}^M\,\hat{\otimes}\,\mathrm{id}_N) \label{SOC5} \\
    &\quad+(\mathrm{id}_{M\otimes_HN}\,\hat{\otimes}\,\mathrm{d})
    \circ(\mathrm{id}_M\,\hat{\otimes}\,\dst{}^N)~. \label{SOC6} 
\end{align}
In the rest of the proof we identify these six terms with the right hand side of
(\ref{eq63}).

\item[ii.)] Utilizing the associativity of $\,\hat{\otimes}\,$ (Proposition~\ref{prop06}),
Lemma~\ref{lemma05} and the fact that $\,\hat{\otimes}\,$ is a morphism of bicovariant
bimodules (again Proposition~\ref{prop06}) it follows that
\begin{align*}
    (\ref{SOC4})+(\ref{SOC6})
    &=\wedge_{34}
    \circ[\mathrm{id}_M\,\hat{\otimes}\,
    (\dst{}^N\,\hat{\otimes}\,\mathrm{id}_{\Omega^1})]
    \circ(\mathrm{id}_M\,\hat{\otimes}\,\dst{}^N)
    +(\mathrm{id}_{M\otimes_HN}\,\hat{\otimes}\,\mathrm{d})
    \circ(\mathrm{id}_M\,\hat{\otimes}\,\dst{}^N)\\
    &=\wedge_{34}
    \circ[(\mathrm{id}_M\circ{}_\alpha\mathrm{id}_M)\,\hat{\otimes}\,
    ({}^\alpha(\dst{}^N\,\hat{\otimes}\,\mathrm{id}_{\Omega^1})
    \circ\dst{}^N)]
    +[\mathrm{id}_{M}\,\hat{\otimes}\,(\mathrm{id}_N\,\hat{\otimes}\,\mathrm{d})]
    \circ(\mathrm{id}_M\,\hat{\otimes}\,\dst{}^N)\\
    &=\wedge_{34}
    \circ[[\mathrm{id}_M\circ(\dst{}^N_{-2}\cdot\mathrm{id}_M\cdot S(\dst{}^N_{-1}))]
    \,\hat{\otimes}\,
    [(\dst{}^N_0\,\hat{\otimes}\,\mathrm{id}_{\Omega^1})
    \circ\dst{}^N]]
    +(\mathrm{id}_{M}\circ{}_\alpha\mathrm{id}_M)
    \,\hat{\otimes}\,({}^\alpha(\mathrm{id}_N\,\hat{\otimes}\,\mathrm{d})\circ\dst{}^N)\\
    &=\wedge_{34}
    \circ[\mathrm{id}_M
    \,\hat{\otimes}\,
    [(\dst{}^N\,\hat{\otimes}\,\mathrm{id}_{\Omega^1})
    \circ\dst{}^N]]
    +\mathrm{id}_{M}
    \,\hat{\otimes}\,[(\mathrm{id}_N\,\hat{\otimes}\,\mathrm{d})\circ\dst{}^N]\\
    &=\mathrm{id}_M\,\hat{\otimes}\,[\wedge_{23}\circ
    (\dst{}^N\,\hat{\otimes}\,\mathrm{id}_{\Omega^1})
    \circ\dst{}^N
    +(\mathrm{id}_N\,\hat{\otimes}\,\mathrm{d})\circ\dst{}^N]\\
    &=\mathrm{id}_M\otimes_{\sigma^\mathcal{W}}\mathrm{R}^{\dst{}^N}~.
\end{align*}

\item[iii.)] 
  For $\phi\in\mathrm{HOM}(M,M\otimes_H\Omega^1)$  a rational morphism,
the left $H$-linearity of $\sigma^\mathcal{W}$ implies the identity
\begin{equation*}
    (\phi\,\hat{\otimes}\,\mathrm{id}_N\,\hat{\otimes}\,\mathrm{id}_{\Omega^1})
    \circ\sigma^\mathcal{W}_{23}
    =\sigma^\mathcal{W}_{34}
    \circ(\phi\,\hat{\otimes}\,\mathrm{id}_{\Omega^1}\,\hat{\otimes}\,\mathrm{id}_N)
  \end{equation*}
  of right $H$-linear maps
$M\otimes_H\Omega^1\otimes_HN\rightarrow
M\otimes\Omega^1\otimes_H\Omega^1\otimes_HN$.
Recalling that
$\wedge=\mathrm{id}-\overline\sigma^\mathcal{W}$ and the
braid relation $\sigma^\mathcal{W}_{12}\circ \sigma^\mathcal{W}_{22}\circ
\sigma^\mathcal{W}_{12}=\sigma^\mathcal{W}_{23}\circ\sigma^\mathcal{W}_{12}\circ
\sigma^\mathcal{W}_{23}$ it is also easy to see that 
\begin{equation*}
    \wedge_{23}\circ\sigma^\mathcal{W}_{\Omega^1\otimes_H\Omega^1,M}=\sigma^\mathcal{W}_{\Omega^2,M}\circ\wedge_{12}
\end{equation*}
as an equation of bicovariant bimodule morphisms
$\Omega^1\otimes_H\Omega^1\otimes_HM
\rightarrow M\otimes_H\Omega^2$. These
relations, with $\phi=\dst{}^M$, imply
\begin{align*}
    (\ref{SOC1})
    &=\wedge_{34}
    \circ[[\sigma^\mathcal{W}_{23}\circ(\dst{}^M\,\hat{\otimes}\,\mathrm{id}_N)]
    \otimes_{\sigma^\mathcal{W}}\mathrm{id}_{\Omega^1}]
    \circ\sigma^\mathcal{W}_{23}
    \circ(\dst{}^M\,\hat{\otimes}\,\mathrm{id}_N)\\
    &=\wedge_{34}\circ\sigma^\mathcal{W}_{23}
    \circ[(\dst{}^M\,\hat{\otimes}\,\mathrm{id}_N)
    \otimes_{\sigma^\mathcal{W}}\mathrm{id}_{\Omega^1}]
    \circ\sigma^\mathcal{W}_{23}
    \circ(\dst{}^M\,\hat{\otimes}\,\mathrm{id}_N)\\
    &=\wedge_{34}\circ\sigma^\mathcal{W}_{23}\circ\sigma^\mathcal{W}_{34}
    \circ[(\dst{}^M\,\hat{\otimes}\,\mathrm{id}_{\Omega^1})
    \otimes_{\sigma^\mathcal{W}}\mathrm{id}_N]
    \circ(\dst{}^M\,\hat{\otimes}\,\mathrm{id}_N)\\
    &=\wedge_{34}\circ\sigma^\mathcal{W}_{23,4}
    \circ[(\dst{}^M\,\hat{\otimes}\,\mathrm{id}_{\Omega^1})
    \otimes_{\sigma^\mathcal{W}}\mathrm{id}_N]
    \circ(\dst{}^M\,\hat{\otimes}\,\mathrm{id}_N)\\
    &=\sigma^\mathcal{W}_{23}\circ\wedge_{23}\circ
    [(\dst{}^M\,\hat{\otimes}\,\mathrm{id}_{\Omega^1})
    \,\hat{\otimes}\,\mathrm{id}_N]
    \circ(\dst{}^M\,\hat{\otimes}\,\mathrm{id}_N)\\
    &=\sigma^\mathcal{W}_{23}\circ\wedge_{23}\circ
    [[(\dst{}^M\,\hat{\otimes}\,\mathrm{id}_{\Omega^1})\circ\dst{}^M]
    \,\hat{\otimes}\,\mathrm{id}_N]\\
    &=\sigma^\mathcal{W}_{23}\circ
    [[\wedge_{23}\circ[(\dst{}^M\,\hat{\otimes}\,\mathrm{id}_{\Omega^1})\circ\dst{}^M]]
    \,\hat{\otimes}\,\mathrm{id}_N]~.
\end{align*}

\item[iv.)] Let $m\in M$ and $n\in N$. We introduce the short notation
$\dst{}^M(m_0S(m_1))\otimes m_2=:\mathcal{M}\otimes_H\Theta\otimes m_2$, where we omitted writing explicitly the dependence on $m_0S(m_1)$ in $\mathcal{M}\otimes_H\Theta$.
From
\begin{align*}
    [\sigma^\mathcal{W}_{23}\circ[((\mathrm{id}_M\,\hat{\otimes}\,\mathrm{d})
    \circ&\dst{}^M)\,\hat{\otimes}\,\mathrm{id}_N]](m\otimes_Hn)\\
    &=\sigma^\mathcal{W}_{23}[
    (\mathrm{id}_M\,\hat{\otimes}\,\mathrm{d})(\dst{}^M(m_0S(m_1)))\otimes_Hm_2n]\\
    &=\sigma^\mathcal{W}_{23}(
    \mathcal{M}_0S(\mathcal{M}_1)
    \otimes_H\mathrm{d}(\mathcal{M}_2\Theta)
    \otimes_Hm_2n)\\
    &=\mathcal{M}_0S(\mathcal{M}_1)
    \otimes_H\mathcal{M}_2\Theta_{-2}(m_2n)_0
    S(\mathcal{M}_3\Theta_{-1}(m_2n)_1)
    \otimes_H\mathrm{d}(\mathcal{M}_4\Theta_0)(m_2n)_2\\
    &=\mathcal{M}_0
    \otimes_H\Theta_{-2}(m_2n)_0
    S(\mathcal{M}_1\Theta_{-1}(m_2n)_1)
    \otimes_H\mathrm{d}(\mathcal{M}_2\Theta_0)(m_2n)_2
\end{align*}
and
\begin{align*}
    [\wedge_{34}\circ\sigma^\mathcal{W}_{23}
    \circ&(\dst{}^M\,\hat{\otimes}\,\mathrm{id}_{N\otimes_H\Omega^1})
    \circ(\mathrm{id}_M\,\hat{\otimes}\,\mathrm{d}^N)](m\otimes_Hn)\\
    &=[\wedge_{34}\circ\sigma^\mathcal{W}_{23}
    \circ(\dst{}^M\,\hat{\otimes}\,\mathrm{id}_{N\otimes_H\Omega^1})]
    (m_0S(m_1)\otimes_H\mathrm{d}^N(m_2n))\\
    &=[\wedge_{34}\circ\sigma^\mathcal{W}_{23}]
    (\dst{}^M(m_0S(m_1))\otimes_H(m_2n)_0S((m_2n)_1)\otimes_H\mathrm{d}((m_2n)_2))\\
    &=\mathcal{M}\otimes_H\Theta_{-2}(m_2n)_0S((m_2n)_1)
    S(\Theta_{-1})\otimes_H\Theta_0
    \wedge\mathrm{d}((m_2n)_2)
\end{align*}
it follows that
\allowdisplaybreaks
\begin{align*}
    (\ref{SOC5})
    &=[(\mathrm{id}_{M\otimes_HN}\,\hat{\otimes}\,\mathrm{d})
    \circ\sigma^\mathcal{W}_{23}
    \circ(\dst{}^M\,\hat{\otimes}\,\mathrm{id}_N)](m\otimes_Hn)\\
    &=(\mathrm{id}_{M\otimes_HN}\,\hat{\otimes}\,\mathrm{d})
    (\sigma^\mathcal{W}_{23}(\dst{}^M(m_0S(m_1))\otimes_Hm_2n))\\
    &=(\mathrm{id}_{M\otimes_HN}\,\hat{\otimes}\,\mathrm{d})[
    \mathcal{M}\otimes_H\Theta_{-2}(m_2n)_0
    S(\Theta_{-1}(m_2n)_1)
    \otimes_H\Theta_0(m_2n)_2]\\
    &=\mathcal{M}_0\otimes_H\Theta_{-2}(m_2n)_0
    S(\Theta_{-1}(m_2n)_1)S(\mathcal{M}_1)
    \otimes_H\mathrm{d}(\mathcal{M}_2\Theta_0(m_2n)_2)\\
    &=\mathcal{M}_0\otimes_H\Theta_{-2}(m_2n)_0
    S(\mathcal{M}_1\Theta_{-1}(m_2n)_1)
    \otimes_H\mathrm{d}(\mathcal{M}_2\Theta_0)(m_2n)_2\\
    &\quad-\mathcal{M}\otimes_H\Theta_{-2}(m_2n)_0
    S(\Theta_{-1}(m_2n)_1)
    \otimes_H\Theta_0\wedge\mathrm{d}((m_2n)_2)\\
    &=[\sigma^\mathcal{W}_{23}\circ[((\mathrm{id}_M\,\hat{\otimes}\,\mathrm{d})
    \circ\dst{}^M)\,\hat{\otimes}\,\mathrm{id}_N]](m\otimes_Hn)\\
    &\quad-[\wedge_{34}\circ\sigma^\mathcal{W}_{23}
    \circ(\dst{}^M\,\hat{\otimes}\,\mathrm{id}_{N\otimes_H\Omega^1})
    \circ(\mathrm{id}_M\,\hat{\otimes}\,\mathrm{d}^N)](m\otimes_Hn)\\
    &=[\sigma^\mathcal{W}_{23}\circ[((\mathrm{id}_M\,\hat{\otimes}\,\mathrm{d})
    \circ\dst{}^M)\,\hat{\otimes}\,\mathrm{id}_N]](m\otimes_Hn)\\
    &\quad-[\wedge_{34}\circ\sigma^\mathcal{W}_{23}
    \circ(\dst{}^M\,\hat{\otimes}\,\mathrm{d}^N)](m\otimes_Hn)~.
\end{align*}

\item[v.)] Adding the results of iii.) and iv.) gives
\begin{align*}
    \sigma^\mathcal{W}_{23}\circ(\mathrm{R}^{\dst{}^M}\,\hat{\otimes}\,\mathrm{id}_N)
    &=\sigma^\mathcal{W}_{23}\circ
    [[\wedge_{23}\circ(\dst{}^M\,\hat{\otimes}\,\mathrm{id}_{\Omega^1})\circ\dst{}^M
    +(\mathrm{id}_M\,\hat{\otimes}\,\mathrm{d})\circ\dst{}^M]
    \,\hat{\otimes}\,\mathrm{id}_N]\\
    &=(\ref{SOC1})+(\ref{SOC5})
    +\wedge_{34}\circ\sigma^\mathcal{W}_{23}
    \circ(\dst{}^M\,\hat{\otimes}\,\mathrm{d}^N).
\end{align*}

\item[vi.)] Since $\mathrm{d}^M\,\hat{\otimes}\,\mathrm{id}_N=0$ it follows that 
\begin{align*}
    \wedge_{34}\circ&\sigma^\mathcal{W}_{23}
    \circ((\dst{}^M-\mathrm{d}^M)\otimes_{\sigma^\mathcal{W}}(\dst{}^N-\mathrm{d}^N))\\
    &\quad+\wedge_{34}\circ(\mathrm{id}_M
    \otimes_{\sigma^\mathcal{W}}(\dst{}^N-\mathrm{d}^N)\otimes_{\sigma^\mathcal{W}}\mathrm{id}_{\Omega^1})
    \circ\sigma^\mathcal{W}_{23}
    \circ((\dst{}^M-\mathrm{d}^M)\otimes_{\sigma^\mathcal{W}}\mathrm{id}_N)\\
    &=\wedge_{34}\circ\sigma^\mathcal{W}_{23}
    \circ((\dst{}^M-\mathrm{d}^M)\,\hat{\otimes}\,(\dst{}^N-\mathrm{d}^N))\\
    &\quad+\wedge_{34}\circ(\mathrm{id}_M
    \,\hat{\otimes}\,(\dst{}^N-\mathrm{d}^N)
    \,\hat{\otimes}\,\mathrm{id}_{\Omega^1})
    \circ\sigma^\mathcal{W}_{23}
    \circ((\dst{}^M-\mathrm{d}^M)\,\hat{\otimes}\,\mathrm{id}_N)\\
    &=\wedge_{34}\circ\sigma^\mathcal{W}_{23}
    \circ(\dst{}^M\,\hat{\otimes}\,(\dst{}^N-\mathrm{d}^N))
    +\wedge_{34}\circ(\mathrm{id}_M
    \,\hat{\otimes}\,(\dst{}^N
    \,\hat{\otimes}\,\mathrm{id}_{\Omega^1}))
    \circ\sigma^\mathcal{W}_{23}
    \circ(\dst{}^M\,\hat{\otimes}\,\mathrm{id}_N)\\
    &=(\ref{SOC2})+(\ref{SOC3})
    -\wedge_{34}\circ\sigma^\mathcal{W}_{23}
    \circ(\dst{}^M\,\hat{\otimes}\,\mathrm{d}^N)~.
\end{align*}

\item[vii.)] Combining ii.), v.) and vi.), the claim (\ref{eq63}) follows.
\end{enumerate}
\end{proof}
\begin{remark}[Classical Sum of Curvatures]
For a commutative Hopf algebra $H$ and two right connections $\dst{}^M\in\mathrm{Con}_H(M)$,
$\dst{}^N\in\mathrm{Con}_H(N)$ on symmetric bicovariant bimodules
$M,N$ we verify that (\ref{eq63}) gives back the formulas
known from differential geometry. Let $m\in M$ and $n\in N$.
We first recall that right $H$-linearity is equivalent to left $H$-linearity in
this setting. Then $\sigma^\mathcal{W}=\sigma^{\mathrm{flip}}$ becomes
the tensor product of right $H$-linear maps $\phi$ and $\psi$ between symmetric
bicovariant bimodules:
$\phi\otimes_{\sigma^\mathcal{W}}\psi=\phi\otimes_{\sigma^\mathrm{flip}}\psi=\psi\otimes_H\phi$
that we denoted as usual
$\psi\otimes_H\phi$. Furthermore
\begin{align*}
    (\mathrm{id}_M\otimes_{\sigma^\mathcal{W}}
    ((\dst{}^N-\mathrm{d}^N)\otimes_H\mathrm{id}_{\Omega^1}))
    \circ\sigma^\mathcal{W}_{23}
    &=(\mathrm{id}_M\otimes_H
    (\dst{}^N-\mathrm{d}^N)\otimes_H\mathrm{id}_{\Omega^1})
    \circ\sigma^\mathrm{flip}_{23}\\
    &=\sigma^\mathrm{flip}_{34}\circ(\mathrm{id}_{M\otimes_H\Omega^1}\otimes_H
    (\dst{}^N-\mathrm{d}^N))~.
\end{align*}
Now consider the last term on the
right hand side of (\ref{eq63}). Since 
$\wedge\circ\sigma^{\mathrm{flip}}=-\wedge$ we obtain
\begin{align*}
    \wedge_{34}\circ(\mathrm{id}_M
    &\otimes_{\sigma^\mathcal{W}}
    (\dst{}^N-\mathrm{d}^N)\otimes_{\sigma^\mathcal{W}}\mathrm{id}_{\Omega^1})
    \circ\sigma^\mathcal{W}_{23}
    \circ((\dst{}^M-\mathrm{d}^M)\otimes_{\sigma^\mathcal{W}}\mathrm{id}_N)\\
    &=\wedge_{34}\circ(\mathrm{id}_M
    \otimes_H
    (\dst{}^N-\mathrm{d}^N)\otimes_H\mathrm{id}_{\Omega^1})
    \circ\sigma^\mathrm{flip}_{23}
    \circ((\dst{}^M-\mathrm{d}^M)\otimes_H\mathrm{id}_N)\\
    &=\wedge_{34}\circ\sigma^\mathrm{flip}_{34}
    \circ(\mathrm{id}_{M\otimes_H\Omega^1}
    \otimes_H
    (\dst{}^N-\mathrm{d}^N))
    \circ((\dst{}^M-\mathrm{d}^M)\otimes_H\mathrm{id}_N)\\
    &=-\wedge_{34}\circ((\dst{}^M-\mathrm{d}^M)\otimes_H(\dst{}^N-\mathrm{d}^N))
\end{align*}
for all $m\in M$ and $n\in N$,
i.e. the last two terms on the right hand side of (\ref{eq63}) cancel and
\begin{align*}
    \mathrm{R}^{\dst{}^M\oplus\dst{}^N}
    =\sigma^{\mathrm{flip}}_{23}\circ(\mathrm{R}^{\dst{}^M}\otimes_H\mathrm{id}_N)
    +\mathrm{id}_M\otimes_H\mathrm{R}^{\dst{}^N}
\end{align*}
remains. So the curvature of the sum of connections is the sum of the individual curvatures.
\end{remark}

\section{Riemannian Geometry on Quantum Groups}\label{Sec5}

In this section we fix a Woronowicz differential calculus 
$(\Omega^\bullet,\wedge,\mathrm{d})$
(cf. Proposition~\ref{propBEA}). We further fix
a module $M$ of ${}_H^H\mathcal{M}_H^H$ which is finitely
generated as a right $H$-module and
assume that $\overline{\sigma}^\mathcal{W}_{M,M}$
is diagonalisable with $\lambda=1$ among its eigenvalues. Following Lemma~\ref{lemma15} there is a direct sum decomposition
$M^2=M\otimes_HM=(M\vee M)\oplus(M\wedge M)$
of bicovariant bimodules, where $M\vee M:=M^{\vee2}:=V_1$ is the eigenspace
corresponding to the eigenvalue $\lambda=1$ and $M\wedge M
:=M^{\wedge 2}:=\bigoplus\nolimits_{\lambda\in\Lambda\setminus\{1\}}V_\lambda$ the direct sum
of all other eigenspaces, the braided exterior algebra in degree two. The projectors
$P_\vee\colon M^2\rightarrow M^2$
and $P_\wedge\colon M^2\rightarrow M^2$
are morphism of bicovariant bimodules.
Dually, there is the decomposition
$(M^*)^2=(M^*)^{\vee2}\oplus(M^*)^{\wedge 2}$
of bicovariant bimodules,
where $(M^*)^{\vee2}=V^*_1$, $(M^*)^{\wedge2}
=\bigoplus\nolimits_{\lambda\in\Lambda\setminus\{1\}}V^*_\lambda$ and the projectors
$P^*_\vee=(P_\vee)^*$, $P^*_\wedge=(P_\wedge)^*$, or more generally
$P^*_\lambda=(P_\lambda)^*$, are the transposed of the projectors $P_\lambda$ (see again Lemma~\ref{lemma15}). Let us denote the corresponding
bicovariant bimodule projections by
$\pi_\vee\colon M^2\rightarrow M\vee M$,
$\pi_\wedge\colon M^2\rightarrow M\wedge M$,
$\pi^*_\vee\colon(M^*)^2\rightarrow M^*\vee M^*$
and $\pi^*_\wedge\colon(M^*)^2\rightarrow M^*\wedge M^*$.

Consider now a right connection $\dst
\in\mathrm{Con}_H(M)$
and its extension $\dst\colon M^2
\rightarrow M^2\otimes_H\Omega^1$ via the sum of connections
(\ref{eq15}).
The following lemma shows that this sum of connections induces a right connection
on $M\vee M$, a right connection on $M\wedge M$ and
two right $H$-linear maps. These connections and $H$-linear maps
determine the sum of connections completely.
\begin{lemma}\label{lemma16}
A right connection $\dst\colon M^2\rightarrow M^2\otimes_H\Omega^1$
induces two right connections
\begin{equation}\label{eq85}
    \dst{}_\vee:=(\pi_\vee\otimes_H\mathrm{id}_{\Omega^1})\circ\dst|_{M^{\vee2}}
       \in\mathrm{Con}_H(M^{\vee2})
\end{equation}
and
\begin{equation}\label{eq86}
    \dst{}_\wedge:=(\pi_\wedge\otimes_H\mathrm{id}_{\Omega^1})\circ\dst|_{M^{\wedge2}}
       \in\mathrm{Con}_H(M^{\wedge2})
\end{equation}
on $M^{\vee2}$ and $M^{\wedge2}$, respectively.
Furthermore, it induces two right $H$-linear maps,
\begin{equation}\label{eq87}
    \dst{}_{12}:=(\pi_\vee\otimes_H\mathrm{id}_{\Omega^1})\circ\dst|_{M^{\wedge2}}
    \colon M^{\wedge2}\rightarrow M^{\vee2}\otimes_H\Omega^1
\end{equation}
and
\begin{equation}\label{eq88}
    \dst{}_{21}:=(\pi_\wedge\otimes_H\mathrm{id}_{\Omega^1})\circ\dst|_{M^{\vee2}}
    \colon M^{\vee2}\rightarrow M^{\wedge2}\otimes_H\Omega^1~.
\end{equation}

On the other hand, two right connections (\ref{eq85}) and (\ref{eq86}), together with
two right $H$-linear maps (\ref{eq87}) and (\ref{eq88}),
completely determine a right connection $\dst\colon M^2
\rightarrow M^2\otimes_H\Omega^1$
by
\begin{equation}\label{eq89}
    \dst
    =\dst{}_\vee\circ\pi_\vee
    +\dst{}_\wedge\circ\pi_\wedge
    +\dst{}_{12}\circ\pi_\wedge
    +\dst{}_{21}\circ\pi_\vee~,
\end{equation}
or, in matrix notation, $\dst=\begin{pmatrix}
\dst{}_\vee & \dst{}_{12} \\
\dst{}_{21} & \dst{}_\wedge
\end{pmatrix}$.

An analogous statement holds for left connections.
\end{lemma}
\begin{proof}
Let $m\in M^{\vee2}$ and $a\in H$. Since $\dst$ satisfies the Leibniz rule
and $\pi_\vee$ is right $H$-linear, 
$\dst{}_\vee(ma)
=(\pi_\vee\otimes_H\mathrm{id}_{\Omega^1})(\dst(m)a+m\otimes_H\mathrm{d}a)
=\dst{}_\vee(m)a+m\otimes_H\mathrm{d}a$, i.e. $\dst{}_\vee$ is a right connection on
$M^{\vee2}$. In complete analogy, one proves that
$\dst{}_\wedge$ is a right connection on $M^{\wedge2}$. Again for
$m\in M^{\vee2}$ and $a\in H$ we obtain
$\dst{}_{21}(ma)
=(\pi_\wedge\otimes_H\mathrm{id}_{\Omega^1})(\dst(m)a+m\otimes_H\mathrm{d}a)
=\dst{}_{21}(m)a$, i.e. $\dst{}_{21}$ is right $H$-linear, and similarly
for $\dst{}_{12}$.

On the other hand, the data (\ref{eq85})-(\ref{eq88}) determines a right connection
on $M^2$ via (\ref{eq89}), since
\begin{align*}
    \dst
    &=((\pi_\vee+\pi_\wedge)\otimes_H\mathrm{id}_{\Omega^1})
    \circ\dst\circ(\pi_\vee+\pi_\wedge)\\
    &=(\pi_\vee\otimes_H\mathrm{id}_{\Omega^1})\circ\dst\circ\pi_\vee
    +(\pi_\vee\otimes_H\mathrm{id}_{\Omega^1})\circ\dst\circ\pi_\wedge\\
    &\quad+(\pi_\wedge\circ\otimes_H\mathrm{id}_{\Omega^1})\dst\circ\pi_\vee
    +(\pi_\wedge\circ\otimes_H\mathrm{id}_{\Omega^1})\dst\circ\pi_\wedge\\
    &=\dst{}_\vee\circ\pi_\vee
    +\dst{}_{12}\circ\pi_\wedge
    +\dst{}_{21}\circ\pi_\vee
    +\dst{}_\wedge\circ\pi_\wedge~.
\end{align*}
\end{proof}

Classically, $\dst{}_{12}=0=\dst{}_{21}$ and thus $\dst\colon M^2
\rightarrow M^2\otimes_H\Omega^1$ simply restricts to connections
on $M^{\vee2}$ and $M^{\wedge2}$.
In the rest of this section we focus on $M=\Omega^1$. As before we
denote the dual bicovariant bimodule by $\mathfrak{X}:=\mathrm{Hom}_H(\Omega^1,H)$.

Let $\dst\in\mathrm{Con}_H(\Omega^1)$ be a right connection, consider the
corresponding sum 
$\dst\oplus\dst\in\mathrm{Con}_H(\Omega^1\otimes_H\Omega^1)$ 
defined in (\ref{eq15}) and denote the
right connections and right $H$-linear maps induced from $\dst\oplus\dst$ 
via Lemma~\ref{lemma16} by
\begin{equation}\label{IndCon1}
\begin{split}
    \dst{}_\vee\in\mathrm{Con}_H((\Omega^1)^{\vee2})~,&~~~~~
    \dst{}_{12}\in\mathrm{Hom}_H((\Omega^1)^{\wedge2},
    (\Omega^1)^{\vee2}\otimes_H\Omega^1)~,\\
    \dst{}_\wedge\in\mathrm{Con}_H((\Omega^1)^{\wedge2})~,&~~~~~
    \dst{}_{21}\in\mathrm{Hom}_H((\Omega^1)^{\vee2},
    (\Omega^1)^{\wedge2}\otimes_H\Omega^1)~.
\end{split}
\end{equation}
Following Proposition~\ref{prop06'}~$iv.)$ there is an induced left connection
$\std\oplus\std\in{}_H\mathrm{Con}(\mathfrak{X}^2)$ dual to $\dst\oplus\dst$.
Using the left linear analogue of Lemma~\ref{lemma16} we obtain left connections
and left $H$-linear maps 
\begin{equation}\label{IndCon2}
\begin{split}
    \std{}_\vee\in\mathrm{Con}_H(\mathfrak{X}^{\vee2})~,&~~~~~
    \std{}_{12}\in{}_H\mathrm{Hom}(\mathfrak{X}^{\wedge2},
    \Omega^1\otimes_H\mathfrak{X}^{\vee2})~,\\
    \std{}_\wedge\in\mathrm{Con}_H(\mathfrak{X}^{\wedge2})~,&~~~~~
    \std{}_{21}\in{}_H\mathrm{Hom}(\mathfrak{X}^{\vee2},
    \Omega^1\otimes_H\mathfrak{X}^{\wedge2})~.
\end{split}
\end{equation}
The following lemma proves the duality of (\ref{IndCon1}) and (\ref{IndCon2}).
\begin{lemma}
Let $\dst
\in\mathrm{Con}_H(\Omega^1)$ be a right connection
and consider the induced connections and $H$-linear maps in
(\ref{IndCon1}) and (\ref{IndCon2}).
Then $\std{}_\vee$ and $\std{}_\wedge$ are the dual connections of
$\dst{}_\vee$ and $\dst{}_\wedge$, respectively, and
\begin{equation*}
\begin{split}
    \langle X_\vee,\dst{}_{12}(\omega_\wedge)\rangle
    &=-\langle\std{}_{21}(X_\vee),\omega_\wedge\rangle~,\\
    \langle X_\wedge,\dst{}_{21}(\omega_\vee)\rangle
    &=-\langle\std{}_{12}(X_\wedge),\omega_\vee\rangle
\end{split}
\end{equation*}
for all $X_\vee\in\mathfrak{X}^{\vee2}$, $X_\wedge\in\mathfrak{X}^{\wedge2}$,
$\omega_\vee\in\Omega^{\vee2}$ and $\omega_\wedge\in\Omega^{\wedge2}$.
\end{lemma}
\begin{proof}
For $X_\vee\in\mathfrak{X}^{\vee2}$, 
$\omega_\vee\in\Omega^{\vee2}$  
we have that $\std{}_\vee$ is the dual connection to $\dst{}_\vee$,
since
\begin{align*}
    \langle X_\vee,\dst{}_\vee(\omega_\vee)\rangle
    =\langle X_\vee,\dst(\omega_\vee)\rangle
    =\mathrm{d}\langle X_\vee,\omega_\vee\rangle
    -\langle\std(X_\vee),\omega_\vee\rangle
    =\mathrm{d}\langle X_\vee,\omega_\vee\rangle
    -\langle\std{}_\vee(X_\vee),\omega_\vee\rangle
\end{align*}
and similarly one proves that $\std{}_\wedge$ and $\dst{}_\wedge$ are dual.
We further have 
\begin{align*}
    \langle\std{}_{21}(X_\vee),\omega_\wedge\rangle
    =\langle\std(X_\vee),\omega_\wedge\rangle
    =\mathrm{d}\underbrace{\langle X_\vee,\omega_\wedge\rangle}_{=0}
    -\langle X_\vee,\dst(\omega_\wedge)\rangle
    =-\langle X_\vee,\dst{}_{12}(\omega_\wedge)\rangle~,
\end{align*}
where $\omega_\wedge\in\Omega^{\wedge2}$
and similarly one proves that $\std{}_{12}$ and $\dst{}_{21}$ are dual.
\end{proof}
Extending the onion-like pairing \eqref{onion} and
  evaluating  it on
  $\mathbf{g}=\mathbf{g}^\mathrm{a}\otimes_H\mathbf{g}_\mathrm{a}\in\mathfrak{X}^2$
  (finite sum on the index $\mathrm{a}$ understood) we obtain the right $H$-linear maps
\begin{equation}\label{pairingG}
 \begin{split}
& \langle\mathbf{g},\cdot\otimes_H\cdot\rangle\colon\Omega^1\otimes_H\Omega^1\to
    H~,~~\\[.6em]
& \langle\mathbf{g},\cdot\otimes_H\cdot\otimes_H\cdot\rangle\colon\Omega^1\otimes_H\Omega^1\otimes_H\Omega^1\to \Omega^1~,~~
    \end{split}
  \begin{split}
  & \omega\otimes_H\eta
    \mapsto\langle\mathbf{g},\omega\otimes_H\eta\rangle=\langle\mathbf{g}^\mathrm{a}\langle\mathbf{g}_\mathrm{a},\omega\rangle,\eta\rangle~,~~\\[.6em]
  &  \omega\otimes_H\eta\otimes_H\rho
  \mapsto\langle\mathbf{g},\omega\otimes_H\eta\otimes\rho\rangle=\langle\mathbf{g}^\mathrm{a}\langle\mathbf{g}_\mathrm{a},\omega\rangle,\eta\rangle\rho
  ~.
\end{split}
\end{equation}
The restrictions to the bicovariant sub-bimodules
$\Omega^1\vee\Omega^1$ and  $(\Omega^1\vee\Omega^1)\otimes_H\Omega^1$
with slight abuse of notation are denoted by
$\langle\mathbf{g},\cdot\vee\cdot\rangle\colon\Omega^1\vee\Omega^1\to
H$ and  $\langle\mathbf{g},\cdot\vee\cdot\rangle\colon
(\Omega^1\vee\Omega^1)\otimes_H\Omega^1\to \Omega^1$.

\begin{definition}[Metric tensor and Levi-Civita connections]\label{def04}\phantom{a}
\begin{enumerate}
\item[i.)] An element $\mathbf{g}=\mathbf{g}^\mathrm{a}\otimes_H\mathbf{g}_\mathrm{a}
\in\mathfrak{X}^{2}$ (finite sum on the index $\mathrm{a}$ understood) is said to be a
(pseudo-Riemannian) metric if $\mathbf{g}\in\mathfrak{X}^{\vee2}$ (i.e. $\sigma^\mathcal{W}_{\mathfrak{X},\mathfrak{X}}(\mathbf{g})=\mathbf{g}$) and
$\mathbf{g}^\sharp:=\mathbf{g}^\mathrm{a}\langle\mathbf{g}_\mathrm{a},\cdot\rangle
\colon\Omega^1\rightarrow\mathfrak{X}$ is an isomorphism of right
$H$-modules.

\item[ii.)] A right connection $\dst
\in\mathrm{Con}_H(\Omega^1)$ is
said to be compatible with a pseudo-Riemannian metric
$\mathbf{g}\in\mathfrak{X}^{\vee2}$, if $\std{}_\vee(\mathbf{g})=0$, where
$\std{}_\vee\colon\mathfrak{X}^{\vee2}
\rightarrow\Omega^1\otimes_H\mathfrak{X}^{\vee2}$ is the induced
connection on $\mathfrak{X}^{\vee2}$, i.e., if 
$$
\langle\mathbf{g},\dst{}^\mathrm{LC}_\vee(\cdot\vee\cdot)\rangle=\mathrm{d}\circ\langle\mathbf{g},\cdot\vee\cdot\rangle
$$
as an equation of linear maps $\Omega\vee\Omega\to\Omega$.\\
If  in addition $\dst\in\mathrm{Con}_H(\Omega^1)$ is torsion-free we call it a
\textit{Levi-Civita connection} with respect to $\mathbf{g}$.
\end{enumerate}\end{definition}
\begin{remark}
We compare Definition~\ref{def04} to other notions of metric and
Levi-Civita connection in the literature.
\begin{enumerate}
\item[i.)] If $H$ is commutative and $\Omega^1$ a symmetric bicovariant
bimodule (in particular
$\sigma^\mathcal{W}_{\Omega^1,\Omega^1}
=\sigma^\mathrm{flip}_{\Omega^1,\Omega^1}$) 
Definition~\ref{def04} reproduces the classical notions of metric and
Levi-Civita connection. The metric axioms become 
$\sigma^\mathrm{flip}_{\mathfrak{X},\mathfrak{X}}(\mathbf{g})=\mathbf{g}$ and 
that contraction with $\mathbf{g}$ determines an isomorphism of the cotangent
and tangent bundle $\mathbf{g}^\#\colon\Omega^1\xrightarrow{\cong}\mathfrak{X}$. For the Levi-Civita
condition we first note that in the commutative case
the sum of connections is the canonical
tensor product extension of a connection $\dst$
according to Remark~\ref{rem01}.
Since $\std{}_{12}=\std{}_{21}=0$, the condition
$\std{}_\vee(\mathbf{g})=0$, $\mathbf{g}\in\mathfrak{X}^{\vee 2}$, given in Definition~\ref{def04},
is equivalent to the classical metric compatibility $\std(\mathbf{g})=0$.
In particular, the notions of pseudo-Riemannian metric and Levi-Civita
connection in differential geometry is covered.
More generally, Definition~\ref{def04} encompasses the Levi-Civita condition of braided symmetric
Riemannian geometry studied in \cite{AschieriCartanStructure,Weber}, as we will discuss in the next subsection.

\item[ii.)] The definition of metrics in Definition~\ref{def04} generalizes
the notion of metrics on Hopf algebras in \cite[Def. 1.15]{MajidBeggsBook}.
There, the isomorphism $\mathbf{g}^\#$ is assumed to be an $H$-bimodule map, 
which implies that $\mathbf{g}$ is a central element.
\end{enumerate}
\end{remark}

Given a pseudo-Riemannian metric $\mathbf{g}\in\mathfrak{X}^{\vee2}$
we define a linear map
\begin{equation}\label{eq57}\begin{split}
    \Phi_\mathbf{g}\colon\mathrm{Hom}_H(\Omega^1,\Omega^1\vee\Omega^1)
    &\rightarrow\mathrm{Hom}_H(\Omega^1\vee\Omega^1,\Omega^1)~,\\
    \phi&\mapsto\Phi_\mathbf{g}(\phi)
    :=\langle\mathbf{g},\cdot\otimes_H\cdot\rangle
    \circ\phi_\vee~,
  \end{split}
  \end{equation}
where $\phi_\vee$ is defined as the composition 
\begin{equation*}
\begin{tikzcd}
    \Omega^1\vee\Omega^1 \arrow[hookrightarrow]{r}{~} \arrow{drr}[swap]{\phi_\vee}
    & \Omega^1\otimes_H\Omega^1 \arrow{r}{\phi\oplus\phi}
    & (\Omega^1\otimes_H\Omega^1)\otimes_H\Omega^1 \arrow{d}{\pi_\vee\otimes_H\mathrm{id}}\\
    & & (\Omega^1\vee\Omega^1)\otimes_H\Omega^1
\end{tikzcd}
\end{equation*}
and $\phi\oplus\phi$ denotes
the sum of right $H$-linear maps
\begin{equation*}
    \phi\oplus\phi
    =\sigma^\mathcal{W}_{23}\circ(\phi\otimes_{\sigma^\mathcal{W}}\mathrm{id}_{\Omega^1})
    +\mathrm{id}_{\Omega^1}\otimes_{\sigma^\mathcal{W}}\phi
    \colon\Omega^1\otimes_H\Omega^1
    \rightarrow(\Omega^1\otimes_H\Omega^1)\otimes_H\Omega^1~,
\end{equation*}
cf. (\ref{eq90}).
\begin{lemma}\label{lemma17}
The map in (\ref{eq57}) is well-defined and right $H$-linear. Moreover, 
the identity
\begin{equation}\label{PhiShort}
    \Phi_\mathbf{g}(\phi)
    =2\langle\mathbf{g},\cdot\otimes_H\cdot\rangle
    \circ(\mathrm{id}_{\Omega^1}\otimes_{\sigma^\mathcal{W}}\phi)|_{\Omega^1\vee\Omega^1}
\end{equation}
holds as an equality of linear maps $\Omega^1\vee\Omega^1\to\Omega^1$. 
\end{lemma}
\begin{proof}
As a composition of right $H$-linear maps $\Phi_\mathbf{g}(\phi)$ is right $H$-linear: $\Phi_\mathbf{g}(\phi)(\omega^\mathrm{a}\otimes_H\omega_\mathrm{a}a)=\Phi_\mathbf{g}(\phi)(\omega^\mathrm{a}\otimes_H\omega_\mathrm{a})a$
for all $\omega^\mathrm{a}\otimes_H\omega_\mathrm{a}\in\Omega^1\vee\Omega^1$ (finite sum over the index $\mathrm{a}$ understood) and elements $a\in H$,
hence $\Phi$ is well-defined. Right $H$-linearity of $\Phi_\mathbf{g}$, i.e. 
$$
\Phi_\mathbf{g}(\phi\cdot a)(\omega^\mathrm{a}\otimes_H\omega_\mathrm{a})
=\Phi_\mathbf{g}(\phi)(a\cdot(\omega^\mathrm{a}\otimes_H\omega_\mathrm{a}))
$$
immediately follows from the identity (\ref{PhiShort}) that we now prove.
For all $\omega^\mathrm{a}\otimes_H\omega_\mathrm{a}\in\Omega^1\vee\Omega^1$ (finite sum over the index $\mathrm{a}$ understood), recalling that $\sigma^\mathcal{W}_{\mathfrak{X},\mathfrak{X}}(\mathbf{g})=\mathbf{g}$, we obtain
\begin{align*}
    \Phi_\mathbf{g}(\phi)(\omega^\mathrm{a}\otimes_H\omega_\mathrm{a})
    &=\langle\mathbf{g},
    \sigma^\mathcal{W}_{23}(\phi(\omega^\mathrm{a})\otimes_H\omega_\mathrm{a})
    +{}_\alpha\omega^\mathrm{a}\otimes_H({}^\alpha\phi)(\omega_\mathrm{a})\rangle\\
    &=\langle\sigma^\mathcal{W}_{\mathfrak{X},\mathfrak{X}}(\mathbf{g}),\sigma^\mathcal{W}_{23}(\phi(\omega^\mathrm{a})\otimes_H\omega_\mathrm{a})\rangle
    +\langle\mathbf{g},{}_\alpha\omega^\mathrm{a}\otimes_H({}^\alpha\phi)(\omega_\mathrm{a})\rangle\\
    &=\langle\mathbf{g},\sigma^\mathcal{W}_{12,3}(\phi(\omega^\mathrm{a})\otimes_H\omega_\mathrm{a})
    +{}_\alpha\omega^\mathrm{a}\otimes_H({}^\alpha\phi)(\omega_\mathrm{a})\rangle\\
    &=\langle\mathbf{g},{}_\alpha\omega_\mathrm{a}\otimes_H{}^\alpha(\phi(\omega^\mathrm{a}))
    +{}_\alpha\omega^\mathrm{a}\otimes_H({}^\alpha\phi)(\omega_\mathrm{a})\rangle\\
    &=\langle\mathbf{g},{}_{\alpha\beta}\omega_\mathrm{a}
    \otimes_H({}^\alpha\phi)({}^\beta\omega^\mathrm{a})
    +{}_\alpha\omega^\mathrm{a}\otimes_H({}^\alpha\phi)(\omega_\mathrm{a})\rangle\\
    &=2\langle\mathbf{g},{}_\alpha\omega^\mathrm{a}\otimes_H({}^\alpha\phi)(\omega_\mathrm{a})\rangle~,
\end{align*}
where we used Lemma~\ref{lemma03}
and ${}_\beta\omega_\mathrm{a}\otimes_H{}^\beta\omega^\mathrm{a}=\omega^\mathrm{a}\otimes_H\omega_\mathrm{a}$.
\end{proof}

\begin{remark}\begin{enumerate}
    \item[i.)]
The definition in equation \eqref{eq57} was inspired by
\cite[Def. 7.2]{BhowmickCov}. There, however, only bicoinvariant
metrics were considered, while 
\eqref{eq57} is
well-defined for arbitrary metrics.

    \item[ii.)]
Due to the
arbitrariness of the metric  $\mathbf{g}$ we can consider
conformally equivalent metrics  $\mathbf{g}'=f\mathbf{g}$ with $f\in
H$ invertible. We then observe that
$\Phi_\mathbf{g}$ is left $H$-linear in $\mathbf{g}$:
$\Phi_{f\mathbf{g}}=\ell_f\circ\Phi_{\mathbf{g}}$, where                  
$\ell_f\colon\mathrm{Hom}_H(\Omega^1\vee\Omega^1,\Omega^1)\rightarrow
\mathrm{Hom}_H(\Omega^1\vee\Omega^1,\Omega^1)$,
$\phi\mapsto(f\cdot\phi)$.

    \item[iii.)]
Using Lemma \ref{lemma15} the map $\Phi_\mathbf{g}$ in equation
\eqref{eq57} can be identified with the right $H$-linear map 
$$\Phi_\mathbf{g}:(\Omega^1\vee\Omega^1)\otimes_H\mathfrak{X}
    \rightarrow\Omega^1\otimes_H(\mathfrak{X}\vee\mathfrak{X})~.$$
Domain and codomain of this right $H$-linear map  are free finitely
generated modules, as follows recalling the discussion after
\eqref{Pnu} and Lemma \ref{lemma15} i.). Thus, upon choosing a basis,  the invertibility
of $\Phi_\mathbf{g}$ is the invertibility of a finite-dimensional matrix
with values in $H$.
\end{enumerate}
\end{remark}

It turns out that existence and uniqueness of the Levi-Civita
connection for a given arbitrary metric $\mathbf{g}$ is proven provided
$\Phi_\mathbf{g}$ is invertible. The connection is
explicitly described in terms of the structure constant
connection $\dst{}^c\in\mathrm{Con}_H(\Omega^1)$ discussed in
Proposition~\ref{prop04}
and the inverse of the $H$-linear map $\Phi_\mathbf{g}$.
\begin{theorem}\label{LCTheorem}
Consider a pseudo-Riemannian metric
$\mathbf{g}\in\mathfrak{X}^{\vee2}$. If $\Phi_\mathbf{g}$ is an
isomorphism of vector spaces then
\begin{equation}\label{eq24}
    \dst{}^\mathrm{LC}
    :=\dst{}^c
    +\Phi_\mathbf{g}^{-1}\bigg(\bigg(
    \mathrm{d}\circ\langle\mathbf{g},\cdot\otimes_H\cdot\rangle
    -\langle\mathbf{g},\dst{}^c(\cdot\otimes_H\cdot)\rangle
    \bigg)\bigg|_{\Omega^1\vee\Omega^1}\bigg)
\end{equation}
is the unique Levi-Civita connection for $\mathbf{g}$.
Surjectivity of $\Phi_\mathbf{g}$ provides existence while injectivity of 
$\Phi_\mathbf{g}$ ensures uniqueness of the Levi-Civita connection.
\end{theorem}
\begin{proof}
We prove that (\ref{eq24}) is a Levi-Civita connection for $\mathbf{g}$.
First of all, the map
$$
\psi:=
\bigg(
\mathrm{d}\circ\langle\mathbf{g},\cdot\otimes_H\cdot\rangle
-\langle\mathbf{g},\dst{}^c(\cdot\otimes_H\cdot)\rangle
\bigg)\bigg|_{\Omega^1\vee\Omega^1}
\colon\Omega^1\vee\Omega^1\rightarrow\Omega^1
$$
is right $H$-linear since
\begin{align*}
\mathrm{d}\langle\mathbf{g},\omega^i\otimes_H\omega_ia\rangle
-\langle\mathbf{g},\dst{}^c(\omega^i\otimes_H\omega_ia)\rangle
&=(\mathrm{d}\langle\mathbf{g},\omega^i\otimes_H\omega_i\rangle)a
+\langle\mathbf{g},\omega^i\otimes_H\omega_i\rangle\mathrm{d}a\\
&\quad-\langle\mathbf{g},\dst{}^c(\omega^i\otimes_H\omega_i)\rangle a
-\langle\mathbf{g},\omega^i\otimes_H\omega_i\rangle\mathrm{d}a\\
&=(\mathrm{d}\langle\mathbf{g},\omega^i\otimes_H\omega_i\rangle)a
-\langle\mathbf{g},\dst{}^c(\omega^i\otimes_H\omega_i)\rangle a
\end{align*}
for all $a\in H$ and $\omega^i\otimes_H\omega_i\in\Omega^1\vee\Omega^1$.
By assumption $\Phi_\mathbf{g}$ is invertible and 
$$
\Phi_\mathbf{g}^{-1}(\psi)
\in\mathrm{Hom}_H(\Omega^1,\Omega^1\vee\Omega^1)
\subseteq\mathrm{Hom}_H(\Omega^1,\Omega^1\otimes_H\Omega^1)~.
$$
It follows that $\dst{}^\mathrm{LC}=\dst{}^c+\Phi_\mathbf{g}^{-1}(\psi)$ is a right connection on
$\Omega^1$.

Then 
$$
\wedge\circ\dst{}^\mathrm{LC}
=\wedge\circ\dst{}^c+\wedge\circ\Phi_\mathbf{g}^{-1}(\psi)
=\wedge\circ\dst{}^c
=-\mathrm{d}
$$
since $\mathrm{im}(\Phi_\mathbf{g}^{-1}(\psi))\subseteq\Omega^1\vee\Omega^1$ and $\dst{}^c$ is 
torsion-free. This implies that $\dst{}^\mathrm{LC}$ is a torsion-free right connection.

Furthermore, $\dst{}^\mathrm{LC}$ is compatible with $\mathbf{g}$, since
\begin{align*}
    \langle\mathbf{g},
    \dst{}_\vee^\mathrm{LC}(\cdot\vee\cdot)\rangle
    &=\langle\mathbf{g},
    \dst{}_\vee^c(\cdot\vee\cdot)\rangle
    +\langle\mathbf{g},
    \Phi_\mathbf{g}^{-1}(\psi)_\vee(\cdot\vee\cdot)\rangle\\
    &=\langle\mathbf{g},
    \dst{}_\vee^c(\cdot\vee\cdot)\rangle
    +\psi(\cdot\vee\cdot)\\
    &=\langle\mathbf{g},
    \dst{}_\vee^c(\cdot\vee\cdot)\rangle
    +\mathrm{d}\circ\langle\mathbf{g},\cdot\vee\cdot\rangle
    -\langle\mathbf{g},\dst{}^c(\cdot\vee\cdot)\rangle\\
    &=\langle\mathbf{g},
    \dst{}_\vee^c(\cdot\vee\cdot)\rangle
    +\mathrm{d}\circ\langle\mathbf{g},\cdot\vee\cdot\rangle
    -\langle\mathbf{g},\dst{}_\vee^c(\cdot\vee\cdot)\rangle\\
    &=\mathrm{d}\circ\langle\mathbf{g},\cdot\vee\cdot\rangle
\end{align*}
is an equation of linear maps $\Omega^1\vee\Omega^1\rightarrow\Omega^1$.
This proves the existence of a Levi-Civita connection for $\mathbf{g}$.

Assume that there is another Levi-Civita connection 
$\dst\colon\Omega^1\rightarrow\Omega^1\otimes_H\Omega^1$ for $\mathbf{g}$. Then
$\dst{}^\mathrm{LC}-\dst\colon\Omega^1\rightarrow\Omega^1\otimes_H\Omega^1$ is a
right $H$-linear map. Since $\wedge\circ(\dst{}^\mathrm{LC}-\dst)
=\wedge\circ\dst{}^\mathrm{LC}-\wedge\circ\dst
=-\mathrm{d}+\mathrm{d}=0$ we can view $\dst{}^\mathrm{LC}-\dst$ as a right $H$-linear map
$\Omega^1\rightarrow\Omega^1\vee\Omega^1$. Consequently, we can apply
$\Phi_\mathbf{g}$ which results in the zero map:
\begin{align*}
    \Phi_\mathbf{g}(\dst{}^\mathrm{LC}-\dst)
    &=\langle\mathbf{g},
    \dst{}_\vee^\mathrm{LC}(\cdot\vee\cdot)\rangle
    -\langle\mathbf{g},
    \dst{}_\vee(\cdot\vee\cdot)\rangle\\
    &=\mathrm{d}\circ\langle\mathbf{g},\cdot\vee\cdot\rangle
    -\mathrm{d}\circ\langle\mathbf{g},\cdot\vee\cdot\rangle\\
    &=0\colon\Omega^1\vee\Omega^1\rightarrow\Omega^1~,
\end{align*}
where we used that both $\dst{}^\mathrm{LC}$ and $\dst$ are 
compatible with $\mathbf{g}$.
Since $\Phi_\mathbf{g}$ is injective this implies that
$\dst{}^\mathrm{LC}-\dst=0\colon\Omega^1\rightarrow\Omega^1\vee\Omega^1$, i.e.
$\dst{}^\mathrm{LC}=\dst\colon\Omega^1\rightarrow\Omega^1\otimes_H\Omega^1$.
This proves uniqueness of the Levi-Civita connection.
\end{proof}
We next apply this theorem to a certain class of metrics that
we call \textit{$\sigma^\mathcal{W}$-central metrics} (named
$\sigma$-metrics on self-dual bicovariant calculi in \cite{Heck});
we also clarify their relation to central metrics.
\begin{definition}\label{DefSigmaCentral}
We say that a metric $\mathbf{g}\in\mathfrak{X}^{\vee2}$ is
\begin{enumerate}
\item[i.)] \textit{central} if $a\mathbf{g}=\mathbf{g}a$ for all $a\in H$,

\item[ii.)] \textit{$\sigma^\mathcal{W}$-central} if $\sigma^\mathcal{W}_{\mathfrak{X}^{\vee 2},\Omega^1}(\mathbf{g}\otimes_H\omega)
=\omega\otimes_H\mathbf{g}$ and $\sigma^\mathcal{W}_{\Omega^1,\mathfrak{X}^{\vee 2}}(\omega\otimes_H\mathbf{g})=\mathbf{g}
\otimes_H\omega$ for all $\omega\in\Omega^1$.

\end{enumerate}
\end{definition}
\begin{lemma}\label{LemmaCentral}
Every $\sigma^\mathcal{W}$-central metric is central and every central $H$-bicoinvariant
metric is $\sigma^\mathcal{W}$-central.
\end{lemma}
\begin{proof}
For the first claim consider a basis $\{\omega^i\}_{i\in I}$ of 
${}^{\mathrm{co}H}\Omega^1$ and an arbitrary element $a\in H$. By the $H$-linearity of
the braiding and the assumption that $\mathbf{g}$ is $\sigma^\mathcal{W}$-central
we obtain
$$
\sigma^\mathcal{W}_{\Omega^1,\mathfrak{X}^{\vee2}}(\omega^i\otimes_H\mathbf{g}a)
=\sigma^\mathcal{W}_{\Omega^1,\mathfrak{X}^{\vee2}}(\omega^i\otimes_H\mathbf{g})a
=\mathbf{g}\otimes_H\omega^ia
=\sigma^\mathcal{W}_{\Omega^1,\mathfrak{X}^{\vee2}}(\omega^ia\otimes_H\mathbf{g})
=\sigma^\mathcal{W}_{\Omega^1,\mathfrak{X}^{\vee2}}(\omega^i\otimes_Ha\mathbf{g})
$$
for any $i\in I$.
Applying the inverse of $\sigma^\mathcal{W}_{\Omega^1,\mathfrak{X}^{\vee2}}$ to the
above equation gives $\omega^i\otimes_H\mathbf{g}a=\omega^i\otimes_Ha\mathbf{g}$,
which implies $\mathbf{g}a=a\mathbf{g}$ since $\{\omega^i\}_{i\in I}$ is a basis.

Now we address the second claim.
Using the explicit expression (\ref{WorBraiding}) of Woronowicz's braiding,
right $H$-covariance and centrality of $\mathbf{g}$ imply that
$\sigma^\mathcal{W}_{\Omega^1,\mathfrak{X}^{\vee2}}(\omega\otimes_H\mathbf{g})
=\omega_{-2}\mathbf{g}\otimes_HS(\omega_{-1})\omega_0
=\omega_{-2}S(\omega_{-1})\mathbf{g}\otimes_H\omega_0
=\mathbf{g}\otimes_H\omega$
for all $\omega\in\Omega^1$. Similarly, left $H$-covariance and centrality of
$\mathbf{g}$ give 
$\sigma^\mathcal{W}_{\mathfrak{X}^{\vee2},\Omega^1}(\mathbf{g}\otimes_H\omega)
=\omega_0S(\omega_1)\otimes_H\mathbf{g}\omega_2
=\omega\otimes_H\mathbf{g}$. Thus, $\sigma^\mathcal{W}$-centrality of $\mathbf{g}$
is proven.
\end{proof}

For $\sigma^\mathcal{W}$-central metrics
we formulate an existence and uniqueness theorem which is independent from invertibility of
$\Phi_\mathbf{g}$ and hence from the specific metric.
It only relies on the properties of the projection $\pi_\vee\colon\Omega^1\otimes_H\Omega^1
\rightarrow\Omega^1\vee\Omega^1$.
\begin{theorem}\label{thm03}
Any $\sigma^\mathcal{W}$-central metric has a unique Levi-Civita connection if
\begin{equation}\label{eq200}
    \pi_\vee^{23}|_{(\Omega^1\vee\Omega^1)\otimes_H\Omega^1}\colon
    (\Omega^1\vee\Omega^1)\otimes_H\Omega^1\rightarrow
    \Omega^1\otimes_H(\Omega^1\vee\Omega^1)
\end{equation}
is invertible.
\end{theorem}
\begin{proof}
Consider the map $\mathbf{g}^{\#2}\colon\Omega^1\otimes_H\Omega^1
\rightarrow(\Omega^1\otimes_H\Omega^1)^*$ defined by
\begin{equation*}
    \mathbf{g}^{\#2}(\omega_1\otimes_H\omega_2)(\omega_3\otimes_H\omega_4)
    :=\langle\mathbf{g},\omega_1\langle\mathbf{g},\omega_2\otimes_H\omega_3\rangle
    \otimes_H\omega_4\rangle~.
\end{equation*}
It is well-defined since, for all $a,b\in H$,
$
\langle\mathbf{g},\omega_1a\langle\mathbf{g},\omega_2\otimes_H\omega_3\rangle
\otimes_H\omega_4\rangle
=\langle\mathbf{g},\omega_1\langle\mathbf{g},a\omega_2\otimes_H\omega_3\rangle
\otimes_H\omega_4\rangle
$
(note that $\mathbf{g}$ is central by Lemma~\ref{LemmaCentral}) and
$
\langle\mathbf{g},\omega_1\langle\mathbf{g},\omega_2\otimes_H\omega_3a\rangle
\otimes_H\omega_4b\rangle
=\langle\mathbf{g},\omega_1\langle\mathbf{g},\omega_2\otimes_H\omega_3\rangle
\otimes_Ha\omega_4\rangle b.
$
It is right $H$-linear since
$
\langle\mathbf{g},\omega_1\langle\mathbf{g},\omega_2a\otimes_H\omega_3\rangle
\otimes_H\omega_4\rangle
=\langle\mathbf{g},\omega_1\langle\mathbf{g},\omega_2\otimes_Ha\omega_3\rangle
\otimes_H\omega_4\rangle.
$
We prove that $\mathbf{g}^{\#2}(\sigma^\mathcal{W}_{\Omega^1,\Omega^1}(\omega_1\otimes_H\omega_2))
(\omega_3\otimes_H\omega_4)=\mathbf{g}^{\#2}(\omega_1\otimes_H\omega_2)
(\sigma^\mathcal{W}_{\Omega^1,\Omega^1}(\omega_3\otimes_H\omega_4))$. 
In fact,
\begin{align*}
    \mathbf{g}^{\#2}(\sigma^\mathcal{W}_{\Omega^1,\Omega^1}(\omega_1\otimes_H\omega_2))(\omega_3\otimes_H\omega_4)
    &=\mathbf{g}^{\#2}({}_\alpha\omega_2\otimes_H{}^\alpha\omega_1)
    (\omega_3\otimes_H\omega_4)\\
    &=\langle\mathbf{g},{}_\alpha\omega_2\langle
    \mathbf{g},{}^\alpha\omega_1\otimes_H\omega_3\rangle
    \otimes_H\omega_4\rangle\\
    &=\langle\mathbf{g},{}_\alpha\omega_2
    \otimes_H{}_{\beta\gamma}\omega_4\rangle
    \langle\mathbf{g},{}^{\beta\alpha}\omega_1\otimes_H{}^\gamma\omega_3\rangle\\
    &=\langle\mathbf{g},{}^{\overline{\delta}\overline{\rho}\beta\alpha}\omega_1
    \langle\mathbf{g},{}_{\overline{\delta}\alpha}\omega_2
    \otimes_H{}_{\overline{\rho}\beta\gamma}\omega_4\rangle
    \otimes_H{}^\gamma\omega_3\rangle\\
    &=\langle\mathbf{g},\omega_1
    \langle\mathbf{g},\omega_2
    \otimes_H{}_\gamma\omega_4\rangle
    \otimes_H{}^\gamma\omega_3\rangle\\
    &=\mathbf{g}^{\#2}(\omega_1\otimes_H\omega_2)(\sigma^\mathcal{W}_{\Omega^1,\Omega^1}(\omega_3\otimes_H\omega_4))~,
\end{align*}
for all $\omega_1,\omega_2,\omega_3,\omega_4\in\Omega^1$, where we first moved 
inner pairing outside via $\sigma^\mathcal{W}$ and then moved the pairing on the left inside the pairing on the right
via $\overline{\sigma}^\mathcal{W}$. From Lemma~\ref{lemma15} it then follows that $\mathbf{g}^{\#2}(\pi_\vee(\omega_1\otimes_H\omega_2))
(\omega_3\otimes_H\omega_4)=\mathbf{g}^{\#2}(\omega_1\otimes_H\omega_2)
(\pi_\vee(\omega_3\otimes_H\omega_4))$ and consequently that 
$\mathbf{g}^{\#2}$ restricts to a right $H$-linear map
$\mathbf{g}^{\#2}|_{\Omega^1\vee\Omega^1}\colon\Omega^1\vee\Omega^1
\rightarrow(\Omega^1\vee\Omega^1)^*$.

Under the identification
$(\Omega^1\vee\Omega^1)^*\cong(\Omega^1)^*\vee(\Omega^1)^*$ we have
$\mathbf{g}^{\#2}=\mathbf{g}^\#\otimes_{\sigma^\mathcal{W}}\mathbf{g}^\#$,
clearly an invertible map.

We next show that the diagram
\begin{equation}\label{eq91}
\begin{tikzcd}
\mathrm{Hom}_H(\Omega^1,\Omega^1\vee\Omega^1)
\cong(\Omega^1\vee\Omega^1)\otimes_H\mathfrak{X}
\arrow{d}[swap]{\frac{1}{2}\Phi_\mathbf{g}}
\arrow{rrr}{\mathrm{id}_{\Omega^1\vee\Omega^1}
\otimes_{\sigma^\mathcal{W}}\mathbf{g}^{\#-1}}
& & & (\Omega^1\vee\Omega^1)\otimes_H\Omega^1
\arrow{d}{\pi_\vee^{23}}\\
\mathrm{Hom}_H(\Omega^1\vee\Omega^1,\Omega^1)
\cong\Omega^1\otimes_H(\Omega^1\vee\Omega^1)^*
& & & \Omega^1\otimes_H(\Omega^1\vee\Omega^1)
\arrow{lll}[swap]{\mathrm{id}_{\Omega^1}
\otimes_{\sigma^\mathcal{W}}\mathbf{g}^{\#2}}
\end{tikzcd}
\end{equation}
commutes. Let $\phi\in\mathrm{Hom}_H(\Omega^1,\Omega^1\vee\Omega^1)$ be arbitrary.
We identify $\phi$ with the element $\omega^i\otimes_H\omega_i^j\otimes_H\chi_j
\in(\Omega^1\vee\Omega^1)\otimes_H\mathfrak{X}$, where
$\omega^i,\omega_i^j\in\Omega^1$ and $\chi_j\in\mathfrak{X}$ (sum over repeated indices understood).
Then,
\begin{align*}
    [(\mathrm{id}_{\Omega^1}\otimes_{\sigma^\mathcal{W}}\mathbf{g}^{\#2})
    \,\circ\,&\pi_\vee^{23}
    \circ(\mathrm{id}_{\Omega^1\vee\Omega^1}
    \otimes_{\sigma^\mathcal{W}}\mathbf{g}^{\#-1})
    ](\phi)\\
    &=[(\mathrm{id}_{\Omega^1}\otimes_{\sigma^\mathcal{W}}\mathbf{g}^{\#2})
    \circ\pi_\vee^{23}]
    (\omega^i\otimes_H\omega_i^j\otimes_H\mathbf{g}^{\#-1}(\chi_j))\\
    &=(\mathrm{id}_{\Omega^1}\otimes_{\sigma^\mathcal{W}}\mathbf{g}^{\#2})
    (\omega^i\otimes_H(\omega_i^j\vee\mathbf{g}^{\#-1}(\chi_j)))\\
    &=\omega^i\otimes_H\mathbf{g}^{\#2}(\omega_i^j\vee\mathbf{g}^{\#-1}(\chi_j))~,
\end{align*}
which on elements $\eta^k\otimes_H\eta_k\in\Omega^1\vee\Omega^1$ (sum over repeated indices understood) reads
\begin{align*}
    [\omega^i\otimes_H\mathbf{g}^{\#2}(\omega_i^j\vee\mathbf{g}^{\#-1}(\chi_j))]
    (\eta^k\otimes_H\eta_k)
    &=\omega^i\mathbf{g}^{\#2}(\omega_i^j\vee\mathbf{g}^{\#-1}(\chi_j))
    (\eta^k\otimes_H\eta_k)\\
    &=\omega^i\langle\mathbf{g},\omega^j_i
    \langle\mathbf{g},\mathbf{g}^{\#-1}(\chi_j)\otimes_H\eta^k\rangle
    \otimes_H\eta_k\rangle\\
    &=\omega^i\langle\mathbf{g},\omega^j_i
    \chi_j(\eta^k)\otimes_H\eta_k\rangle~.
\end{align*}
On the other hand, by Lemma~\ref{lemma17},
\begin{align*}
    \frac{1}{2}\Phi_\mathbf{g}(\phi)(\eta^k\otimes_H\eta_k)
    &=\langle\mathbf{g},{}_\alpha\eta^k\otimes_H({}^\alpha\phi)(\eta_k)\rangle\\
    &=\langle\mathbf{g},{}_{\alpha\beta\gamma}\eta^k
    \otimes_H{}^\alpha\omega^i\rangle~
    {}^\beta\omega^j_i\,({}^\gamma\chi_j)(\eta_k)\\
    &=\langle\mathbf{g},{}_{\alpha\beta\gamma}\eta^k
    \otimes_H{}^\alpha{}_\delta\omega^j_i\rangle~
    {}^{\beta\delta}\omega^i\,({}^\gamma\chi_j)(\eta_k)\\
    &=\langle\mathbf{g},{}_{\beta\alpha\gamma}\eta^k
    \otimes_H{}_\delta{}^\alpha\omega^j_i\rangle~
    {}^{\delta\beta}\omega^i\,({}^\gamma\chi_j)(\eta_k)\\
    &=\omega^i\langle\mathbf{g},{}_{\alpha\gamma}\eta^k
    \otimes_H{}^\alpha\omega^j_i\rangle
    ({}^\gamma\chi_j)(\eta_k)\\
    &=\omega^i\langle\mathbf{g},\omega^j_i\otimes_H{}_\gamma\eta^k\rangle
    ({}^\gamma\chi_j)(\eta_k)\\
    &=\omega^i\langle\mathbf{g},
    \omega^j_i\,({}^{\overline{\alpha}\gamma}\chi_j)({}^{\overline{\beta}}\eta_k)
    \otimes_H{}_{\overline{\beta}\overline{\alpha}\gamma}\eta^k\rangle\\
    &=\omega^i\langle\mathbf{g},
    \omega^j_i\chi_j\,({}^{\overline{\beta}}\eta_k)
    \otimes_H{}_{\overline{\beta}}\eta^k\rangle\\
    &=\omega^i\langle\mathbf{g},\omega^j_i\,
    \chi_j(\eta^k)\otimes_H\eta_k\rangle~,
\end{align*}
where we used that $\omega^i\otimes_H\omega_i^j={}_\delta\omega_i^j\otimes_H{}^\delta\omega^i$,
the quantum Yang-Baxter equation,
${}_\delta\omega^j_i\otimes_H{}^\delta\omega^i=\omega^i\otimes_H\omega^j_i$, that $\mathbf{g}$ is $\sigma^\mathcal{W}$-central, then that $\mathbf{g}$ is symmetric and that
 ${}^{\overline{\beta}}\eta_k\otimes_H{}_{\overline{\beta}}\eta^k
=\eta^k\otimes_H\eta_k$.
So the commutativity of (\ref{eq91}) follows.

Since the horizontal arrows in (\ref{eq91}) are isomorphisms and, by hypothesis, (\ref{eq200}) is an isomorphism, so is $\Phi_\mathbf{g}$. Then the claim of the theorem follows from Theorem~\ref{LCTheorem}.
\end{proof}

\begin{remark}\label{remBM}
Theorem \ref{thm03} has been inspired by \cite[Thm.7.9]{BhowmickCov}
which claims an analogous result for bicoinvariant and noncentral
metrics. However we have not been able to reproduce their claim since
the last step in the proof of their key Lemma 4.8 to our knowledge
does not hold in general for bicoinvariant and noncentral metrics.   
\end{remark}
We now relax the $\sigma^\mathcal{W}$-centrality condition of the metric. If $\mathbf{g}\in\mathfrak{X}^{\vee 2}$ so does $f\mathbf{g}$ with $f\in H$. If $\mathbf{g}$ is a metric and $f\in H$ is an invertible element then $f\mathbf{g}$ is a metric, by the inverse of $(f\mathbf{g})^\#$ being $(f\mathbf{g})^{\#-1}=\mathbf{g}^{\#-1}\circ\ell_{f^{-1}}$, where $\ell_{f^{-1}}\colon\Omega^1\to\Omega^1$, $\omega\mapsto f^{-1}\omega$, is the left multiplication with $f^{-1}$. We call $\mathbf{g}$ and $f\mathbf{g}$ conformally equivalent metrics. 
The Levi-Civita connections of conformally equivalent
metrics are easily related.
\begin{theorem}\label{thm04}
Let $\mathbf{g}$ be an arbitrary metric and $f\in H$ an invertible element.
If $\Phi_\mathbf{g}$ is invertible with corresponding Levi-Civita connection
$\dst{}^\mathrm{LC}$ then
$$
\dst{}^{'\mathrm{LC}}
=\dst{}^\mathrm{LC}+\Phi^{-1}_\mathbf{g}\bigg(f^{-1}\mathrm{d}f~
\langle\mathbf{g},\cdot\otimes_H\cdot\rangle\bigg|_{\Omega^1\vee\Omega^1}\bigg)
$$
is the unique Levi-Civita connection of $\mathbf{g}':=f\mathbf{g}$.
\end{theorem}
\begin{proof}
Assume that $\Phi_{\mathbf{g}}\colon
\mathrm{Hom}_H(\Omega^1,\Omega^1\vee\Omega^1)
\rightarrow\mathrm{Hom}_H(\Omega^1\vee\Omega^1,\Omega^1)$ is an isomorphism.
From $\Phi_{f\mathbf{g}}=\ell_f\circ\Phi_{\mathbf{g}}$, where
$\ell_f\colon\mathrm{Hom}_H(\Omega^1\vee\Omega^1,\Omega^1)\rightarrow
\mathrm{Hom}_H(\Omega^1\vee\Omega^1,\Omega^1)$,
$\phi\mapsto(f\cdot\phi)$, and invertibility of $\Phi_\mathbf{g}$,
we have that
$\Phi_{f\mathbf{g}}$ is invertible with inverse 
$\Phi_{f\mathbf{g}}^{-1}=\Phi_{\mathbf{g}}^{-1}\circ\ell_{f^{-1}}$.
By Theorem~\ref{LCTheorem} the unique Levi-Civita connection of $\mathbf{g}'$
is then
\begin{align*}
    \dst{}^{'\mathrm{LC}}
    &:=\dst{}^c
    +\Phi_{\mathbf{g}'}^{-1}\bigg(\bigg(
    \mathrm{d}\circ\langle\mathbf{g}',\cdot\otimes_H\cdot\rangle
    -\langle\mathbf{g}',\dst{}^c(\cdot\otimes_H\cdot)\rangle
    \bigg)\bigg|_{\Omega^1\vee\Omega^1}\bigg)\\
    &=\dst{}^c
    +(\Phi_\mathbf{g}^{-1}\circ\ell_{f^{-1}})\bigg(\bigg(
    \mathrm{d}\circ\langle f\mathbf{g},\cdot\otimes_H\cdot\rangle
    -\langle f\mathbf{g},\dst{}^c(\cdot\otimes_H\cdot)\rangle
    \bigg)\bigg|_{\Omega^1\vee\Omega^1}\bigg)\\
    &=\dst{}^c
    +\Phi_\mathbf{g}^{-1}\bigg(\bigg(
    \mathrm{d}\circ\langle\mathbf{g},\cdot\otimes_H\cdot\rangle
    -\langle\mathbf{g},\dst{}^c(\cdot\otimes_H\cdot)\rangle
    \bigg)\bigg|_{\Omega^1\vee\Omega^1}\bigg)
    +\Phi^{-1}_\mathbf{g}\bigg(f^{-1}\mathrm{d}f~
    \langle\mathbf{g},\cdot\otimes_H\cdot\rangle\bigg|_{\Omega^1\vee\Omega^1}\bigg)\\
    &=\dst{}^\mathrm{LC}
    +\Phi^{-1}_\mathbf{g}\bigg(f^{-1}\mathrm{d}f~
    \langle\mathbf{g},\cdot\otimes_H\cdot\rangle\bigg|_{\Omega^1\vee\Omega^1}\bigg)~,
\end{align*}
where we used the explicit expression (\ref{eq24}) of the Levi-Civita connection
$\dst{}^\mathrm{LC}$.
\end{proof}
Combining the last two theorems we immediately obtain 
\begin{theorem}\label{corlast}
Any metric $\mathbf{g}$ conformally equivalent to a 
$\sigma^\mathcal{W}$-central metric has a unique Levi-Civita connection if $\pi_\vee\colon\Omega^1\otimes_H\Omega^1\rightarrow\Omega^1\vee\Omega^1$ is invertible.
\end{theorem}
We conclude by discussing some explicit examples of Levi-Civita
connections.
\begin{example}\label{Ex:sigmasquare}
Let $(\Omega^1,\mathrm{d})$ be a bicovariant FODC such that
$(\sigma^\mathcal{W}_{\Omega^1,\Omega^1})^2=\mathrm{id}$, implying
that $\sigma^\mathcal{W}_{\Omega^1,\Omega^1}$ is diagonalisable with eigenvalues $\pm 1$, hence $\pi_\vee=\frac{1}{2}(\mathrm{id}+\sigma^\mathcal{W}_{\Omega^1,\Omega^1})$ and $\pi_\wedge=\frac{1}{2}(\mathrm{id}-\sigma^\mathcal{W}_{\Omega^1,\Omega^1})$. Similarly to \cite[Prop. 7.11]{BhowmickCov} we show that
$\pi_\vee^{23}|_{(\Omega^1\vee\Omega^1)\otimes_H\Omega^1}\colon
(\Omega^1\vee\Omega^1)\otimes_H\Omega^1\rightarrow
\Omega^1\otimes_H(\Omega^1\vee\Omega^1)$ is an isomorphism in this case.
Let $\vartheta\in(\Omega^1\vee\Omega^1)\otimes_H\Omega^1$ be such that
$\pi_\vee^{23}|_{(\Omega^1\vee\Omega^1)\otimes_H\Omega^1}(\vartheta)=0$.
Then $\sigma^\mathcal{W}_{12}(\vartheta)=\vartheta$ and
$\sigma^\mathcal{W}_{23}(\vartheta)=-\vartheta$. By the quantum Yang-Baxter 
equation 
$$
\vartheta
=(\sigma^\mathcal{W}_{23}\circ\sigma^\mathcal{W}_{12}
\circ\sigma^\mathcal{W}_{23})(\vartheta)
=(\sigma^\mathcal{W}_{12}\circ\sigma^\mathcal{W}_{23}
\circ\sigma^\mathcal{W}_{12})(\vartheta)
=-\vartheta~,
$$
hence $\vartheta=0$ and $\pi_\vee^{23}|_{(\Omega^1\vee\Omega^1)\otimes_H\Omega^1}$
is injective. Invertibility of the $H$-module map 
$\pi_\vee^{23}|_{(\Omega^1\vee\Omega^1)\otimes_H\Omega^1}$ follows observing that the $H$-modules $(\Omega^1\vee\Omega^1)\otimes_H\Omega^1$ and $\Omega^1\otimes_H(\Omega^1\vee\Omega^1)$ have the same dimension.
Then, by Theorem~\ref{thm03} every $\sigma^\mathcal{W}$-central metric
$\mathbf{g}$ on $\Omega^1$ admits a unique Levi-Civita connection.
In particular, the classical Levi-Civita theorem is recovered observing that
if $\sigma^\mathcal{W}$ is given by the flip every metric is trivially
$\sigma^\mathcal{W}$-central.
\end{example}
\begin{example}[Levi-Civita connections for conformally equivalent metrics on $SL_q(2)$]
Consider the quantum group $SL_q(2)$ and let  $(\Omega^1,\mathrm{d})$
be one of the $4D_{\pm}$ bicovariant FODCi.
According to \cite{Mukhopadhyay} the corresponding projection
$\pi^{23}_\vee|_{(\Omega^1\vee\Omega^1)\otimes_H\Omega^1}$ is invertible.
Then Theorem~\ref{thm03} implies that every $\sigma^\mathcal{W}$-central
metric on $\Omega^1$ admits a unique Levi-Civita connection.
Recalling Lemma~\ref{LemmaCentral} this applies
in particular to the central $H$-bicoinvariant metrics $\mathbf{g}$ on
the $4D_{\pm}$ calculi of $SL_q(2)$ constructed in \cite[Lem. 6.1 ]{Heck} and \cite[Prop. 2.60]{MajidBeggsBook}. 
More in general, from Theorem \ref{corlast}, there is a unique Levi-Civita connection corresponding to
$f\mathbf{g}$ for every  element $f$ that is invertible in a suitable completion of $SL_q(2)$.
For example, considering the compact form $SU_q(2)$ it is natural to consider $f^{-1}$ is in its $C^*$-algebra completion.
\end{example}

\subsection*{The cotriangular case}

There is a canonical Woronowicz calculus on a cotriangular Hopf algebra $H$, and Woronowicz's braiding on forms is involutive. Example~\ref{Ex:sigmasquare} then implies that there is a 
unique Levi-Civita connection for any metric which is conformally
equivalent to a bicoinvariant (and thus $\sigma^\mathcal{W}$-central) metric. The obtained Levi-Civita
connections coincide with the ones constructed via Koszul formula in the braided-derivation approach of \cite{AschieriCartanStructure,Weber}.

Recall that a cotriangular Hopf algebra $(H,\mathscr{R})$ is a Hopf
algebra $H$ with a linear map $\mathscr{R}\colon H\otimes H\to\Bbbk$
that is convolution invertible with inverse $\mathscr{R}_{21}$ and such that, for all $a,b,c\in H$,
$\mathscr{R}(a_1\otimes b_1)a_2b_2=b_1a_1\mathscr{R}(a_2\otimes b_2)$ and
\begin{equation*}
\begin{split}
    \mathscr{R}(ab\otimes c)=\mathscr{R}(a\otimes c_1)\mathscr{R}(b\otimes c_2)~,\qquad\qquad
    \mathscr{R}(a\otimes bc)=\mathscr{R}(a_1\otimes c)\mathscr{R}(a_2\otimes b)~.
\end{split}
\end{equation*}
Recall that the category of comodules $(\mathcal{M}^H,\otimes)$ of a cotriangular Hopf algebra is symmetric monoidal with symmetric braiding given by $\sigma^\mathscr{R}_{M,N}\colon M\otimes N\to N\otimes M$, $\sigma^\mathscr{R}_{M,N}(m\otimes n):=n_0\otimes m_0\mathscr{R}(m_1\otimes n_1)$, where $M,N$ are right $H$-comodules, see e.g. \cite{Ka95,Ma95}.

Given a cotriangular Hopf algebra $(H,\mathscr{R})$, the co-opposite Hopf algebra $H^\mathrm{cop}$ (i.e. $H$ endowed with the opposite coproduct and inverse antipode) is cotriangular with respect to $\mathscr{R}^{-1}$ and the tensor product $H^\mathrm{cop}\otimes H$ (with its obvious Hopf algebra structure) is cotriangular with respect to
$$
\mathcal{R}:=(\mathscr{R}^{-1}\otimes\mathscr{R})\circ(\mathrm{id}\otimes\tau_{H,H^\mathrm{cop}}\otimes\mathrm{id})\colon(H^\mathrm{cop}\otimes H)\otimes(H^\mathrm{cop}\otimes H)\to\Bbbk~.
$$
The reader easily verifies that the category of bicovariant bimodules
${}_H^H\mathcal{M}_H^H$ is equivalent to the category
${}_H\mathcal{M}_H^{H^\mathrm{cop}\otimes H}$ of right
$(H^\mathrm{cop}\otimes H)$-covariant $H$-bimodules, where the
 $(H^\mathrm{cop}\otimes H)$-coaction  on a bicovariant bimodule $M$ in
terms of the left and right $H$-coactions reads
\begin{align*} 
  M&\to M\otimes(H^\mathrm{cop}\otimes H)~,\\
  m&\mapsto m_{[0]}\otimes m_{[1]}:=m_0\otimes m_{-1}\otimes m_1~.
\end{align*}
A bicovariant bimodule $M$ is called $\mathrm{R}$-symmetric if $m\cdot
a=a_{[0]}\cdot m_{[0]}\mathcal{R}(m_{[1]}\otimes a_{[1]})$ for all
$m\in M$ and $a\in H$.
We denote the full subcategory of bicovariant
bimodules which are $\mathrm{R}$-symmetric by
${}_H^H\mathcal{M}_H^{H,\mathrm{sym}}$.
It is a symmetric (also termed braided symmetric)  monoidal category
$({}_H^H\mathcal{M}_H^{H,\mathrm{sym}},\otimes_H, \sigma^\mathcal{R})$  with
the balanced tensor product over $H$ and  braiding given by
$\sigma^\mathcal{R}_{M,N}\colon M\otimes_HN\to N\otimes_HM$, $\sigma^\mathcal{R}_{M,N}(m\otimes_Hn)=n_{[0]}\otimes_Hm_{[0]}\mathcal{R}(m_{[1]}\otimes n_{[1]})$. 
This braiding equals Woronowicz's braiding $\sigma^\mathcal{W}$ on
${}_H^H\mathcal{M}_H^{H,\mathrm{sym}}$: they are both
$H$-bilinear and  equal the flip on  the vector subspace ${}^{\mathrm{co}H\!}M\otimes N^{\mathrm{co}H}$.
As in \cite[Remark 3.11 and Example 3.9]{AschieriCartanStructure} one
constructs the $\mathrm{R}$-symmetric bicovariant bimodule of
$\mathrm{R}$-braided derivations
$$
\mathfrak{X}_\mathcal{R}:=\{X\in\mathrm{END}_\Bbbk(H)~|~X(ab)=X(a)b+a_{[0]}X_{[0]}(b)~\mathcal{R}(X_{[1]}\otimes a_{[1]})\text{ for all }a,b\in H\},
$$
where $\mathrm{END}_\Bbbk(H)$ denotes the vector space of
endomorphisms of $H$ which are (left and right) rational. We have the
canonical decomposition
$\mathfrak{X}_\mathcal{R}= {}^{\mathrm{co}H\!}\mathfrak{X}_\mathcal{R}\otimes H$ and
assume the linear space of left-invariant vector fields
${}^{\mathrm{co}H\!}\mathfrak{X}_\mathcal{R}$ to be finite-dimensional. 
The dual
space of left $H$-linear maps $\mathfrak{X}_\mathcal{R}\to H$ is the $\mathrm{R}$-symmetric 
bicovariant bimodule of one forms
$\Gamma:={}_H\mathrm{HOM}(\mathfrak{X},H)$. 
The differential is determined by the dual pairing
$\langle\cdot,\cdot\rangle\colon\mathfrak{X}_\mathcal{R}\otimes_H\Gamma\to
H$ via $\langle X,\mathrm{d}a\rangle=X(a)$ for all $a\in H$. From the equality
$\sigma^\mathcal{W}_{\Gamma,\Gamma}=\sigma^\mathcal{R}_{\Gamma,\Gamma}$
we see that Woronowicz's braiding in this cotriangular case is involutive, hence its
eigenvalues are $\pm 1$ and the projections
$\Gamma\otimes_H\Gamma\to\Gamma\wedge\Gamma$,
$\Gamma\otimes_H\Gamma\to\Gamma\vee\Gamma$ are the
(skew-)symmetrizers: $\frac{1}{2}(\mathrm{id}^{}\mp \sigma^\mathcal{R}_{\Gamma,\Gamma})$.

In this setting, given a bicoinvariant metric
$\mathbf{g}\in\mathfrak{X}_\mathcal{R}\vee\mathfrak{X}_\mathcal{R}$ we
have that $\mathbf{g}\cdot
a=a_{[0]}\cdot\mathbf{g}_{[0]}\mathcal{R}(\mathbf{g}_{[1]}\otimes
a_{[1]})=a\cdot\mathbf{g}$ for all $a\in H$, i.e., that $\mathbf{g}$
is central. Lemma~\ref{LemmaCentral} then implies that $\mathbf{g}$ is
$\sigma^\mathcal{W}$-central.
Thus, from Example \eqref{Ex:sigmasquare} in this cotriangular setting every metric which is conformally equivalent to a bicoinvariant one admits a unique Levi-Civita connection.

\medskip

In the remaining part of this section we show that the above
Levi-Civita connections coincide with the ones obtained in the
braided-derivation approach of \cite{Weber} for equivariant
metrics and \cite{AschieriCartanStructure} for arbitrary metrics. As a
first step we prove that in the current cotriangular setting the sum
of connections $\dst\oplus\dst$ introduced in Section~\ref{Sec4.4}
coincides with the one $\dst\oplus_\mathcal{R}\dst$ in
\cite[Thm. 4.8]{AschieriCartanStructure}. It is sufficient to prove
this for elements $\omega\otimes_H\eta$ with
$\omega\in\Gamma^{\mathrm{co}H}$ and $\eta\in\Gamma$ since
$\dst\oplus\dst$ and $\dst\oplus_\mathcal{R}\dst$ are both
well-defined on the balanced tensor product $\Gamma\otimes_H \Gamma$. Now,
\begin{align*}
    (\dst\oplus\dst)(\omega\otimes_H\eta)
    &=\sigma^\mathcal{W}_{23}(\dst(\omega)\otimes_H\eta)
    +\dst{}_{-2}\omega\otimes_HS(\dst{}_{-1})\dst{}_0(\eta)\\
    &=\sigma^\mathcal{R}_{23}(\dst(\omega)\otimes_H\eta)
    +{}_\alpha\omega\otimes_H({}^\alpha\dst)(\eta)\\
    &=(\dst\oplus_\mathcal{R}\dst)(\omega\otimes_H\eta)~,
\end{align*}
where we have used that $\sigma^\mathcal{W}$ and $\sigma^\mathcal{R}$
coincide in ${}_H^H\mathcal{M}_H^{H, \mathrm{sym}}$.
Then, we observe that the braiding commutes with the sum of connections, i.e., $\sigma^\mathcal{W}_{12}\circ(\dst\oplus\dst)=(\dst\oplus\dst)\circ\sigma^\mathcal{W}_{\Gamma,\Gamma}$.
This can be shown equivalently, and more easily, for
$\dst\oplus_\mathcal{R}\dst$ and $\sigma^\mathcal{R}$, using that for
$M,N$ in ${}_H^H\mathcal{M}_H^{H, \mathrm{sym}}$ and any rational
morphism $\phi\in\mathrm{HOM}_\Bbbk(M,N)$ we have ${}_\alpha
n\otimes({}^\alpha\phi)(m)=\sigma^\mathcal{R}_{N,N}((\phi\otimes\mathrm{id}_N)\overline{\sigma}_{M,N}^\mathcal{R}(n\otimes
m))$ for all $m\in M$, $n\in N$, where
$(\phi\otimes\mathrm{id}_N)(m\otimes n)=\phi(m)\otimes n$. Furthermore,
\begin{align*}
    \sigma^\mathcal{R}_{N,N}\circ(\phi\otimes\mathrm{id}_N+\sigma^\mathcal{R}_{N,N}\circ(\phi\otimes\mathrm{id}_N)\circ\overline{\sigma}^\mathcal{R}_{M,N})
    &=\sigma^\mathcal{R}_{N,N}\circ(\phi\otimes\mathrm{id}_N+\sigma^\mathcal{R}_{N,N}\circ(\phi\otimes\mathrm{id}_N)\circ\overline{\sigma}^\mathcal{R}_{M,N})\circ\sigma^\mathcal{R}_{M,N}\circ\sigma^\mathcal{R}_{N,M}\\
    &=(\phi\otimes\mathrm{id}_N+\sigma^\mathcal{R}_{N,N}\circ(\phi\otimes\mathrm{id}_N)\circ\overline{\sigma}^\mathcal{R}_{M,N})\circ\sigma^\mathcal{R}_{N,M}
\end{align*}
by the involutivity of $\sigma^\mathcal{R}$.
Hence the projections too commute
with the sum of connections so that the additional projection in
(\ref{eq85}) becomes superfluous. The notion of torsion-freeness
coincides in both approaches since it merely relies on the
differential and wedge product.
Thus, the Levi-Civita connections obtained in Theorem
\eqref{corlast} for metrics conformally equivalent to bicoinvariant
ones in this cotriangular case are among those obtained in \cite[Thm. 6.7]{AschieriCartanStructure}.

\section*{Acknowledgements}

We would like to thank Alessandro Ardizzoni, Jyotishman Bhowmick and Rita Fioresi for fruitful discussions. The work of P.A. and T.W. is partially
supported by INFN, CSN4, Iniziativa Specifica GSS, and by Universit\`a del
Piemonte Orientale. P.A. is also
affiliated to INdAM, GNFM (Istituto Nazionale di Alta Matematica,
Gruppo Nazionale di Fisica Matematica). This article is based upon work from COST Action CaLISTA CA21109 supported by COST (European Cooperation in Science and Technology).

\end{document}